\documentclass[11pt]{amsart}
\usepackage[official]{eurosym}
\usepackage{babel}
\usepackage{bbm}
\usepackage[utf8]{inputenc}
\usepackage[T1]{fontenc}
\usepackage{pdfpages}
\usepackage{tensor}
\usepackage{enumerate}
\usepackage[stable]{footmisc}
\usepackage{mathtools}
\usepackage[mathscr]{eucal}
\usepackage[a4paper, twoside=true, hmargin={2cm,3cm}, vmargin={1.5 cm,1.5cm}, includehead]{geometry}
\usepackage{pgfgantt}
\usepackage{graphicx}
\usepackage{xcolor}
\ganttset{group/.append style={orange},
milestone/.append style={red},
progress label node anchor/.append style={text=red}}
\theoremstyle{plain}
\newtheorem{theorem}{Theorem}[section]
\newtheorem{lemma}[theorem]{Lemma}
\newtheorem{proposition}[theorem]{Proposition}

\newtheorem{conjecture}{Conjecture}
\newtheorem{corollary}[theorem]{Corollary}

\newtheorem{definition}[theorem]{Definition}
\newtheorem{claim}{Claim}
\newtheorem*{Graal1*}{Special Case of Problem~\ref{Graal1}}
\newtheorem*{Graal2*}{Special Case of Problem~\ref{Graal2}}

\newtheorem*{theorem*}{Theorem}

\newtheorem{innercustomthm}{Theorem}
\newenvironment{customthm}[1]
  {\renewcommand\theinnercustomthm{#1}\innercustomthm}
  {\endinnercustomthm}

\theoremstyle{remark}
\newtheorem*{remark*}{{\bf Remark}}
\newtheorem*{remarks*}{{\bf Remarks}}
\newtheorem*{comment*}{{\bf Comment}}

\usepackage{amsmath}
\newcommand{\ue}{\underline{\varepsilon}}

\newcommand{\nnorm}[1]{\lvert\!|\!| #1|\!|\!\rvert}
\newcommand{\bignorm}[1]{\big\lVert #1 \big\rVert}
\newcommand{\norm}[1]{\lVert #1 \rVert}
\newcommand{\N}{\mathbb{N}}

\newcommand{\R}{\mathbb{R}}

\newcommand{\Z}{\mathbb{Z}}
\newcommand{\A}{\mathcal{A}}

\newcommand{\E}{\mathbb{E}}
\newcommand{\F}{\mathcal{F}}

\newcommand{\m}{\mu}

\newcommand{\bh}{{\bf h}}

\newcommand{\bm}{{\bf m}}

\newcommand{\bW}{{\bf W}}

\newcommand{\e}{\varepsilon}

\def \colon{{:}\;}
\DeclarePairedDelimiter\floor{\lfloor}{\rfloor}
\DeclarePairedDelimiter\bigfloor{\big\lfloor}{\big\rfloor}

\usepackage{fancyhdr}
\usepackage{theoremref}
\usepackage{hyperref}
\hypersetup{
    colorlinks=true,
    linkcolor=red,
    filecolor=magenta,      
    urlcolor=cyan,
}

\title{Joint ergodicity of Hardy field sequences}

\thanks{The author was supported  by the Research Grant - ELIDEK HFRI-FM17-1684.}

\author{Konstantinos Tsinas}

\address[Konstantinos Tsinas]{University of Crete, Department of mathematics and applied mathematics, Voutes University Campus, Heraklion 71003, Greece} \email{kon.tsinas@gmail.com}

\subjclass[2020]{Primary: 37A44; Secondary:    28D05, 05D10, 11B30.}

\keywords{Ergodic averages, joint ergodicity, Hardy fields}

\begin{document}

\maketitle

\begin{abstract}
    We study mean convergence of multiple ergodic averages, where the iterates arise from smooth functions of polynomial growth that belong to a Hardy field.  Our results include all logarithmico-exponential functions of polynomial growth, such as the functions $t^{3/2}, t\log t$ and $e^{\sqrt{\log t}}$. We show that if all non-trivial linear combinations of the functions $a_1,...,a_k$ stay logarithmically away from rational polynomials, then the $L^2$-limit of the ergodic averages $\frac{1}{N} \sum_{n=1}^{N}f_1(T^{\floor{a_1(n)}}x)\cdot ...\cdot f_k(T^{\floor{a_k(n)}}x)$ exists and is equal to the product of the integrals of the functions $f_1,...,f_k$ in ergodic systems, which establishes a conjecture of Frantzikinakis. Under some more general conditions on the functions $a_1,...,a_k$, we also find characteristic factors for convergence of the above averages and deduce a convergence result for weak-mixing systems.
\end{abstract}

\tableofcontents

\section{Introduction and main results}

The study of the multiple ergodic averages \begin{equation}\label{0}
    \frac{1}{N} \sum_{n=1}^{N}f_1(T^{a_1(n)}x)\cdot ... \cdot f_k(T^{a_k(n)}x)
\end{equation}for general sequences $a_1(n),...,a_k(n)$ of integers, where $T$ is an invertible measure preserving map on a probability space $(X,\mathcal{X},\m)$, has been an active field of research in recent years. Born from Furstenberg's proof of Szemer\'{e}di's theorem \cite{Fu1} with ergodic theoretic tools, mean convergence of the averages \eqref{0} has been established for a wide variety of sequences. In this article, our main result is that if the sequences $a_1,...,a_k$ arise from smooth functions of polynomial growth\footnote{A function $f$ is said to have polynomial growth, if there exists a positive integer $d$, such that the ratio $\frac{f(t)}{t^d}$ converges to 0, as $t\to+\infty$.} belonging to a Hardy field \cite{Hardy 1,Hardy 2} and satisfy certain independence assumptions, then they are jointly ergodic, that is the $L^2$-limit of the averages in \eqref{0} exists and is equal to the product of the integrals of the functions $f_1,...,f_k$, whenever the underlying system $(X,\mathcal{X},\m,T)$ is ergodic. Some typical examples of sequences that we study are the polynomial sequences with real coefficients, the sequences $\floor{n^{3/2}}$, $\floor{n\log n}$, $\floor{n\log\log n + \exp(\sqrt{\log(n^2+1))}}$, $\floor{n^{\sqrt{2}}/\log^2 n}$ and, in general, sequences arising from functions in the Hardy field $\mathcal{LE}$ of logarithmico-exponential functions.
Our main results also establish a conjecture of Frantzikinakis, namely \cite[Problem 23]{Fraopen} (first appearing in \cite{FraHardy1}), which is the content of Theorem \ref{problem}. Furthermore, it gives a partial answer to \cite[Problem 22]{Fraopen}, which asks for general convergence of averages of functions from a Hardy field and generalizes several known results. In the case of weak mixing systems, we can relax our assumptions on the functions $a_1,...,a_k$ even further and establish a Furstenberg type weak-mixing theorem (generalizing the results in \cite{Bergelson}), which gives a positive answer to \cite[Problem 3]{FraHardy1}.

\subsection{Statement of the problem and main results}\label{mainresults}

In order to state our theorems below, we shall work with Hardy fields $\mathcal{H}$ that contain the Hardy field $\mathcal{LE}$ of logarithmico-exponential functions and are closed under composition and compositional inversion of functions. All the subsequent results in this section will be stated under the above assumption. More background on Hardy fields will be presented in Section \ref{background}, where we also present a Hardy field that satisfies the above property. Before we begin, we present below a theorem concerning the case of single ergodic averages. This is a consequence of Theorems 3.2 and 3.3 in
 \cite{Boshernitzan2}. More precisely, those two theorems handle the case that the function $a$ below has at least linear growth, but the case that $a$ has sub-linear growth rate follows using the same arguments and  the equidistribution results in \cite{Boshernitzan1}. The notions of a (measure-preserving) system and ergodicity are defined in Section \ref{background}.
 
 We denote by $C\Z[t]$ the collection of all real multiples of integer polynomials on some variable $t$.

\begin{customthm}{1}\cite{Boshernitzan2}\label{ergodic}
 Let $a\in\mathcal{H} $ be a function of polynomial growth that satisfies the following condition:  \begin{equation}
   \label{P1}  \tag{{\bf A}}\ \lim\limits_{t\to +\infty}\frac{|a(t)-p(t)|}{\log t}=+\infty \text{ for any polynomial}\  p\in C\Z[t].
 \end{equation}
Then, for any measure preserving system $(X,\m,T)$ and function $f\in L^{2}(\m)$, the averages \begin{equation*}
    \frac{1}{N}\sum_{n=1}^{N} T^{\floor{a(n)}} f
\end{equation*}converge in mean to the conditional expectation $\mathcal{E_{\m}}(f|I(T))$, where $I(T)$ is the invariant factor of the system $(X,\m,T)$.
\end{customthm}

\begin{remark*}The above condition is sufficient, but not necessary for convergence in the single iterate case. However, it encompasses most typical functions in $\mathcal{H}$ that are not rational polynomials.
\end{remark*}

We will give the following definition.
\begin{definition}
We will call a function $a\in\mathcal{H}$ 1-good, if it satisfies \eqref{P1}.
\end{definition}The term 1-good indicates the relation of these functions with single ergodic averages averages.

 We show that a natural extension of the above condition implies norm convergence in the case of multiple averages. If $a_1,...,a_k$ are general sequences or functions, we will denote by $\mathcal{L}(a_1,...,a_k)\subseteq\mathcal{H}$ the set of non-trivial linear combinations of the functions  $a_1,...,a_k$ (here $\mathcal{H}$ is a vector space over $\R$). The following theorem is the main result of this article:

\begin{theorem}\label{problem}
Let $a_1,...,a_k\in\mathcal{H}$ have polynomial growth and assume that every function in $ \mathcal{L}(a_1,...,a_k)$ is 1-good. Then, for any ergodic measure preserving system $(X,\m,T)$ and functions $f_1,...,f_k\in L^{\infty}(\m)$, the averages \begin{equation}\label{multiple}
     \frac{1}{N}\sum_{n=1}^{N} T^{\floor{a_1(n)}} f_1\cdot ...\cdot T^{\floor{a_k(n)}}f_k
\end{equation}converge in mean to the product of the integrals $\int f_1\   d\m \cdot...\cdot \int f_k\   d\m.$
\end{theorem}

\begin{remark*}It is a consequence of our proof that the condition on the linear combinations of the functions $a_1,...,a_k$ can be substituted by the following more general assumption: for any real numbers $t_1,...,t_k\in [0,1)$, not all of them zero, we have \begin{equation*}
    \lim\limits_{N\to+\infty} \frac{1}{N}\sum_{n=1}^{N}e(t_1\floor{a_1(n)}+\cdots + t_k\floor{a_k(n)})=0.
\end{equation*}Actually, this is a necessary and sufficient condition in order to have convergence to the product of the integrals in every ergodic system. This is a consequence of Theorem \ref{jointlyergodic} in the next section.
\end{remark*}

 If we do not impose an ergodicity assumption on the system $(X,\m,T)$, then we can show that the averages in the above theorem converge to the product \begin{equation*}
    \mathcal{E}_{\m}(f_1|\mathcal{I}_T)\cdot...\cdot \mathcal{E}_{\m}(f_k|\mathcal{I}_T),
\end{equation*}where $\mathcal{E}_{\m}(f|\mathcal{I}_T)$ is again the projection of $f$ to the invariant factor of the system. This follows from a standard ergodic decomposition argument, and thus, we will usually assume below that the system $(X,\m,T)$ is ergodic.

This theorem extends known results about ergodic averages of functions from a Hardy field. In the case of real polynomials, Theorem \ref{problem} was established in \cite{Koutsogianniscorrelation}.
Theorem \ref{problem} was also proven in \cite{FraHardy1} when all functions $a_1,...,a_k$ have different growth rates and satisfy $t^{N_i+\e}\ll a_i(t)\prec t^{N_i+1}$ for non-negative integers $N_i$ and some $\e>0$. In addition, Theorem \ref{problem} was established in \cite{BerMorRic} under a variant of our condition. More precisely, an independence condition on the functions $a_1,...,a_k$ and on all of their derivatives was imposed. It was proven, however, that if we use a weaker averaging scheme than Ces\'{a}ro averages, we can establish uniform convergence results for the corresponding multiple ergodic averages\footnote{ In our setting, if we substitute the standard Ces\'{a}ro averages in \eqref{multiple} with uniform ones, then Theorem \ref{problem} is known to fail. This is because of  the fact that, if a function $f\in\mathcal{H}$ satisfies $t^k\prec f(t)\prec t^{k+1}$ for some non-negative integer $k$,  we can find arbitrarily large intervals, such that $\floor{f(n)}$ takes only odd (or only even) values. Then, this assertion fails for the rotation by $1/2$ on the torus $\R/\Z$.} (cf. \cite{BerMorRic2} for similar arguments and some nice multiple recurrence and combinatorial results). Finally, Theorem \ref{problem} was established recently for linear combinations of tempered functions from a Hardy field and real polynomials in \cite{Frajoint} (for functions $f$ belonging to $\mathcal{H}$, the tempered condition is equivalent to the relation $t^k\log t\prec f(t)\ll t^{k+1}, $ for some non-negative integer $k$). Our result is more general, since for example we can see that it covers even simple collections of functions like $\{t\log t,t^2\log t\}$, for which convergence has not been established in the literature.

A variant of Theorem \ref{problem} for commuting transformations was proven in \cite{FraHardy2} (under more restrictive conditions). Our methods fail to extend Theorem \ref{problem} to this case, the main reason being that we cannot establish seminorm estimates for convergence of the averages \eqref{multiple}. Indeed, even in the case when the iterates are integer polynomials, characteristic factors have only been described in some special cases like \cite{Chu} or \cite{Host} and more recently in \cite{DFK} and \cite{DFKS}, where joint ergodicity of commuting transformations along polynomials was studied. Finally, we also remark that a similar problem regarding tempered functions of different growth that do not necessarily belong to some Hardy field was handled in \cite{Koutsogiannistempered}.

\subsection{Characteristic factors and the case of weak-mixing systems}

If our only objective is to find characteristic factors for the averages in \eqref{multiple}, we can relax the conditions of Theorem \ref{problem} considerably. More precisely, we have the following theorem which appeared as a conjecture in \cite[Problem 3]{FraHardy1}. The notion of the Host-Kra factor of a system is defined in the following section.

\begin{theorem}\label{generalfactors}
Assume that the functions $a_1,...,a_k\in\mathcal{H}$ have polynomial growth and satisfy \begin{equation*}
    \lim\limits_{t\to+\infty}\frac {|a_i(t)|}{\log t}= +\infty \ \ \ \text{      for all }\ 1\leq i\leq k 
\end{equation*}and \begin{equation*}
     \lim\limits_{t\to+\infty}\frac {|a_i(t)-a_j(t)|}{\log t}= +\infty \ \ \ \text{      for all }\ \ i\neq j .
\end{equation*}Then, there exists a positive integer $s$ such that, for any measure preserving system $(X,\m,T)$, we have \begin{equation*}
   \lim\limits_{N\to\infty} \bignorm{\frac{1}{N}\sum_{n=1}^{N}T^{\floor{a_1(n)}} f_1\cdot ...\cdot T^{\floor{a_k(n)}}f_k-\frac{1}{N}\sum_{n=1}^{N} T^{\floor{a_1(n)}} \widehat{f}_1\cdot ...\cdot T^{\floor{a_k(n)}}\widehat{f}_k}_{L^2(\m)}=0,
\end{equation*}where $\widehat{f}_i:=\mathcal{E}_{\m}(f_i|Z_s(X))$ is the projection of $f_i$ to the $s$-step Host-Kra factor of the system.
\end{theorem}

The conditions above are necessary (one can consider some weakly-mixing systems that are not strongly-mixing to see this).
Since for weak-mixing systems, the Host-Kra factors of any order are trivial, we get the following corollary, which extends the results in \cite[Theorem 1.2]{Bergelson} where the iterates are polynomials taking integer values on the integers, as well as some of the results in \cite{BergelsonHaland} involving tempered functions.

\begin{corollary}
Assume that the functions $a_1,...,a_k\in\mathcal{H}$ have polynomial growth and satisfy \begin{equation*}
    \lim\limits_{t\to+\infty}\frac {|a_i(t)|}{\log t}= +\infty \ \ \ \text{      for all }\ 1\leq i\leq k 
\end{equation*}and \begin{equation*}
     \lim\limits_{t\to+\infty}\frac {|a_i(t)-a_j(t)|}{\log t}= +\infty \ \ \ \text{      for all }\ \ i\neq j .
     \end{equation*}
     Then, for any weak-mixing system $(X,\m,T)$, we have \begin{equation*}
          \lim\limits_{N\to\infty} \frac{1}{N}\sum_{n=1}^{N} T^{\floor{a_1(n)}} f_1\cdot ...\cdot T^{\floor{a_k(n)}}f_k=\int f_1\ d\m \cdot...\cdot\int f_k\  d\m ,
     \end{equation*}where convergence takes place in $L^2(\m)$.
\end{corollary} 

\begin{remark*}The proof of this restricted form of Theorem \ref{generalfactors} still requires a large portion of the arguments that are used in this article.
\end{remark*}

\subsection{Combinatorial Applications}

As a corollary of Theorem \ref{problem}, we get the following multiple recurrence result.

\begin{corollary}
Let $a_1,...,a_k$ be functions from a Hardy field $\mathcal{H}$ such that every non-trivial linear combination of the functions is 1-good. Then, for any measure preserving system $(X,\m,T)$ and any set $A\subset X$ with $\m(A)>0$, we have \begin{equation*}
   \lim\limits_{N\to\infty} \frac{1}{N}\sum_{n=1}^{N}\  \m(A\cap T^{-\floor{a_1(n)}}A\cap\cdots\cap T^{-\floor{a_k(n)}}A)\geq\m(A)^{k+1}.
\end{equation*}
\end{corollary}
A similar result was established in \cite{BerMorRic} with $\limsup$ in place of the limit, but under more general conditions on the functions $a_1,...,a_k$. 

 Utilizing Furstenberg's correspondence principle, we can deduce a combinatorial result about large sets of integers. First of all, we give a definition of the asymptotic density of a set.

Assume $\Lambda\subset \N$. Then, we define the upper density of the set $\Lambda$ as the limit \begin{equation*}
    {\bar{d}}(\Lambda):= \limsup\limits_{N\to\infty} \frac{|\Lambda\cap [1,N]|}{N}
\end{equation*}and the lower density $\underline{d}$ is defined similarly with $\liminf$ instead of $\limsup$. If those limits coincide, then we say that the set $\Lambda$ has natural density $d$ equal to the limit.

\begin{customthm}{}[Furstenberg's correspondence principle]
For any set $\Lambda\subset \N$ with positive upper density, there exists an invertible measure preserving system $(X,\m,T)$ and a measurable set $A\subset X$, such that $\overline{d}(\Lambda)=\m(A)$ and for any $r_1,...,r_k\in \Z$, we have \begin{equation*}
    \bar{d}(\Lambda \cap (\Lambda -r_1)\cap\cdots\cap (\Lambda -r_k))\geq \m(A\cap T^{-r_1}A\cap\cdots\cap T^{-r_k}A).
\end{equation*}
 
\end{customthm}
\begin{corollary}
Let $\Lambda\subset\N$ have positive upper density and let $a_1,...,a_k$ be functions from a Hardy field $\mathcal{H}$, such that every non-trivial linear combination of these functions is 1-good. Then,\begin{equation*}
     \liminf\limits_{N\to\infty} \frac{1}{N}\sum_{n=1}^{N} \bar{d}(\Lambda \cap (\Lambda -\floor{a_1(n)})\cap\cdots\cap (\Lambda -\floor{a_k(n)}))\geq (\bar{d}(\Lambda))^{k+1}.
\end{equation*} 
\end{corollary}

\subsection{General overview of the proof and organization of the paper}

Similarly to the work in \cite{FraHardy2,FraHardy1,BerMorRic}, our approach is to show that the Host-Kra factor introduced in \cite{Host-Kra1} (see also \cite{Host-Kra structures} for a presentation of the general theory) is characteristic for convergence of our averages. This is a technique used extensively in the literature to reduce the problem of convergence in general measure preserving systems to the case where the system is a rotation in a nilpotent homogeneous space. However, we shall use a recent result of Frantzikinakis \cite{Frajoint} which roughly asserts that in order to prove Theorem \ref{problem}, we only need to prove that the Host-Kra factor is characteristic for the averages in \eqref{multiple} plus some simple equidistribution results on the torus (which are simple consequences of the equidistribution results in \cite{Boshernitzan1}). This bypasses the usual hassle of proving convergence of the corresponding averages in nilmanifolds. We remark that this technique can only be used when we expect convergence of certain ergodic averages to the product of the integrals of the involved functions, which is the case in this article. 

Furthermore, there are also several differences between our methods and the methods used in \cite{BerMorRic} and \cite{FraHardy2} to establish seminorm estimates, where a standard PET induction argument was utilized to reduce to the case of functions with sub-linear growth rate. This technique restricts the cases that can be handled, because the van der Corput operation may eventually yield functions that do not satisfy the condition \eqref{P1} (a typical example in this case is the pair of functions $(t\log t, t\log\log t)$, which "drop" to functions that have growth rate smaller than $\log t$ after applications of the van der Corput inequality). In order to overcome this, we use the fact that Hardy field functions of polynomial growth behave "locally" as polynomials. This observation was used in \cite{FraHardy1} to handle the special case of the family $\{\floor{a(n)},2\floor{(a(n)},...,k\floor{a(n)}\}$, where there is only one Hardy field function. It was shown in this case that the corresponding multiple ergodic averages over small intervals converge to 0, provided that $f_1$ is orthogonal to one particular Host-Kra factor of the system. This technique does not extend to the more general case that we wish to cover here. However, we can prove instead that these averages can be bounded by finite ergodic averages, where the iterates are polynomials. We use then a double averaging trick and the asserted asymptotic bounds to show that the Host-Kra factor is, indeed, characteristic for convergence of the multiple averages in our setting.
The price to pay is that our argument has to be somewhat finitary in nature and this makes the proof slightly more technical and cumbersome. We also note here that our work concerns finding bounds for ergodic averages involving families of variable polynomials. Some general convergence results regarding multiple ergodic averages of variable polynomials were recently established in \cite{Koutsogiannisvariable}. 

A general difficulty in the proofs is that functions in $f  \in \mathcal{H}$ that satisfy $f(t)\ll t^{\delta}$ for all $\delta>0$ (such as the functions $(\log t)^c$, where $c>1$) behave differently from functions that dominate some fractional power. We describe this difference more clearly in the Appendix, where we also provide several propositions and lemmas that will be used extensively in Sections \ref{sectionfactors} through \ref{reductionestimates}. In addition, we will revisit the ideas discussed above in Section \ref{sectionfactors} and also present some examples that we believe help illustrate the argument of the main proofs.

Our results do not cover the case of general convergence (not necessarily to the product of the integrals) of the averages in \eqref{multiple}. In this case, the 1-good assumption on the functions can be relaxed further to include more functions, like the polynomials with integer coefficients. In order to establish this, we need to deal with the case of convergence in nilsystems, which will be done in a subsequent article.

\subsection{Some open problems }

An interesting problem that arose when trying to prove our main result is whether sequences of the form $\floor{a(n)}^{\ell}$ are good for the multiple ergodic theorem, where $\ell$ is a natural number. In the special case $\ell=2$, we present the following problem:

\begin{conjecture}
Let $c_1,...,c_k$ be distinct positive non-integers. Do the averages
\begin{equation*}
    \frac{1}{N}\sum_{n=1}^{N} T^{\floor{n^{c_1}}^2} f_1\cdot...\cdot T^{\floor{n^{c_k}}^2} f_k
\end{equation*}converge in mean?
\end{conjecture}

Conjecture 1 seems non-trivial even in the case where the fractional powers $n^{c_i}$ are replaced by general (non-integer) real polynomials.
In the case $k=1$, we can use the spectral theorem and the equidistribution results presented in \cite{Fraeq} to give a positive answer. In particular, it was proven in the same article that $\floor{a(n)}^k$ is good for the ergodic theorem when $a$ stays logarithmically far from real multiples of integer polynomials (this is condition \eqref{P1}). If we apply the van der Corput inequality, the resulting sequences at each step become very complicated and may oscillate substantially. As a consequence, the classical methods of finding characteristic factors do not seem to yield a result in this case. 

Furthermore, we expect that the above averages are jointly ergodic for totally ergodic systems: \begin{conjecture}
Let $c_1,...,c_k$ be distinct positive non-integers and let $(X,\m,T)$ be a totally ergodic system. Show that the averages
\begin{equation*}
    \frac{1}{N}\sum_{n=1}^{N} T^{\floor{n^{c_1}}^2} f_1\cdot...\cdot T^{\floor{n^{c_k}}^2} f_k
\end{equation*}converge to the product of the integrals $\int f_1\ d\m \cdot...\cdot \int f_k\ d\m$ for any functions $f_1,...,f_k\in L^{\infty}(\m)$.
\end{conjecture}
The above problem is interesting even for weak-mixing systems.
 In addition, the problem of multiple recurrence is also open. In its simplest form, we have the following open question. 

\begin{conjecture}
Show that any set of positive upper density contains patterns of the form \begin{equation*}
    \{m,m+\floor{n^a}^2,m+\floor{n^b}^2, \ m,n\in\N\},
\end{equation*}where $a,b>1$ are distinct non-integers.
\end{conjecture}

In order to establish this, it may be possible to sidestep the more difficult problem of proving convergence of the corresponding ergodic averages (see, for example, the arguments in \cite{Wierdl}). We do not concern ourselves with this here, however.

\subsection{Acknowledgements}
I would like to thank my advisor Nikos Frantzikinakis for helpful discussions. I thank the anonymous referee for useful remarks and corrections on the previous version of this article. Finally, I would like to thank Dibyendu De for pointing out corrections in several parts of the previous version.

\subsection*{Notational conventions}We use $\N$ to denote the set of natural numbers, while $\Z^{+}$ denotes the non-negative integers.
For two sequences $a_N$ and $b_N$, we say that $b_N$ {\em dominates} $a_N$ and write $a_N\prec b_N$ or $a_N=o(b_N)$, if and only if the fraction $\Big|\frac{a_N}{b_N}\Big|$ tends to 0 as $N\to\infty$ and we write $a_N\sim b_N$, when this limit is is a finite non-zero real number. In the latter case, we say that the sequences $a_N$ and $b_N$ have the same growth rate. In addition, we write $a_N\ll b_N$ or $a_N=O(b_N)$ if there exists a constant $C$ such that, $|a_N|\leq C|b_N|$ for all $N\in \N$. When we want to express dependence on some parameters $c_1,...,c_k$ in the above bounds, we will use the notation $a_N=O_{c_1,..,c_k}(b_N)$ instead. We use similar asymptotic notation when we compare growth rates of functions in some real variable $t$. Furthermore, for a real valued function $f$ we will denote by $f^{(k)}$ the $k$-th order derivative of $f$, assuming it is well defined.

We will sometimes use bold letters to distinguish between scalar and vector valued quantities. For a positive integer $M$, we will use $[M]$ to denote the set $\{1,2,...,M\}$. Given a sequence $a(n)$ and a real number $x\geq 1$, we will use the averaging notation \begin{equation*}
    \underset{1\leq n\leq x}{\E}a(n):=\frac{1}{\floor{x}}\sum_{n=1}^{\floor{x}}a(n).
\end{equation*}

Consider a positive integer $s$. We will denote by $[[s]]$ the set $\{0,1\}^s$ of ordered $s$-tuples of zeroes and ones, which contains $2^s$ elements.  For elements of the set $[[s]]$, we will use the notation $\underline{\e}$ instead of bold letters. For convenience, we will write $\underline{0}, \underline{1}$ for the elements $(0,0,...0)$ and $(1,1,...,1)$ of $[[s]]$ respectively.
We will also define $|\underline{\e}|$ to be the sum of elements of $\ue$. For a finite set $Y$, we will similarly use the notation $Y^{[[s]]}$ to denote the set $Y^{2^s}$. Each element ${\bf h}\in Y^{[[s]]}$ can be represented as ${\bf h}=(h_{\underline{\e}},\underline{\e}\in [[s]])$ where each $h_{\underline{\e}}$ belongs to $Y$.

For complex numbers $z$, we define the operator $\mathcal{C}^kz$, where $\mathcal{C}^kz:=z$, if $k$ is an even number and $\mathcal{C}^kz:=\bar{z}$ otherwise. Finally, we use the notation $e(t):=e^{2\pi it}$ for $t\in \R$.

\section{Background material}\label{background}
 
\subsection{Preliminaries on Hardy fields}

Let $\mathcal{B}$ denote the set of germs at infinity of real valued functions defined on a half-line $[x,+\infty]$. Then, $(\mathcal{B},+,\cdot)$ is a ring. A sub-field $\mathcal{H}$ of $\mathcal{B}$ that is closed under differentiation is called a Hardy field. We will say  that $a(n)$ is a Hardy sequence, if for $n\in \N$ large enough we have $a(n)=f(n)$ for some function $f\in\mathcal{H}$. We will make some small abuse of language and sometimes also refer to sequences of the form $\floor{f(n)}$ as Hardy sequences. 

An example of a Hardy field is the field $\mathcal{LE}$ of logarithmico-exponential functions. These are the functions defined on some half line of $\R$ by a finite combination of the operations $+,-,\cdot,\div,\  \exp$, $\log$ and composition of functions acting on a real variable $t$ and real constants. The set $\mathcal{LE}$ contains functions such as the polynomials $p(t)$, $t^c$ for all real $c>0$, $t\log t $, $t^{(\log t)^2}$ and $e^{\sqrt{t}}/t^2$. 

The main advantage when working with functions in a Hardy field (instead of just the $C^{\infty}$ functions) is that any two functions $f,g\in\mathcal{H}$ are comparable. That means that the limit \begin{equation*}
    \lim\limits_{t\to\infty} \frac{f(t)}{g(t)}
\end{equation*}exists and thus it makes sense to talk about and compare their growth rates. In addition, since every function in our Hardy field has a multiplicative inverse, we can easily infer that every function in $\mathcal{H}$ is eventually monotone (and, therefore, has constant sign eventually). 

It will be crucial in the proof to assume that $\mathcal{H}$ is closed under composition and compositional inversion of functions, when defined. More precisely, if $f,g\in \mathcal{H}$ are such that $\lim\limits_{t\to+\infty}g(t)= +\infty$, then we have that $f\circ g\in\mathcal{H}$ and $g^{-1}\in\mathcal{H}$. The Hardy field $\mathcal{LE}$ does not have this property. This can be achieved by working with the Hardy field $\mathcal{P}$ of Pfaffian functions \cite{Kovanski}, which contains $\mathcal{LE}$ and satisfies the previously mentioned assumptions. This field can be defined inductively as follows:\\
i) Let $\mathcal{P}_1$ be the set of the smooth functions satisfying the differential equation $f'=p(t,f)$ for some polynomial $p$ with integer coefficients.\\
ii) Let $\mathcal{P}_k$ be the set of the smooth functions satisfying the differential equation $f'=p(t,f_1,...,f_k)$ for some polynomial $p$ with integer coefficients and $f_i\in P_i$ for $1\leq i\leq k-1$.
Then $\mathcal{P}$ contains all germs at infinity of the set $\cup_{i=1}^{\infty} \mathcal{P}_i$.

From now on, we will assume that $\mathcal{H}$ has all the above properties. In the appendix, we have gathered some lemmas regarding growth rates of functions in $\mathcal{H}$, which will play a crucial role in the approximations in the following sections.

Finally, we give some definitions for functions whose growth rate is of particular interest.

\begin{definition}We say that a function $f\in\mathcal{H}$ has sub-linear growth rate (or is sub-linear), if $f(t)\prec t$.
We say that a function $f\in\mathcal{H}$ has sub-fractional growth rate (or is sub-fractional), if for all $\delta>0$, we have $f(t)\ll t^{\delta}$.
\end{definition}

Typical examples of sub-linear functions are $\sqrt{t}$, $e^{\sqrt{\log t}}$ and $\log^3(t)$. Among these, the functions $e^{\sqrt{\log t}}$ and $\log^3(t)$ are also sub-fractional, while the first one is not sub-fractional.

\begin{definition}
We will call a function $f\in\mathcal{H}$ of polynomial growth strongly non-polynomial, if there exists a non-negative integer $d$, such that \begin{equation*}
    t^{d}\prec f(t)\prec t^{d+1}.
\end{equation*}
\end{definition}
For example, the functions $t^{3/2}$ and $\log^3(t)$ are strongly non-polynomial, while the function $t^2+\sqrt{t}$ is not.

\subsection{Background in ergodic theory}

\subsubsection{Ergodicity and factors}

A measure preserving system is a probability space $(X,\mathcal{X},\m)$ equipped with an invertible measure preserving transformation $T$. We call a system ergodic, if the only $T$-invariant functions in $L^{\infty}(\m)$ are the constant ones. The system $(X,\mathcal{X},\m,T)$ is called weak-mixing, if the product system $(X\times X,\mathcal{X}\times \mathcal{X},\m\times\m,T\times T)$ is ergodic.
We say the system $(Y,\mathcal{Y},\nu,S)$ is a factor of $(X,\mathcal{X},\m,T)$, if there exist $X'\subset X$, $Y'\subset Y$ of full measure that are invariant under $T$ and $S$ respectively and a map $p:X'\to Y'$ such that $\nu=\m\circ p^{-1}$ and $p\circ T(x)=S\circ p(x)$ for all $x\in X'$. If $p$ is a bijection, we say that the two systems are isomorphic. A factor of the system $(X,\mathcal{X},\m,T)$ corresponds to a $T$-invariant sub-$\sigma$-algebra of $\mathcal{X}$ (in the above example this $\sigma$-algebra is $p^{-1}(\mathcal{Y})$). From now on, we will omit the $\sigma$-algebra $\mathcal{X}$ from the quadruple $(X,\mathcal{X},\m,T)$.

\subsubsection{Host-Kra seminorms and factors}
Let $(X,\m,T)$ be an invertible measure preserving system and let $f\in L^{\infty}(\m)$. We define the Host-Kra uniformity seminorms inductively as follows:\begin{equation*}
    \nnorm{f}_{0,T}:=\int f \ d\m
\end{equation*}and, for $s\in \Z^{+}$, \begin{equation}\label{uniformitynorms}
    \nnorm{f}_{s+1,T}^{2^{s+1}}:=\lim\limits_{H\to\infty}\underset{0\leq h\leq H}{\E} \nnorm{\bar{f}\cdot T^h f}_{s,T}^{2^s}.
\end{equation}

The existence of the limits above was proven in \cite{Host-Kra1} in the ergodic case (for the non-ergodic case, see \cite{Host-Kra structures} for a proof) and it was also established that the $\nnorm{\cdot}_s$ are indeed seminorms (for $s\neq 0)$. More importantly, it was also shown in the same article that the seminorms $\nnorm{f}_{s,T}$ define a factor $Z_{s-1}(X)$ of $X$, which is characterized by the following property:\begin{equation*}
    f\perp L^2(Z_{s-1}(X))\iff \nnorm{f}_{s,T}=0.
\end{equation*}
It can be shown that the factors $Z_s(X)$ form an increasing sequence of factors. This follows from the inequality $\nnorm{f}_{s,T}\leq \nnorm{f}_{s+1,T}$, for all non-negative integers $s$. For weak-mixing systems, it can be shown that all the factors $Z_s(X)$ are trivial.

Furthermore, it is easy to prove that $\nnorm{\bar{f}\otimes f}_{s,T\times T}\leq \nnorm{f}_{s+1,T}^2$, where $\bar{f}\otimes f$ denotes the function $(x,y)\to \overline{f(x)}f(y)$ on $(X\times X,\m\times \m,T\times T)$. Finally, when there is no danger of confusion, we will omit the subscript $T$ in the seminorms and write simply $\nnorm{f}_s$.

\subsubsection{Joint ergodicity of sequences } Let $a_1(n),...,a_k(n)$ be sequences of integers. Following the terminology in \cite{Frajoint}, we call these sequences {\em jointly ergodic}, if for any ergodic measure preserving system $(X,\m,T)$ and functions $f_1,...,f_k\in L^{\infty}(\m)$, we have \begin{equation*}
   \lim\limits_{N\to+\infty} \frac{1}{N} \sum_{n=1}^{N}T^{a_1(n)}f_1\cdot ... \cdot T^{a_k(n)}f_k =\int f_1\ d\m \cdot ....\cdot \int f_k\ d\m,
\end{equation*}where convergence takes place in $L^2(\m)$. We also give the following definitions:

\begin{definition}We say that a collection of sequences $a_1,...,a_{k}$ of integers:\\
i) is {\em good for seminorm estimates}, if for every ergodic system $(X,\m,T)$ there exists an $s\in \N$, such that if $f_1,...,f_k\in L^{\infty}(\m)$ and $\nnorm{f_{\ell}}_s=0$ for some $\ell\in\{1,...,k\}$, then\footnote{In \cite{Frajoint}, this property is called "{\em very good for seminorm estimates}".}\begin{equation*}
    \lim\limits_{N\to+\infty }\frac{1}{N} \sum_{n=1}^{N}T^{a_1(n)}f_1\cdot ... \cdot T^{a_k(n)}f_k=0
\end{equation*}in $L^2(\m)$.\\
ii) is {\em good for equidistribution}, if for all $t_1,...,t_{k}\in [0,1)$, not all of them zero, we have \begin{equation*}
    \lim\limits_{N\to+\infty} \frac{1}{N} \sum_{n=1}^{N} e(t_1a_1(n)+\cdots+t_ka_k(n))=0.
\end{equation*}
\end{definition}

The main result in \cite{Frajoint}, which we are also going to use is the following:
\begin{customthm}{2}\cite[Theorem 1.1]{Frajoint}   \label{jointlyergodic}
 Let $a_1,...,a_k$ be a collection of sequences of integers. Then, the following are equivalent:\\
 i) The sequences $a_1,...,a_k$ are jointly ergodic.\\
 ii) The sequences $a_1,...,a_k$ are good for seminorm estimates and good for equidistribution.
\end{customthm}

\begin{proof}[Proof that Theorem \ref{problem} follows from Theorem \ref{generalfactors}]
Note that every 1-good function dominates the logarithmic function $\log t$. Therefore, if the functions $a_1,...,a_k$ are such that every non-trivial linear combination of them is 1-good,
then the hypotheses of Theorem \ref{generalfactors} are satisfied, which means that the sequences $\floor{a_1(n)},...,\floor{a_k(n)}$ are good for seminorm estimates. Therefore, due to Theorem \ref{jointlyergodic} we only need to prove that they are good for equidistribution. This, however, follows from the equidistribution results in \cite{Boshernitzan1} and has been established in \cite[Proposition 6.3]{Frajoint}.
\end{proof}

\section{Characteristic factors for Hardy sequences}\label{sectionfactors}

 In this section, we present the main proposition that asserts that the Host-Kra factors of a given system are characteristic for the convergence of the averages \eqref{multiple}. That means that if we substitute the functions $f_i$ by their projections on $Z_{s}(X)$ for some suitable $s\in \N$, then the limiting behavior of the average in \eqref{multiple} remains unchanged. We will also make some small reductions to the original problem and prove some useful lemmas. We also provide a brief overview of the proof and some examples that present the main ideas, while avoiding most of the technicalities. The following proposition will be proven in subsequent sections.

\begin{proposition}\label{factors}
Assume that the functions $a_1,a_2,...,a_k\in\mathcal{H}$ have polynomial growth and suppose that the following two conditions hold:

i) The functions $a_1,...,a_k$ dominate the logarithmic function $\log t$.

ii) The pairwise differences $a_i-a_j$ dominate the logarithmic function $\log t$ for any $i\neq j$.

Then, there exists a positive integer $s$ depending only on the functions $a_1,...,a_k$, such that for any measure preserving system $(X,\m,T)$, functions $f_1\in L^{\infty}(\m)$ and $f_{2,N}...,f_{k,N}\in L^{\infty}(\m)$, all bounded by $1$, with $f_1\perp Z_{s}(X)$, the expression \begin{equation}\label{factor}
   \sup_{|c_{n}|\leq 1} \norm{\underset{1\leq n\leq N}{\E}\ c_{n}\  T^{\floor{a_1(n)}} f_1 \cdot T^{\floor{a_2(n)}}f_{2,N}\cdot ...\cdot T^{\floor{a_k(n)}}f_{k,N}}_{L^2(\m)}
\end{equation}converges to 0, as $N\to+\infty$.
\end{proposition}

\begin{remarks*}
i) It is possible to establish Proposition \ref{factors} under the weaker assumption that only the functions $a_1,a_1-a_2,...,a_1-a_k$ dominate the logarithmic function, but this requires a few more details in the proof and is not required for the proof of Theorem \ref{generalfactors}.\\
ii) It may be possible to establish that the number $s$ does not, in fact, depend on the functions $a_1,...,a_k$, but it can be bounded by a function involving the number $k$ of functions and the highest degree\footnote{This means the smallest integer $d$, for which $a_i(t)\ll t^{d}$ for all $1\leq i\leq k$.} $d$ of the involved functions. However, we do not concern ourselves here with the optimal value of $s$. In particular, we will use polynomial expansions of the functions $a_1,...,a_k$ with degrees very large compared to the number $d$, which means that any possible dependence on $d$ will be lost in the proof.
\end{remarks*}

It is obvious that Proposition \ref{factors} implies Theorem \ref{generalfactors} (this follows from a standard telescoping argument). Therefore, all of our remaining results follow if we establish this proposition.

The reason that we work with sequences of functions and the bounded sequence $c_n$ is because that will be helpful in some spots to absorb some of the error terms that will appear in the iterates and also allows us to "transform" the sequences in the iterates, so that we can reduce our problem to the case that the first sequence $a_1$ has some specific properties depending on the situation. As an example, we claim that we only need to consider the case when the function $a_1(t)$ has maximal growth in the family $\{a_1,...,a_k\}$. Indeed, suppose that this is not the case. Then, there exists a function $a_i$ for some $ i\in\{1,...,k\}$ with $a_1\prec a_i$. Without loss of generality, assume that the function $a_k$ has maximal growth rate. It is sufficient to show that for any sequence of functions $g_N$ with $||g_N||_{L^{\infty}(\m)}\leq 1$, we have \begin{equation*}
   \lim\limits_{N\to\infty} \underset{1\leq n\leq N}{\E}c_{n,N}\int g_N\  T^{\floor{a_1(n)}} f_1\cdot ...\cdot T^{\floor{a_k(n)}}f_{k,N}\  d\m =0. 
\end{equation*}Then, we can choose the function $g_N$ to be the conjugate of the average \begin{equation*}
    \underset{1\leq n\leq N}{\E}c_{n,N}\int   T^{\floor{a_1(n)}} f_1\cdot ...\cdot T^{\floor{a_k(n)}}f_{k,N} \ d\m
\end{equation*}to get our claim. Composing with $T^{-\floor{a_k(n)}}$ and applying the Cauchy-Schwarz inequality, it is sufficient to show that \begin{equation*}
   \lim\limits_{N\to+\infty} \sup_{|c_{n}|\leq 1} \norm{\underset{1\leq n\leq N}{\E}\  c_{n}\ T^{-\floor{a_{k}(n)}}g_N\cdot T^{\floor{a_1(n)}-\floor{a_k(n)}} f_1\cdot ...\cdot T^{\floor{a_{k-1}(n)}-\floor{a_k(n)}}f_{k-1,N}}_{L^2(\m)}=0.
\end{equation*}We can write $\floor{a_i(n)}-\floor{a_k(n)}=\floor{a_i(n)-a_k(n)} +e_{i,n}$, where the errors $e_{i,n}$ take values in $\{0,\pm 1\}$. Using Lemma \ref{errors} below, the errors can be absorbed by the supremum outside the average and, therefore, the function that corresponds to $f_1$ is equal to $a_1-a_k$, which now has maximal growth rate among the new family of functions. It is also easy to check that the new family satisfies the conditions of Proposition \ref{factors}.

This notion of absorbing the errors that we described above can be made more precise by the next lemma.
\begin{lemma}\label{errors}
    Assume that the integers $e_{i,n,N}$ take values in a finite set $S$. Then, for any sequences $a_{i,N}$ of integers, complex numbers $c'_{n,N}$ bounded in magnitude by 1 and any 1-bounded functions $f_{i,N}$, we have \begin{align*}
      \norm{  \underset{1\leq n\leq N}{\E}\  c'_{n,N}\ T^{a_{1,N}(n)+e_{1,n,N}}f_{1,N}\cdot...\cdot T^{a_{k,N}(n)+e_{k,n,N}}f_{k,N}}_{L^2(\m)}&\ll_{k,S} \\
      \sup_{|c_{n,N}|\leq 1}\ \sup_{||f_2||_{\infty}\leq 1,...,||f_k||_{\infty}\leq 1}^{} \ \norm{\underset{1\leq n\leq N}{\E}\  c_{n,N}\  T^{a_{1,N}(n)}f_{1,N}\cdot T^{a_{2,N}(n)}f_{2} \cdot ...\cdot T^{a_{k,N}(n)}f_{k}}_{L^2(\m)}&.
    \end{align*} As a consequence, there exist 1-bounded functions $f'_{i,N}$, such that the original expression is bounded by a constant multiple of the quantity \begin{equation*}
        \sup_{|c_{n,N}|\leq 1}\norm{  \E_{1\leq n\leq N} \ c_{n,N}\  T^{a_{1,N}(n)}f_{1,N}\ T^{a_{2,N}(n)}f'_{2,N}\cdot ...\cdot T^{a_{k,N}(n)}f'_{k,N}}_{L^2(\m)} +o_N(1).
    \end{equation*} 
\end{lemma}
\begin{proof}
   We partition the integers $n$ into a finite number of sets, in which all the quantities $e_{i,n,N} $ are constant (as $n$ varies). There are at most $|S|^{k}$ such sets. If $A_1,...,A_{|S|^{k}}$ are these sets, then we have \begin{align*}
      &\norm{  \underset{1\leq n\leq N}{\E} c'_{n,N}\  T^{a_{1,N}(n)+e_{1,n,N}}f_{1,N}\cdot...\cdot T^{a_{k,N}(n)+e_{k,n,N}}f_{k,N}}_{L^2(\m)}\leq \\
       &\sum_{i=1}^{|S|^k} \bignorm{\frac{1}{N}\sum_{n\in A_i}^{} \ c'_{n,N}\  T^{a_{1,N}(n)+e_{1,n,N}}f_{1,N}\cdot...\cdot T^{a_{k,N}(n)+e_{k,n,N}}f_{k,N} }_{L^2(\m)}\leq \\
       &|S|^{k} \max_{1\leq i\leq |S|^k}  \bignorm{\frac{1}{N}\sum_{1\leq n\leq N}^{} \ c'_{n,N} \mathbbm{1}_{A_i}(n)\ T^{a_{1,N}(n)}f_{1,N}\cdot...\cdot T^{a_{k,N}(n)+e_{k,n,N}-e_{1,n,N}}f_{k,N} } _{L^2(\m)}   \leq \\
       &|S|^{k}  \sup_{|c_{n,N}|\leq 1}\ \sup_{||f_2||_{\infty}\leq 1,...,||f_k||_{\infty}\leq 1}^{} \norm{\underset{1\leq n\leq N}{\E} c_{n,N}\  T^{a_{1,N}(n)}f_{1,N}\ T^{a_{2,N}(n)}f_{2}\cdot...\cdot T^{a_{k,N}(n)}f_{k}}_{L^2(\m)},
  \end{align*}which is the required result. In the second to last relation, we composed with $T^{-e_{1,n,N}}$, because $e_{1,n,N}$ is  constant when $n$ is restricted to the set $A_i$.
\end{proof}

\begin{remark*}In the following sections, we will encounter situations where we have some error terms in the iterates. The above lemma is not applied verbatim to all cases below. However, the reasoning presented above (i.e. partitioning into sets where the error sequences are constant) can be applied directly every time to remove these error terms. In particular, we can also show (using the same arguments) that a similar statement holds for double averages, that is, if $I_r$ are a sequence of intervals with lengths going to infinity, $d$ is a natural number and the error terms $e_{i,n,R}$ take values on a finite set $S$ of integers, then \begin{multline*}
    \underset{1\leq r\leq R}{\E}\bignorm{\underset{n\in I_r}{\E} c'_{n,R}\  T^{a_{1,R}(n)+e_{1,n,R}}f_{1,R}\cdot...\cdot T^{a_{k,R}(n)+e_{k,n,R}}f_{k,R}}_{L^2(\m)}^d\ll_{S,k,d} \\
  \sup_{||f_2||_{\infty}\leq 1,...,||f_k||_{\infty}\leq 1}^{}  \ \underset{1\leq r\leq R}{\E}\ \sup_{|c_{n,R}|\leq 1}\ \bignorm{\underset{n\in I_r}{\E} c_{n,R}\  T^{a_{1,R}(n)}f_{1,R}\cdot T^{a_{2,R}(n)}f_{2}...\cdot T^{a_{k,R}(n)}f_{k}}_{L^2(\m)}^d,
    \end{multline*}where we also use the H\"{o}lder inequality (which gives dependence on the exponent $d$ in the implicit constants). Therefore, instead of using the same argument repeatedly, we will cite this lemma in such instances and add a comment when a modified version is required.
\end{remark*}

\subsection{Overview of the proof} 
Our main objective is to reduce our problem to the study of ergodic averages of some variable polynomials. Therefore, we will first study asymptotic bounds for certain polynomial families in Section \ref{PETsection}, since they will be required for the proof of Proposition \ref{factors}. This will rely on the van der Corput inequality and an induction argument on the complexity of the family. In Section \ref{sublinearsection}, we will establish bounds for Hardy sequences of a specific form, namely when the involved functions are a sum of a sub-linear function and a polynomial. This will also be required for the general case. In Section \ref{reductionestimates}, we shall finish the proof.

The main idea is that we can approximate the given Hardy functions by Taylor polynomials (possibly constant) in suitable smaller intervals (with lengths going to infinity). We shall reduce our problem to proving a statement of the form \begin{equation}\label{expansion0}
   \lim\limits_{R\to+\infty} \underset{1\leq r\leq R}{\E}   \norm{\underset{n\in I_r}{\E}\  c_{n,R}\ T^{\floor{p_{1,r}(n)}} f_{1,r}\cdot ...\cdot T^{\floor{p_{k,r}(n)}}f_{k,r}}_{L^2(\m)}^{2^t} =0,
\end{equation}where the iterates are variable polynomials and $f_{1,r}$ has the form\begin{equation*}
    f_{1,r}=f_1\cdot T^{\floor{b_1(r)}}h_1\cdot ...\cdot T^{\floor{b_{\ell}(r)}}h_{\ell}
\end{equation*}for sub-linear functions $b_1,...,b_{\ell}$ and $h_1,...,h_{\ell}\in L^{\infty}(\mu)$.

After this reduction, we bound the innermost average using the results from Section \ref{PETsection}. More precisely, we claim that the inner average can be bounded by a quantity of the form \begin{equation*}
    \underset{ \bm\in  [-M,M]^{t}}{\E} \Big|\int T^{\floor{q_1(r,\bm)}}g_{r,1}\cdot...\cdot T^{\floor{q_{\ell}(r,\bm)}} g_{r,\ell} \  d\m  \Big|
\end{equation*}plus some small error terms, where $M$ is a finite integer (independent from the rest of our parameters) and all the functions $g_{r,i}$ are  either $\tilde{f}_r$ or $\overline{\tilde{f}_r}$. In addition, the functions $q_{i}(r,\bm)$ in the iterates are such that, for (almost all) $\bm\in \Z^{l}$, they can be written as a sum of a sublinear function plus a polynomial, which is the special case that we discussed above. Thus, taking first the limit $R\to+\infty$ to use the bounds established in the special case and then taking the limits $M\to+\infty$, we shall reach our conclusion.

The fact that we can reduce our original problem to \eqref{expansion0} is based on the following elementary lemma.

\begin{lemma}\label{mainlemma}
Let $d$ be a positive integer and consider a two-parameter sequence $\big(A_{R,n}\big)_{R,n\in\N}$ in a normed space such that $\norm{A_{R,n}}\leq 1 $ for all possible choices of $R,n\in\N$. Let $L(t)\in\mathcal{H}$ be an  eventually positive function such that $1\prec L(t)\prec t$ and assume that \begin{equation*}
    \limsup\limits_{R\to+\infty} \underset{1\leq r\leq R}{\E} \ \bignorm{ \underset{r\leq n\leq r+L(r)}{\E}    A_{R,n}               }^{d} \leq C
\end{equation*}for some $C>0$. Then, we also have \begin{equation*}
     \limsup\limits_{R\to+\infty} \bignorm{ \underset{1\leq n\leq R}{\E}   A_{R,n}        }\leq C^{1/d}.
\end{equation*}
\end{lemma}

\begin{proof} Combining the power mean inequality and the triangle inequality, we can easily deduce that \begin{equation*}
     \underset{1\leq r\leq R}{\E} \ \bignorm{ \underset{r\leq n\leq r+L(r)}{\E}    A_{R,n}               }^{d}\geq \bignorm{\underset{1\leq r\leq R}{\E}    \big(  \underset{r\leq n\leq r+L(r)}{\E}    A_{R,n} \big) }^d.
 \end{equation*}Therefore, our result will follow if we show that \begin{equation*}
    \bignorm{ \underset{1\leq r\leq R}{\E}    \big( \underset{r\leq n\leq r+L(r)}{\E}    A_{R,n}\big) -\underset{1\leq n\leq R}{\E}   A_{R,n}       }=o_R(1).
 \end{equation*}
 
 Let $u$ be the compositional inverse of the function $t+L(t)$. Our assumptions on the Hardy field $\mathcal{H}$ imply that $u\in\mathcal{H}$. In addition, it is easy to check that $\lim\limits_{t\to+\infty} u(t)/t=1$. Now, we have \begin{equation*}
     \underset{1\leq r\leq R}{\E}    \big(\   \underset{r\leq n\leq r+L(r)}{\E}    A_{R,n} \big) = \frac{1}{R} \big( \sum_{n=1}^{R}p_R(n)A_{R,n} +  \sum_{n=R+1}^{R+L(R)} p_R(n)A_{R,n}\big)
 \end{equation*}for some real numbers $p_R(n)$. Assuming that $n$ (and thus $R$) is sufficiently large (so that $u(n)$ is positive) we can calculate $p_R(n)$ to be equal to \begin{equation*}
     p_R(n)=   \frac{1}{L(\floor{u(n)})+1}+\cdots +\frac{1}{L(n)+1} +o_n(1),
 \end{equation*}since the number $A_{R,n}$ appears on the average $\underset{r\leq n\leq r+L(r)}{\E}$ if and only if $u(n)\leq r\leq n$. Note that $p_R(n)$ is actually independent of $R$ (for $n$ large enough) and therefore, we will denote it simply as $p(n)$ from now on.
We claim that \begin{equation}\label{p(n)limit}
    \lim_{n\to +\infty } p(n)=1.
\end{equation}Let us first see how this finishes the proof. Since for $n$ large enough we must have $p(n)\leq 2$, we can easily deduce that \begin{equation*}
    \frac{1}{R} \sum_{n=R+1}^{R+L(R)} p(n)A_{R,n}=o_R(1).
\end{equation*}Here, we used the fact that $L(t)\prec t$. In addition, we have \begin{equation*}
   \bignorm{ \frac{1}{R}\sum_{n=1}^{R}p(n)A_{R,n}-\frac{1}{R}\sum_{n=1}^{R}A_{R,n} }\leq  \frac{1}{R}\sum_{n=1}^{R}|p(n)-1|, 
\end{equation*}which is also $o_{R}(1)$. Combining the above we reach the desired conclusion. 

In order to establish \eqref{p(n)limit}, we observe that $L(t)$ is eventually strictly increasing, and therefore, we can easily get \begin{equation*}
    \int_{\floor{u(n)}}^{n+1}  \frac{1}{L(t)+1}\ dt \leq p(n)\leq \int_{\floor{u(n)}-1}^{n}\frac{1}{L(t)+1} \ dt.
\end{equation*}Thus, it suffices to show that the integrals on both sides of the above inequality converge to 1. It is straightforward to check that each of these integrals is $o_n(1)$ close to the integral \begin{equation*}
 I_n=   \int_{u(n)}^{n}  \frac{1}{L(t)+1}\ dt.
\end{equation*}Therefore, we only need to prove that $I_n\to 1$. Using the mean value theorem, we can find a real number $h_n\in[u(n),n]$ such that, \begin{equation*}
    I_n=\frac{n-u(n)}{L(h_n)+1}=\frac{L(u(n))}{L(h_n)+1}.
\end{equation*}The last equality follows easily from the definition of $u$. Since $L$ is eventually strictly increasing, we conclude that $I_n$ is smaller than $L(u(n))/(L(u(n))  +1)\leq 1$. In addition, we also have \begin{equation*}
    I_n\geq \frac{L(u(n))}{L(n)+1}.
\end{equation*}The result follows if we show (note that the function $u^{-1}$ is onto in a half line of $\R$) \begin{equation*}
    \lim\limits_{t\to +\infty} \frac{L(t)}{L(u^{-1}(t))+1}=1.
\end{equation*}However, \begin{equation*}
    \frac{L(t)}{L(u^{-1}(t))+1}=\frac{L(t)}{L(t+L(t))+1}=\frac{L(t)}{L(t+L(t))}+o_t(1).
\end{equation*}Using the mean value theorem, we can write \begin{equation*}
    L(t+L(t))=L(t)+L(t)L'(x_t)\ ,
\end{equation*}where $x_t\in [t,t+L(t)]$. Thus, \begin{equation*}
    \frac{L(t+L(t))}{L(t)}=1+L'(x_t)=1+o_t(1)\ ,
\end{equation*}since $L'(t)\ll L(t)/t\prec 1$. The result follows.
\end{proof}

\subsection{Two examples}
a) Whenever we use $\ll$ without indices in this example, we imply that the constants are absolute. Assume that $a(t)=t\log t+\log^3 t $, $b(t)=t\log t$ and $c(t)=\sqrt{t}$. We want to show that there exists $s\in \N$, such that, if $\nnorm{f}_s=0$, then \begin{equation*}
     \underset{1\leq n\leq N}{\E} T^{\floor{n\log n+\log^3 n}}f\cdot T^{\floor{n\log n}} g_1\cdot  T^{\floor{\sqrt{n}}} g_2
 \end{equation*}converges to 0 in $L^2$ as $N\to+\infty$. Here, $g_1$ and $g_2$ are arbitrary 1-bounded functions in $L^{\infty}(\m)$. In view of Lemma \ref{mainlemma}, it suffices to show that \begin{equation}\label{example0}
     \underset{1\leq r\leq R}{\E} \bignorm{ \underset{r\leq n\leq r+L(r)}{\E}   T^{\floor{n\log n+\log^3 n}}f\cdot T^{\floor{n\log n}} g_1\cdot  T^{\floor{\sqrt{n}}}g_2             }_{L^2(\m)}^{2^d}=\underset{1\leq r\leq R}{\E} A_r
 \end{equation}converges to 0 as $R\to+\infty$, for some sub-linear function $L(t)\in \mathcal{H}$ and an integer $d$, both of which we will choose later.

 \subsection*{ Step 1: Reduction to averages of variable polynomials.}
 
 We observe that \begin{equation*}
     A_r = \bignorm{ \underset{0\leq h\leq L(r)}{\E}\  T^{\floor{     (r+h)\log (r+h) +\log^3(r+h)}} f\cdot  T^{\floor{(r+h)\log (r+h)}}g_1 \cdot T^{\floor{\sqrt{r+h}}}    g_2      }_{L^2(\m)}^{2^d}.
 \end{equation*}
 Now, we can use the Taylor expansion to write \begin{equation*}
     (r+h)\log (r+h)=-\frac{h^3}{6x_h^2} + \frac{h^2}{2r}+h(\log r+1)+r\log r, \ \text{ for some} \ \ \ x_h\in[r,r+h],
 \end{equation*}and \begin{equation*}
     \sqrt{r+h}=-\frac{h^2}{8(x'_{h})^{3/2}}+\frac{h}{2\sqrt{r}} +\sqrt{r},  \ \text{ for some} \ \ \ x'_h\in[r,r+h],
 \end{equation*}
 for every $0\leq h\leq L(r)$. Since  \begin{equation*}
     \Big|\frac{h^3}{6x_h^2} \Big|\leq \frac{L(r)^3}{r^2}
 \end{equation*} and 
 \begin{equation*}
     \Big|\frac{h^2}{8(x'_h)^{3/2}}\Big|\leq \frac{L^2(r)}{8r^{3/2}},
 \end{equation*}we conclude that these two last terms are both $o_r(1)$, provided that we choose the function $L(t)$ to satisfy $L(t)\prec t^{2/3}$. We also choose $L(t)\succ t^{1/2}$, so that both the 2-degree term in the expansion of $(r+h)\log (r+h)$ and the 1-degree term in the expansion of $\sqrt{r+h}$ are not bounded (for $h $ taking values in the range $[0,L(r)]$).
 In addition, under the above assumptions, we can also show that \begin{equation*}
    \max_{0\leq h\leq L(r)} |\log^3(r+h)-\log^3(r)|=o_r(1)
 \end{equation*}using the mean-value theorem.
 Therefore, we have\footnote{In this example, we split and combine the integer parts freely, which is not true in general. In our main proof, we explain this argument using Lemma \ref{errors}.} \begin{multline}\label{example}
      A_r\simeq \bignorm{ \underset{0\leq h\leq L(r)}{\E}\  T^{\floor{\frac{h^2}{2r}+h(\log r+1) +r\log r+\log^3 r }} f\cdot  T^{\floor{\frac{h^2}{2r}+h(\log r+1) +r\log r}}g_1\cdot T^{\floor{\frac{h}{2\sqrt{r}}  +\sqrt{r}}}g_2      }_{L^2(\m)}^{2^d}=\\
      \bignorm{ \underset{0\leq h\leq L(r)}{\E}\  T^{\floor{\frac{h^2}{2r}+h(\log r+1) +r\log r }} (g_1\cdot T^{\floor{\log^3 r}}f) \cdot T^{\floor{\frac{h}{2\sqrt{r}}  +\sqrt{r}}}g_2      }_{L^2(\m)}^{2^d},
 \end{multline}which is an average where the iterates are polynomials in $h$. The fact that the $o_r(1)$ terms can be discarded follows from Lemma \ref{errors} and will be explained more thoroughly in the formal proof. Note that the iterates have now become polynomials in the variable $h$.
 
 \begin{remark*}
 In the proof of Proposition \ref{factors} in Section \ref{reductionestimates}, we will choose the function $L(t)$ in order to have a common polynomial expansion as above. Although in this example this is easily done by hand, this will be accomplished in the general case using some lemmas and propositions that are proven in the appendix.
 \end{remark*}
 
  We will use the van der Corput inequality (Lemma \ref{vdc}):\begin{equation*}
     |\underset{1\leq n\leq N}{\E}\ a_n|^{2^d}\ll_d \frac{1}{M}+ \underset{| m|\leq M}{\E} | \underset{1\leq n\leq N}{\E} \langle a_{n+m},a_m \rangle|^{2^{d-1}} +o_N(1)\ ,
 \end{equation*}which holds as long as $ M=o (N)$.
 
 We will deal with a simpler case here, since \eqref{example} requires many applications of the van der Corput inequality and the estimates are quite complicated. We shall find a bound for the average \begin{equation*}
   \underset{1\leq r\leq R}{\E} \bignorm{ \underset{0\leq h\leq L(r)}{\E} T^{\floor{\frac{h^2}{2r}}}f_r }^4=\underset{1\leq r\leq R}{\E} A_r^4\ ,
 \end{equation*}where $f_r =g_1\cdot T^{\floor{\log^3(r)} }f$.

 \subsection*{Step 2: A change of variables trick and bounds for the polynomial averages}

 First of all, we can write $h=k\floor{\sqrt{2r}}+s$, where the integers $k,s$ satisfy $0\leq k\leq L(r)/\floor{\sqrt{2r}}$ and $0\leq s\leq \floor{\sqrt{2r}}-1$. Then, we have \begin{equation*}
     \frac{h^2}{2r}=\frac{k^2\floor{\sqrt{2r}}^2}{2r}+\frac{2k\floor{\sqrt{r}}s}{2r}+\frac{s^2}{2r}.
 \end{equation*}Note that \begin{equation*}
     \Big|\frac{k^2\floor{\sqrt{2r}}^2}{2r}-k^2\Big|\leq 2k^2\frac{\{\sqrt{2r}\}}{\sqrt{2r}} \leq 2\frac{L^2(r)}{\floor{\sqrt{2r}}^2 \sqrt{2r}}.
 \end{equation*}If we choose $L(t)$ to satisfy the additional hypothesis $L(t)\prec t^{3/4}$, then we get that the above quantity is $o_r(1)$. In this example, we can take $L(t)=t^{3/5}$ as our sub-linear function (observe that all of the restrictions we imposed above are satisfied). Therefore, we can use the power mean inequality to deduce that\begin{equation}\label{poiuy}
     A_r^4\leq  \underset{0\leq s\leq \floor{\sqrt{2r}}-1}{\E} \ \bignorm { \underset{1\leq k\leq \frac{L(r)}{\floor{\sqrt{2r}}}}{\E}  T^{\floor{k^2+p_{s,r}(k)}} f_r }^4 
 \end{equation}for some linear polynomials $p_{s,r}(k)$. Denote by $A_{s,r}$ the innermost average in the above relation.

 We fix a positive integer parameter $M$. Applying the van der Corput inequality twice, we deduce that \begin{equation*}
     A_{s,r}^4\ll \frac{1}{M} +\underset{|m_1|,|m_2|\leq M}{\E}  \Big|\int \bar{f_r} \cdot T^{2m_1m_2} f_r\  d\m  \Big| +o_r(1) ,
 \end{equation*}where the implied constant is absolute (and, in particular, independent of $M$). We omitted the routine computations here (the general case is more complicated than this and is handled in Section \ref{PETsection}). This bound holds regardless of the choice of the polynomial $p_{s,r}(k)$. Using this bound in \eqref{poiuy} we deduce that  \begin{equation*}
     A_r^4\ll \frac{1}{M}+\underset{|m_1|,|m_2|\leq M}{\E}  \Big|\int \overline{(g_1\cdot T^{\floor{\log^3 r}}f )} \cdot T^{2m_1m_2} (g_1\cdot T^{\floor{\log^3 r}}f )\  d\m  \Big| +o_r(1).
 \end{equation*}Therefore, the quantity in \eqref{example0} is $\ll$
 \begin{multline}\label{afterstep2}
     \frac{1}{M}+\underset{1\leq r\leq R}{\E}\ \underset{|m_1|,|m_2|\leq M}{\E}  \Big|\int \overline{(g_1\cdot T^{\floor{\log^3 r}}f )} \cdot T^{2m_1m_2} (g_1\cdot T^{\floor{\log^3 r}}f )\ d\m  \Big| +o_R(1)=\\
     \frac{1}{M}+ \underset{|m_1|,|m_2|\leq M}{\E} \ \underset{1\leq r\leq R}{\E} \Big|\int (\bar g_1 \cdot T^{2m_1m_2}g_1)\cdot T^{\floor{\log^3 r}}(\bar{f}\cdot T^{2m_1m_2}f) \  d\m  \Big| +o_R(1).
 \end{multline}
 
 \begin{remark*}
 In the proof of the general case, instead of the sub-linear function $\floor{\log^3(r)}$ in the iterates in \eqref{afterstep2}, we may also have functions of the form $\floor{u(r)}^{k}$, where $u\in\mathcal{H}$ is a sub-linear function and $k\in \Z^{+}$ (like $\floor{\sqrt{r}}^3$ and $\floor{r^{2/3}}^5$). 
  For instance, assume we want to study the limit of the averages \begin{equation*}
     \underset{1\leq n\leq N}{\E} T^{\floor{\sqrt{n}+n^3}}f\cdot T^{\floor{\sqrt{n}}}g.
 \end{equation*}Using Lemma \ref{mainlemma}, it suffices to show that \begin{equation*}
     \underset{1\leq r\leq R}{\E}\bignorm{\underset{0\leq h\leq L(r)}{\E}   T^{\floor{\sqrt{r+h}+(r+h)^3}}f\cdot T^{\floor{\sqrt{r+h}}}g }_{L^2(\m)}^{2^d}
 \end{equation*}for some $d\in \N$ and some sub-linear function $L(t)\in \mathcal{H}$. If we choose $L(t)$ appropriately, then we can write \begin{equation*}
     \sqrt{r+h}=\sqrt{r}+\frac{h}{2\sqrt{r}}+o_r(1)
 \end{equation*}for $0\leq h\leq L(r)$. Now, using the change of variables $h=k\floor{2\sqrt{r}}+s$, we observe that the leading coefficient of the polynomial $(r+h)^3$ in the iterates
 becomes $\floor{2\sqrt{r}}^3$. If we proceed similarly as in step 2 above using repeated applications of the van der Corput inequality, we will arrive at a similar bound as the one in \eqref{afterstep2}, but now the term $\floor{2\sqrt{r}}^3$ will appear in the iterates.

 In order to combat this situation, we need another intermediate step in our proof (this is Step 7 in Section \ref{reductionestimates}). We shall use a lemma that allows us to replace the sub-linear function $2\sqrt{r}$ by the identity function $a(r)=r$. 
 As an example, suppose we want to bound the limit of the averages \begin{equation*}
     \underset{1\leq r\leq R}{\E} T^{\floor{\sqrt{r}}}f\cdot T^{\floor{\sqrt{r}}^{3}+\floor{r^{2/5}} }g
 \end{equation*}as $R\to+\infty$. We rewrite this expression as a function of $\sqrt{r}$ \begin{equation*}
     \underset{1\leq r\leq R}{\E} T^{\floor{\sqrt{r}}}f\cdot T^{\floor{\sqrt{r}}^{3}+\floor{({\sqrt{r}})^{4/5}} }g.
 \end{equation*}Then, we can prove that \begin{equation*}
     \limsup\limits_{R\to+\infty}\bignorm{\underset{1\leq r\leq R}{\E} T^{\floor{\sqrt{r}}}f\cdot T^{\floor{\sqrt{r}}^{3}+\floor{({\sqrt{r}})^{4/5}} }g}_{L^2(\mu)}\leq   C \limsup\limits_{R\to+\infty}\bignorm{\underset{1\leq r\leq R}{\E} T^{r}f\cdot T^{r^{3}+\floor{r^{4/5}} }g}_{L^2(\mu)}
 \end{equation*}for some positive real number $C$. Now the functions in the iterates are sub-linear functions and polynomials, which we are now able to handle (this is the content of Section \ref{sublinearsection}).
 \end{remark*}

 \subsection*{Step 3: Dealing with the sub-linear function.}
 
 In this step we show that the quantity in \eqref{afterstep2} goes to 0, if we take $R\to +\infty$ and then $M\to+\infty$. While steps 1 and 2 of this example correspond to parts of the proof in Sections \ref{PETsection} and \ref{reductionestimates}, this step corresponds to the proofs in Section \ref{sublinearsection}.
 
  We observe that the function $\log^3(r)$ in the iterates is a sub-linear function. We will show that \begin{equation}\label{expl}
    \lim\limits_{R\to+\infty} \underset{1\leq r\leq R}{\E} \Big|\int (\bar g \cdot T^{2m_1m_2}g) \cdot T^{\floor{\log^3 r}}(\bar{f}\cdot T^{2m_1m_2}f)   d\m  \Big|\ll \nnorm{\bar f\cdot T^{2m_1m_2} f}_3.
 \end{equation} In addition, the implicit constants do not depend on $m_1,m_2$.
 Assuming that \eqref{expl} holds, we take the limit as $M\to +\infty$ (this can be done because all implied asymptotic constants do not depend on $m_1,m_2$) and we need to show that \begin{equation*}
     \lim\limits_{M\to +\infty} \underset{|m_1|,|m_2|\leq M}{\E} \nnorm{       \bar{f}\cdot T^{2m_1m_2}f   }_3=0.
 \end{equation*}Applying the H\"{o}lder inequality, we are left with showing that \begin{equation*}
     \lim\limits_{M\to +\infty} \underset{|m_1|,|m_2|\leq M}{\E} \nnorm{       \bar{f}\cdot T^{2m_1m_2}f   }_3^8=0.
 \end{equation*}
 Using the definition of the Host-Kra seminorms, this relation reduces to an ergodic average with polynomial iterates, which is well known to converge to $0$ under our hypothesis on the function $f$ (namely, that $\nnorm{f_1}_s=0$ for some suitable $s\in \N$).
 
 We now establish \eqref{expl}. It suffices to show that \begin{equation*}
      \lim\limits_{R\to+\infty} \underset{1\leq r\leq R}{\E} \Big|\int  g\cdot   T^{\floor{\log^3 r}}f   d\m  \Big|\ll \nnorm{f}_{3,T}
 \end{equation*}for any 1-bounded functions $f \text{ and }g$, where the implied constant is absolute. We square the above expression and apply the Cauchy-Schwarz inequality to bound it by \begin{equation*}
      \underset{1\leq r\leq R}{\E} \int {G} \cdot S^{\floor{\log^3(r)}}F \ d(\m\times\m),
 \end{equation*}where $F:=\overline{f}\otimes f$, $G:=\overline{g}\otimes g$ and $S:=T\times T$.
 Then, \eqref{expl} follows if we show \begin{equation*}
     \bignorm{ \underset{1\leq r\leq R}{\E}  S^{\floor{\log^3(r)}}F }_{L^2(\m\times \m)}\ll  \nnorm{f}_{3,T}^2.
 \end{equation*}We use Lemma \ref{mainlemma} once more: it suffices to show that \begin{equation*}
     \limsup\limits_{r\to+\infty} \bignorm{\underset{r\leq n\leq r+L(r)}{\E}  S^{\floor{\log^3(n)}}F }_{L^2(\m\times \m)}\ll  \nnorm{f}_{3,T}^2\ ,
 \end{equation*}where $L(t)\in\mathcal{H}$ is sub-linear. Using the Taylor expansion, we can write \begin{equation*}
     \log^3(r+h)=\log^3(r)+\frac{3\log^2 r}{r}h-\frac{6\log x_h-3\log^2 x_h  }{2x_h^2}h^2,
 \end{equation*}where $0\leq h\leq L(r)$ and $x_h\in [r,r+h]$. If we choose the function $L(t)$ so that \begin{equation*}
     \frac{t}{\log^2 t}\prec L(t)\prec \frac{t}{\log t}, 
 \end{equation*}we can then deduce that the last term in the above expansion is $o_r(1)$. Our problem reduces to \begin{equation*}
     \limsup\limits_{r\to+\infty} \bignorm{\underset{0\leq h\leq L(r)}{\E}  S^{\floor{\log^3(r)+ \frac{3\log^2 r}{r}h }}F}_{L^2(\m\times \m)}\ll  \nnorm{f}_{3,T}^2.
 \end{equation*}
 We have again reduced our problem to finding a bound for an ergodic average with (variable) polynomials. In order to finish the proof, we work similarly as in the previous steps, using the change of variables trick and one application of the van der Corput inequality (we also need to use the inequality $\nnorm{ F}_{2,T\times T} \leq \nnorm{f}_{3,T}^2$).

 b) In this second example we describe the strategy that will be used in the special case that we discussed above, that is when our functions are sums of sublinear functions and polynomials. This case is covered in full generality in Section \ref{sublinearsection}. We consider the triplet of functions in $\mathcal{H}$ $(t+\log^3 t, t, \log^2 t) $ and we shall show that there exists $s\in \N$ so that, if $\nnorm{f}_s=0$, then \begin{equation*}
      \underset{1\leq n\leq N}{\E} T^{\floor{n+\log^3 n}}f\cdot T^{{n}} g_1\cdot  T^{\floor{\log^2 n}}g_2
 \end{equation*}converge to $0$ in mean ($g_1,g_2$ are again arbitrary 1-bounded functions).
 
 \subsection*{Step 1: Reducing to the case when all iterates have sub-linear growth.}
 We start by using Lemma \ref{mainlemma} to reduce our problem to \begin{equation}
   \limsup\limits_{R\to+\infty}\underset{1\leq r\leq R}{\E} \bignorm{ \underset{r\leq n\leq r+L(r)}{\E}   T^{\floor{n+\log^3 n}}f\cdot T^{{n}} g_1\cdot  T^{\floor{\log^2 n}}g_2         }_{L^2(\m)}^{2}=0
 \end{equation}for some sub-linear function $L(t)\in\mathcal{H}$. In this example, we will choose the function $L(t)$, so that \begin{equation*}
  \max_{r\leq n\leq r+L(r)}  |\log^3(n)-\log^3(r)|=o_r(1)   \  \ \ \ \text{ and }\  \ \ \max_{r\leq n\leq r+L(r)}   |\log^2(n)-\log^2(r)|=o_r(1).
 \end{equation*}
 For instance, the function $L(t)=\sqrt{t}$ can easily be checked to satisfy the above. Therefore, if $r$ is very large, we can write \begin{multline*}
     \bignorm{ \underset{r\leq n\leq r+L(r)}{\E}   T^{\floor{n+\log^3 n}}f\cdot T^{{n}} g_1\cdot T^{\floor{\log^2 n}}g_2 }_{L^2(\m)}=\\
     \bignorm{ \underset{r\leq n\leq r+L(r)}{\E}   T^{n+\floor{\log^3 r}+e_{1,n}}f\cdot\  T^{n} g_1\cdot\  T^{\floor{\log^2 r}+e_{2,n}}g_2         }_{L^2(\m)} ,
 \end{multline*}where $e_{1,n},e_{2,n}\in \{0,\pm 1\}$. We assume here that all the error terms are zero (in the main proof, we will invoke Lemma \ref{errors} to remove the error terms). Therefore, we want to show that \begin{equation*}
          \limsup\limits_{R\to+\infty}\underset{1\leq r\leq R}{\E}\bignorm{ \underset{r\leq n\leq r+L(r)}{\E}   T^{n+\floor{\log^3 r}}f\cdot\  T^{n} g_1\cdot \ T^{\floor{\log^2 r}}g_2         }_{L^2(\m)}^{2}=0.
 \end{equation*}Since $\norm{g_2}_{\infty}\leq 1$, we reduce our problem to \begin{equation*}
      \limsup\limits_{R\to+\infty}\underset{1\leq r\leq R}{\E}\bignorm{ \underset{r\leq n\leq r+L(r)}{\E}  T^n (g_1\cdot    T^{\floor{\log^3 r}}f   )}_{L^2(\m)}^{2}=0.
 \end{equation*}Note that the inner average is a polynomial average in the variable $n$. We fix a positive integer $M$ and use the van der Corput inequality to deduce that \begin{equation*}
     \bignorm{ \underset{r\leq n\leq r+L(r)}{\E}  T^n (g_1\cdot    T^{\floor{\log^3 r}}f   )}_{L^2(\m)}^{2}\ll \frac{1}{M} +\underset{|m|\leq M}{\E}\Big|\int \overline{(g_1\cdot  T^{\floor{\log^3(r)}})} \cdot  T^m (g_1\cdot  T^{\floor{\log^3(r)}}) \ d\m \Big| +o_r(1)\ ,
 \end{equation*}where the implied constant is absolute. Thus, we want to show that \begin{equation*}
     \frac{1}{M} +\underset{|m|\leq M}{\E}\ \underset{1\leq r\leq R}{\E} \Big|\int (\overline{g_1}\cdot T^m g_1) \cdot T^{\floor{\log^3(r)}}(\overline{f}\cdot T^m f) \ d\m \Big|+o_R(1)
 \end{equation*}goes to $0$, as $R\to+\infty$ and then as $M\to+\infty$.
 
 \subsection*{Step 2: Dealing with the sub-linear functions.}

 Our problem follows by taking the limit as $R\to +\infty$ and then using the bound \begin{equation}\label{lst}
     \limsup\limits_{R\to+\infty}\underset{1\leq r\leq R}{\E} \Big|\int (\overline{g_1}\cdot T^m g_1) \cdot T^{\floor{\log^3(r)}}(\overline{f}\cdot T^m f) \ d\m \Big|\ll \nnorm{\overline{f} \cdot T^m f}_{3,T}.
 \end{equation}This was established in the previous example. Using this relation and taking the limit $M\to+\infty$ (note that our asymptotic constants do not depend on $M$), we reach the conclusion. 

Since \eqref{lst} follows from the previous example, we will describe our arguments for a more representative case. We shall prove that \begin{equation}\label{eld}
\limsup\limits_{N\to+\infty}  \bignorm{  \underset{1\leq n\leq N}{\E} T^{\floor{\log^3 n+\log^2 n}}f\cdot T^{\floor{\log^3 n}} g_1\cdot T^{\floor{\log^2 n}}g_2 }_{L^2(\m)}\ll \nnorm{f}_4,
\end{equation}where the implied constant is absolute. 
Using Lemma \ref{mainlemma}, it suffices to show that \begin{equation*}
    \limsup\limits_{R\to+\infty}  \underset{1\leq r\leq R}{\E}\bignorm{  \underset{r\leq n\leq r+L(r)}{\E} T^{\floor{\log^3 n+\log^2 n}}f\cdot T^{\floor{\log^3 n}} g_1\cdot T^{\floor{\log^2 n}}g_2 }_{L^2(\m)}^2\ll \nnorm{f}_4^2
\end{equation*}for some sub-linear function $L(t)\in\mathcal{H}$. We choose $L(t)=t(\log t)^{-3/2} $. Using similar approximations as in the first example, we can show that for any $0\leq h\leq L(r)$ \begin{equation*}
    \log^3(r+h)=\log^3 r+h\frac{3\log^2 r}{r}+o_r(1),
\end{equation*}while  \begin{equation*}
   \log^2(r+h)=\log^2 r+o_r(1)
\end{equation*}for all $0\leq h \leq L(r)$. Disregarding the error terms $o_r(1)$ in this example, it suffices to show that \begin{equation*}
    \limsup\limits_{R\to+\infty} \underset{1\leq r\leq R}{\E} \bignorm{  \underset{0\leq h\leq L(r)}{\E} T^{\floor{\log^3 r +h\frac{3\log^2 r}{r}}}\big( T^{\floor{\log^2 r}}f\cdot g_1   \big)\cdot T^{\floor{\log^2 r}}g_2 }_{L^2(\m)}^2\ll \nnorm{f}_4^2.
\end{equation*}Since $g_2$ is bounded by 1, the above bound follows from \begin{equation*}
    \limsup\limits_{N\to+\infty}\underset{1\leq r\leq R}{\E}  \bignorm{  \underset{0\leq h\leq L(r)}{\E} T^{\floor{\log^3 r +h\frac{3\log^2 r}{r}}}\big( T^{\floor{\log^2 r}}f\cdot g_1   \big)}_{L^2(\m)}\ll \nnorm{f}_4^2.
\end{equation*} This is an average where the iterates are variable polynomials. Working similarly to the previous example, we can show that \begin{multline*}
    \bignorm{  \underset{0\leq h\leq L(r)}{\E} T^{\floor{\log^3 r +h\frac{3\log^2 r}{r}}}\big( T^{\floor{\log^2 r}}f\cdot g_1   \big)}_{L^2(\m)}^2\ll \\
    \frac{1}{M}+\underset{|m|\leq M}\E \Big| \int  \overline{\big(T^{\floor{\log^2 r}}f\cdot g_1\big)} \cdot T^m\big( T^{\floor{\log^2 r}}f\cdot g_1 \big)\  d\m  \Big|   +o_r(1).
\end{multline*}Thus, it suffices to show that \begin{equation*}
    \limsup\limits_{M\to+{\infty}} \underset{|m|\leq M}{\E} \limsup\limits_{R\to+\infty}\underset{1\leq r\leq R}{\E}\Big| \int (\bar{g_1}\cdot T^m g_1)\cdot T^{\floor{\log^2 r}} (\bar{f}\cdot T^mf) \ d\m\Big|\ll \nnorm{f}_4^2.
\end{equation*}Note that we started with three sub-linear functions in the iterates and now we have an average with only one sub-linear function (our argument in the general case is based on this induction scheme). 
The result follows by working similarly to step 3 in the previous example.

\section{Bounds of polynomial averages}\label{PETsection}

Our main goal in this section is to establish Proposition \ref{PET} below. Before stating that proposition, we will first give some definitions.

\subsection{Families of variable polynomials }

Assume we are given a family $P_N=\{p_{1,N},...,p_{k,N}\}$ of essentially distinct (i.e. their pairwise differences are non-constant polynomials) variable polynomials, such that the degrees of the polynomials in $P_N$ and of their pairwise differences are independent of $N$ (for $N$ large enough). Then, we can assign to $p_{1,N}$ its own vector $(v_{1,N},...,v_{k,N})$, where $v_{1,N}$ is the leading coefficient of $p_{1,N}$ and $v_{j,  N}$ is the leading coefficient of $p_{1,N}-p_{j,N}$ for $j\neq 1$. We symbolize this by $\mathcal{S}(p_{1,N})$ and call this the {\em leading vector} of the family $P_N$ corresponding to $p_{1,N}$. We similarly define $\mathcal{S}(p_{i,N})$ for every $i\in \{1,...,k\}$ and call it the {\em leading vector} corresponding to $p_{i,N}$. Let us remark that the leading vector has no elements equal to 0, because we have assumed that the polynomials are essentially distinct. Finally, we call $P_N$ {\em ordered}, if the degrees of the polynomials $p_{i,N}$ are non-increasing. In this case, the polynomial $p_{1,N}$ has maximal degree and we call it the {\em leading polynomial}. The {\em leading vector} of an ordered polynomial family is defined as the leading vector corresponding to its leading polynomial.

\subsection{Types of polynomial families}

We define the {\em type} $(d,w_d,...,w_1)$ of the polynomial family, where $d$ is the largest degree appearing in the polynomials of $P_N$ and $w_i$ is the number of distinct leading coefficients of the members of $P_N$ with degree exactly $i$ among all polynomials in the family. Note that for families of variable polynomials, the value of this vector may depend on the variable $N$. We order the types by the value of $d$ and then order types of same degree lexicographically. We observe that a decreasing sequence of types must eventually be constant. The type of a family is a classical quantity used in the literature when an induction scheme on polynomial families is required.

\subsection{Good sequences and nice polynomial families}

Now, we define the notion of a nice polynomial family. Namely, we will deal with polynomials whose coefficients are well-behaved sequences. Our arguments fail to work in the general case where the coefficients can be arbitrary sequences.
 \begin{definition}\label{good sequence}
 a) A sequence $(a_n)_{n\in\N}$ of real numbers is called "good", if there exists a function $f\in\mathcal{H}$ with $\lim\limits_{t\to+\infty}f(t)\neq 0$  such that \begin{equation*}
     \lim\limits_{n\to+\infty} \frac{a_n}{f(n)}=1.
 \end{equation*}
b) Let $P_N=\{p_{1,N},...,p_{k,N}\}$ be a collection of polynomials. The family $P_N$ is called nice, if all the degrees of the polynomials $p_{i,N}$ and $p_{i,N}-p_{j,N}$ are independent of $N$ for $N$ large enough and their leading coefficients are good sequences, for all admissible values of the $i,j$.\\
\end{definition}

Note that any good sequence has a limit (possibly infinite). An example of a good sequence that is not a Hardy sequence is the sequence $\frac{\floor{N^{2/3}}}{\sqrt{N}}$, which is asymptotically close to $N^{1/6}$. In particular, all sequences of the form $\floor{f(n)}$, where the function $f\in\mathcal{H}$ does not converge to $0$ (as $t\to+\infty$), are good sequences, while, for example, $\floor{\frac{1}{\log n}}$ is not a good sequence. 

\begin{lemma}\label{type}
The type of a nice polynomial family is well-defined (independent of $N$) for $N$ large enough.
\end{lemma}
\begin{proof}
 This is fairly straightforward. Indeed, assume that the polynomials $p_{i,N}$ and $p_{j,N}$ of the given family have the same degree $s$. Let $a_i(N),a_{j}(N), a_{ij}(N)$ be the leading coefficients of $p_{i,N},p_{j,N}$ and $p_{i,N}-p_{j,N}$, which are all good sequences. The degree of the polynomial $p_{i,N}-p_{j,N}$ does not depend on $N$ for $N$ sufficiently large. Then, we have either one of the following:\\
 i) If the polynomial $p_{i,N}-p_{j,N}$ has degree equal to $s$, then for $N$ large enough, $a_{ij}(N)=a_i(N)-a_j(N)\neq 0$ and therefore the polynomials $p_{i,N},p_{j,N}$ have distinct leading coefficients eventually.\\
 ii) If the polynomial $p_{i,N}-p_{j,N}$ has degree smaller than $s$, then that means that, for $N$ large enough, we have $a_{i}(N)-a_j(N)= 0$ and the polynomials $p_{i,N},p_{j,N}$ have equal leading coefficients eventually.\\
 The claim easily follows.
\end{proof}
 
\subsection{The van der Corput inequality}
 We shall rely heavily on the following variant of the van der Corput inequality in our proofs.

\begin{lemma}
For a sequence $u_n$ in a Hilbert space with $\norm{u_n}\leq 1$ and a quantity $M=o(N)$, we have\begin{equation*}
\bignorm{\frac{1}{N}\sum_{n=0}^{N-1} u_n}^{2^d}\ll_d \frac{1}{M} + \underset{-M\leq m\leq M}{\E}\  \Big|\underset{0\leq n\leq N-1}{\E}\langle u_{n+m},u_{n}\rangle\Big|^{2^{d-1}}+o_N(1).
\end{equation*}

\end{lemma}
\begin{proof}
This follows from the basic van der Corput inequality \begin{equation*}
    \bignorm{\frac{1}{N}\sum_{n=0}^{N-1} u_n}\ll  \frac{1}{M^{1/2}} +\big(\underset{-M\leq m\leq M}{\E} \Big|\underset{0\leq n\leq N-1}{\E}\langle u_{n+m},u_{n}\rangle\Big|\big)^{1/2}+\frac{M^{1/2}}{N^{1/2}}
\end{equation*}
by successively squaring and applying the Cauchy-Schwarz inequality.
\end{proof}

We will use this inequality to derive asymptotic bounds for multiple ergodic averages involving polynomials. The above inequality holds, in particular, when $M$ is a fixed positive integer.
We state here the equivalent result for variable sequences, since this is more consistent with the notation used in the proof below.
\begin{lemma}\label{vdc}
For sequences $(u_{n,N})_{n,N\in \N}$ in a Hilbert space with $\norm{u_{n,N}}\leq 1$ and a quantity $M=o(N)$, we have\begin{equation*}
    \bignorm{\frac{1}{N}\sum_{n=0}^{N-1} u_{n,N}}^{2^d}\ll_d \frac{1}{M} + \underset{ |m|\leq M}{\E}\  \Big|\underset{0\leq n\leq N-1}{\E}\langle u_{n+m,N},u_{n,N}\rangle\Big|^{2^{d-1}}+o_N(1).
\end{equation*}
\end{lemma}

\subsection{Bounds of polynomial averages}

The remainder of the section will be dedicated to establishing the following proposition:

\begin{proposition}\label{PET}
Let $k, d$ be positive integers and let $M$ be a positive integer parameter. Suppose $\bW=(d,w_d,...,w_1)$ is a (d+1)-tuple of positive integers that is also a type for some polynomial family. Then, there exist positive integers $t=t(d,k,\bW) $, $ s=s(d,k,\bW)$, a finite set $Y=Y(d,k,\bW)$ of integers and integer polynomials in $t$ variables $p_{\underline{\e},j},\text{ with } \underline{\e}\in [[s]]$ and $1\leq j\leq \ k$, that are at most linear in each variable\footnote{This means that when regarded as polynomials only in one variable, then they are linear. Examples are $p_1(m_1,m_2)=m_1-2m_2$ and $p_2(m_1,m_2,m_3)=m_1m_2-3m_3$.}, such that for any ordered nice family of non-constant, essentially distinct polynomials \begin{equation*}
    P_N=\{p_{1,N},...,p_{k,N}\}
\end{equation*}of degree $d$ and type $\bW$ with leading vector $\mathcal{S} (P_N) =\{u_{1,N},...,u_{k,N}\}$, any increasing sequence $L_N\to \infty $, any measure preserving system $(X,\m,T)$ and sequences of 1-bounded functions $f_{1,N},...,f_{k,N}$, we have \begin{multline}\label{asdfghjkl}
    \sup_{|c_{n,N}| \leq 1}\bignorm{ \underset{0\leq n\leq L_N}{\E}\  c_{n,N} \prod_{i=1}^{k} T^{\floor{p_{i,N}(n)}}f_{i,N}}_{L^2(\m)}^{2^t}\ll_{d,k,\bW} \\ 
     \frac{1}{M}+
  \sum_{{\bf h}\in     Y^{[[s]]}}^{} \underset{{\bf m}\in [-M,M]^t}{\E} \Big| \int \prod_{\underline{\e}\in [[s]]}^{}  T^{\floor{A_{\ue,N}(\bm)}+ h_{\underline{\e}}}(\mathcal{C}^{|\underline{\e}|}f_{1,N})  \ d\m  \Big|+
    o_N(1),
\end{multline}where \begin{equation*}
    A_{\ue,N}(\bm)=\sum_{1\leq j\leq k} \ p_{\underline{\e},j}({\bf m})u_{j,N}
\end{equation*}are real polynomials in $\bm$. In addition, we have the following:\\
i) For $\ue\neq \underline{0}$, we have that the polynomial $A_{\ue,N}(\bm)$ is non-constant.\\
ii) The polynomials $A_{\ue,N}(\bm),\ \ue\in [[s]]$ are pairwise essentially distinct.\\
iii) We have the relation $$A_{\ue,N}(\bm)+A_{\ue^c,N}(\bm)=A_{\underline{1},N}(\bm)$$ for any $\ue\in[[s]]$. More generally, if $\ue,\ue'\in [[s]]$ are such that $\ue+\ue'\in [[s]]$ \footnote{This means that the $i$-th entries of $\ue,\ue'$ cannot simultaneously be 1, for any admissible value of $i$.}, then \begin{equation*}
    A_{\ue,N}(\bm)+A_{\ue',N}(\bm) =A_{\ue+\ue',N}(\bm)
\end{equation*}
iv) For any $\ue \in [[s]]$, we have that if \begin{equation*}
    c_1p_{\ue,1}(\bm)+...+c_k p_{\ue,k}(\bm)
\end{equation*}is the zero polynomial for some $c_1,...,c_k\in\R$, then we have $c_i=0$ or $p_{\ue,i}(\bm)$ is the zero polynomial, for every $1\leq i\leq k$.

\end{proposition}
\begin{comment*}
The $\sum\limits_{{\bf h}\in Y^{[[s]]}}^{}$ means that we take the sum for all choices of ${\bf h}=(h_{\underline{\e}},\ \underline{\e} \in [[s]])$ where $h_{\underline{\e}}\in Y$. In addition, we will make a small abuse of notation and write $\underset{\bm\in[-M,M]^t}{\E}$ to denote the average over all $ \bm\in \Z^t\cap [-M,M]^t$.
\end{comment*}
\begin{remarks*}

i) The polynomials $p_{\underline{\e},j}$ are independent of the leading vector $\{u_{1,N},...,u_{k,N}\}$ and are, more importantly, independent of the variable $N$.\\
ii) The existence of the errors $h_{\underline{\e}}$ is merely technical and arises from the floor function in the last expression inside the integral, since we cannot use Lemma \ref{errors} to remove the error terms in this case. This will be more easily understood in the proof of the case of linear polynomials that follows. \\
iii) The quantity $o_N(1)$ depends of course on the values of $d $ and $k$. It also depends on the value of the fixed number $M$. However, this dependence plays no role in arguments of the following sections (where we will usually take limits first as $N\to+\infty$ and, then, as $M\to+\infty$). For ease of notation, we will omit all other subscripts for the term $o_N(1)$.\\
iv) The final condition $iv)$ above implies that, for a fixed $\ue\in[[s]]$, if we exclude all the constant polynomials among the $p_{\ue,j}$, the remaining polynomials are linearly independent.

\end{remarks*}
Ignoring the technical parts of the statement, the above proposition asserts that when working with multiple averages on some polynomials that vary with $N$, we can instead bound them by the averages of a polynomial correlation sequence of only the function $f_{1,N}$. Even though the new polynomials $A_{\ue,N}$ have several variables, they only depend on the sequences $u_{1,N},...,u_{k,N}$ and, assuming we they have good limiting behavior, we can take the limits first as $N\to+\infty$ and then as $M\to+\infty$ to get some nice bounds for the original averages. For instance, in the case where we have a fixed function $f_{1,N}=f_1$ and the sequences $u_{i,N}$ converge to non-zero real numbers, the above statement can be used to prove that the $\limsup$ of the ergodic averages in the left-hand side of \eqref{asdfghjkl} can be bounded by a power of $\nnorm{f_{1}}_s$ for some suitable positive integer $s$.
This last assertion follows from minor modifications to the argument present in \cite{Leibmanseveral} (to cover the case of real polynomials instead of just integer polynomials).

\begin{proof}[Proof in the linear case] Firstly, we shall establish Proposition \ref{PET} in the case where all the polynomials have degree 1. Thus,
 assume that $p_{i,N}(t)= a_{i,N}t+b_{i,N}$ where $a_{i,N}, b_{i,N}\in \R$ so that the variables $a_{i,N}$ are (eventually) non-zero. The assumption that our polynomials are essentially distinct implies that the  numbers $a_{i,N}$ and $a_{j,N}$ are distinct. The leading vector of $P_{N}$ is the set \begin{equation*}
  \{a_{1,N}, a_{1,N}-a_{2,N},...,a_{1,N}-a_{k,N}\}
\end{equation*}and these are good sequences.

We induct on $k$. For $k=1$, we apply the van der Corput inequality to get \begin{multline*}
    \bignorm{ \underset{0\leq n\leq L_N}{\E}\  c_{n,N}  T^{\floor{a_{1,N}n +b_{1,N}}}f_{1,N}}_{L^2(\m)}^2 \ll \\
 \frac{1}{M}+   \underset{ |m|\leq M}{\E}   \Big|  \underset{0\leq n\leq L_N}{\E} \overline{c_{n,N}}c_{n+m,N}  \int \overline{f_{1,N}} \cdot T^{\floor{a_{1,N} n+b_{1,N} +m a_{1,N}}-\floor{a_{1,N}n+b_{1,N}}} f_{1,N} \ d\m \Big|   +o_N(1).
    \end{multline*}We rewrite the last quantity as \begin{equation*}
        \frac{1}{M}+  \underset{|m|\leq M}{\E}   \Big|  \underset{0\leq n\leq L_N}{\E} \overline{c_{n,N}}c_{n+m,N}  \int \overline{f_{1,N}} \cdot T^{\floor{ma_{1,N}}+e_{n,m,N}} f_{1,N} \ d\m \Big|  +o_N(1),
    \end{equation*}
where $e_{n,m,N}\in \{0,\pm 1\}$ (the implied constant is independent of all variables in the above relation). Let $A_{z,m,N}=\{n\in\Z^{+}\colon 0\leq n\leq L_N \text{ and } e_{n,m,N}=z \}$ for $z\in\{0,\pm 1\}=Y$. Then, the innermost average can be rewritten as \begin{multline*}
    \Big|\frac{1}{L_N} \sum_{z\in Y}^{} \sum_{n\in A_{z,N,m}} \overline{c_{n,N}}c_{n+m,N}  \int \overline{f_{1,N}} \cdot T^{\floor{ma_{1,N}}+z} f_{1,N} \ d\m|
     \leq 
   \sum_{z\in Y}^{} \Big|\int \overline{f_{1,N}} \cdot T^{\floor{ma_{1,N}}+z} f_{1,N}\ d\m\Big|,
\end{multline*}which, combined with the above, gives the desired result (for constants $t=1$ and $s=1$, polynomials $p_1(m)=ma_{1,N}$ and $p_0(m)=0$ and set $Y=\{0,\pm 1\}$).

Now assume that we have proven the result for $k-1$ ($k\geq 2$), with the constants of the proposition given by $t=k-1$ and $s=k-1$. Then, we use the van der Corput inequality to get \begin{multline*}
     \bignorm{ \underset{0\leq n\leq L_N}{\E}\  c_{n,N} \prod_{i=1}^{k} T^{\floor{a_{i,N}n +b_{i,N}}}f_{i,N}}_{L^2(\m)}^{2^k}\ll_k  \frac{1}{M}+o_N(1)+\\
     \underset{| m|\leq M}{\E}   \Big|  \underset{0\leq n\leq L_N}{\E} \overline{c_{n,N}}c_{n+m,N} \int \prod_{i=1}^{k}    T^{\floor{a_{i,N}n+b_{i,N}} +\floor{m a_{i,N}}+e_{i,m,n,N}} f_{1,N}\  T^{\floor{a_{i,N}n+b_{i,N}}}\overline{f_{i,N}}  \ d\m \Big|^{2^{k-1}}, \end{multline*} 
     which is smaller than \begin{multline}\label{23}
     \underset{|m|\leq M}{\E}   \Big|  \underset{0\leq n\leq L_N}{\E}\  \overline{c_{n,N}}c_{n+m,N}\  \int \prod_{i=1}^{k}    T^{\floor{a_{i,N}n+b_{i,N}}-\floor{a_{k,N}n+b_{k,N}} +\floor{m a_{i,N}}+e_{i,m,n,N}} f_{1,N}\cdot \\ T^{\floor{a_{i,N}n+b_{i,N}}-\floor{a_{k,N}n+b_{k,N}}}\overline{f_{i,N}}  \ d\m \Big|^{2^{k-1}} +
     1/M +o_N(1),
\end{multline}where we again have $e_{i,m,,n,N}\in\{0,\pm 1\}$. In the last step, we composed with $T^{-\floor{a_{k,N}n+b_{k,N}}}$ inside the integral.

We have \begin{equation*}
    \floor{a_{i,N}n+b_{i,N}}-\floor{a_{k,N}n+b_{k,N}} =\floor{{(a_{i,N}-a_{k,N})n+b_{i,N}-b_{k,N}}}+e'_{i,n,N}
\end{equation*}where $e'_{i,n,N}\in\{0,\pm 1\}$. Therefore, we can rewrite the last expression in \eqref{23} as \begin{multline*}
\frac{1}{M}+
         \underset{|m|\leq M}{\E} \  \Big|  \underset{0\leq n\leq L_N}{\E}\  \overline{c_{n,N}}c_{n+m,N} \\ \prod_{i=1}^{k} \int   T^{\floor{(a_{i,N}-a_{k,N})n+b_{i,N}-b_{k,N}}+e'_{i,n,N}}
         \big( \overline{f_{i,N}} \cdot T^{\floor{ma_{i,N}}+e_{i,m,n,N}}f_{i,N}  \big)  \ d\m \Big|^{2^{k-1}}+o_N(1).
\end{multline*}

Then, using the Cauchy-Schwarz inequality and the argument in Lemma \ref{errors}, we can bound the innermost average in the above expression by $O_k(1)$ times the quantity \begin{equation*}
    A_{m,N}=\sup_{|c_{n,N}|\leq 1} \bignorm{ \underset{0\leq n\leq L_N}{\E}\ c_{n,N}\prod_{i=1}^{k-1}\   T^{\floor{(a_{i,N}-a_{k,N})n+(b_{i,N}-b_{k,N})}}(\overline{f_{i,N}}\cdot T^{\floor{m a_{i,N}}+e_{i,m,n,N} }f_{i,N})                                }_{L^2(\m)}^{2^{k-1}}.
\end{equation*}Now, we use the argument of Lemma \ref{errors} again to deduce that $A_{m,N}$ is bounded by $O_k(1)$ times \begin{equation*}
    \sum_{\underset{1\leq i\leq k-1}{z_i\in\{0,\pm 1\}} }^{} \sup_{|c_{n,N}|\leq 1} \bignorm{ \underset{0\leq n\leq L_N}{\E}\ c_{n,N}\prod_{i=1}^{k-1}\   T^{\floor{(a_{i,N}-a_{k,N})n+(b_{i,N}-b_{k,N})}}(\overline{f_{i,N}}\cdot T^{\floor{m a_{i,N}}+z_i }f_{i,N})                                }_{L^2(\m)}^{2^{k-1}}.
\end{equation*}

We fix some ${\bf z}=(z_1,...,z_{k-1})\in \{0,\pm 1\}^{k-1}$.
 If we take the polynomial that corresponds to $\overline{f_{1,N}}\cdot T^{\floor{m a_{1,N}}+z_1}f_{1,N}$ to be the new leading polynomial, then the new leading vector is the set \begin{equation*}
    \{a_{1,N}-a_{k,N},a_{1,N}-a_{2,N},...,a_{1,N}-a_{k-1,N}\}. 
\end{equation*}
By the induction hypothesis, there exists a finite set $Y_{k-1}$, for which\begin{multline*}
  \sup_{|c_{n,N}|\leq 1} \bignorm{ \underset{0\leq n\leq L_N}{\E}\ c_{n,N}\prod_{i=1}^{k-1}\   T^{\floor{(a_{i,N}-a_{k,N})n+(b_{i,N}-b_{k,N})}}(\overline{f_{i,N}}\cdot T^{\floor{m a_{i,N}}+z_i }f_{i,N})                                }_{L^2(\m)}^{2^{k-1}}  \ll_k   \\    \frac{1}{M} +
\sum_{{\bf h}\in [[Y_{k-1}]]}^{} \underset{|m_1|,...,|m_{k-1}|\leq M}{\E}    \Big| \int \prod_{\underline{\e}\in [[k-1]]}^{}  T^{\floor{\sum_{1\leq j\leq k-1} \ p_{\underline{\e},j}(m_1,...,m_{k-1})(a_{1,N}-a_{j,N})}+h_{\underline{\e}}}\\
   \mathcal{C}^{|\underline{\e}|}(\overline{f_{1,N}}\cdot T^{\floor{m a_{1,N}}+z_1}f_{1,N})
   \ d\m      \Big|+o_N(1).
\end{multline*}
Using the identification $[[k]]=[[k-1]]\times \{0,1\}$, we can write an $\underline{\e}\in[[k]]$ as $\underline{\e}=(\underline{\e_1},\e_2)$ where $\underline{\e_1}\in [[k-1]]$ and $\e_2\in \{0,1\}$. We also write $\bm=(m,m_1,...,m_{k-1})$. Combining the integer parts, we rewrite the last integral as \begin{equation*}
    \int \prod_{\e\in [[k]]]} T^{\floor{ \sum_{1\leq j\leq k-1} \ p'_{\underline{\e},j}(m_1,...,m_{k-1})(a_{1,N}-a_{j,N})+ p'_{\underline{\e},k}(m)a_{1,N}} +h'_{\underline{\e},\bm}}\  \mathcal{C}^{|\underline{\e}|}f_{1,N} \ d \m\  ,
\end{equation*}
where \begin{enumerate}
    \item $p'_{\underline{\e},j}$ is the polynomial $p_{\underline{\e}_1,j}$ for $1\leq j \leq k-1$,
    \item the polynomial $p'_{\underline{\e},k}$ is equal to $m$ when $\underline{\e_2}=0$  and is zero otherwise and
    \item $h'_{\underline{\e},\bm}=h_{\underline{\e}_1}+h_{2,\underline{\e},\bm}$, where \footnote{ In particular, $h_{2,\ue,\bm}$ is the sum of $z_1$ plus the error term appearing by combining $\floor{ma_{1,N}}$ with the other integer part, whenever they both appear. Otherwise, it is zero. Thus, it takes values on a finite set of integers.} $h_{2,\underline{\e},\bm}\in \{0,\pm 1,\pm 2\}$. More importantly, $h'_{\underline{\e},m}$ takes values in a finite set $Y_k$.
\end{enumerate}We observe that \begin{multline*}
    \Big|\int \prod_{\e\in [[k]]]} T^{\floor{ \sum_{1\leq j\leq k-1} \ p'_{\underline{\e},j}(m_1,...,m_{k-1})(a_{1,N}-a_{j,N})+ p'_{\underline{\e},k}(m)a_{1,N}} +h'_{\underline{\e},\bm}}\  \mathcal{C}^{|\underline{\e}|}f_{1,N} \ d \m \Big|\leq \\
    \sum_{{\bf h}\in [[Y_k]]}^{} \Big|\int \prod_{\e\in [[k]]]} T^{\floor{ \sum_{1\leq j\leq k-1} \ p'_{\underline{\e},j}(m_1,...,m_{k-1})(a_{1,N}-a_{j,N})+ p'_{\underline{\e},k}(m)a_{1,N}} +h_{\underline{\e}}}\  \mathcal{C}^{|\underline{\e}|}f_{1,N} \ d \m \Big|.
\end{multline*}

Averaging over $m,m_1,...,m_{k-1}$ and summing over ${\bf z}\in \{0,\pm 1\}^{k-1}$, we have that for the finite set $Y_k$ above, the original expression is bounded by $O_k(1)$ times \begin{equation*}
    \frac{1}{M}+
   \sum_{{\bf h}\in [[Y_k]]}^{} \ \underset{{\bf m}\in [-M,M]^k}{\E}
  \Big| \int \prod_{\underline{\e}\in [[k]]}^{}  T^{\floor{\sum_{1\leq j\leq k} \ p'_{\underline{\e},j}({\bf m})u_{j,N}}+h_{\underline{\e}}}\ (\mathcal{C}^{|\underline{\e}|}f_{1,N}) \ d \m       \Big|+
    o_N(1),
\end{equation*}where $u_{1,N}=a_{1,N}$ and $u_{j,N}=a_{1,N}-a_{j,N}$.
The conclusion follows.
\end{proof}

\begin{remark*}
It follows from the above proof that the polynomials $A_{\ue,N}$ in the statement of Proposition \ref{PET}  have the following form:
\begin{equation*}
    A_{\ue,N}(m_1,...,m_k)= \ue \cdot (m_1u_{1,N},...,m_k u_{k,N})
\end{equation*}where "$\cdot$" denotes here the standard inner product on $\R^k$.
Thus, it is straightforward to check that the polynomials $A_{\ue,N}$  satisfy the conditions $i),ii)$, $iii)$ and $iv)$ of Proposition \ref{PET}. Note that all these polynomials have degree 1. This will not be the case when working with polynomials of higher degree, where we may have higher degree terms (like products of the form $m_1m_2$), but they will be linear in each variable separately. 
\end{remark*}

\subsection{The PET induction.}

For a polynomial $p_N$, a family $P_N$ and $h\in\N$, we define the {\em van der Corput} operation (or vdC operation), where we form the family \begin{equation*}
    \{ p_{1,N}(t+h)-p_{N}(t),...,p_{k,N}(t+h)-p_N(t),\    p_{1,N}(t)-p_{N}(t),...,p_{k,N}(t)-p_N(t)                                    \}
\end{equation*}and remove polynomials of degree 0. We denote this new family by $(p_N,h)^{*}P_N$. At first glance, it is not obvious that this operation is well defined, because the constant polynomials that we discard may be different for different values of $N$. We will see that this is not the case for nice polynomial families below.
We will use the vdC operation successively to reduce the "complexity" of a polynomial family. Our main observation is that the leading vector of a polynomial family is well behaved under the vdC operation.

 Consider a family of variable polynomials  $P_N=\{p_{1,N},...,p_{k,N}\}$ and let the leading vector of $P_{N}$ corresponding to $p_{1,N}$ be \begin{equation*}
    \mathcal{S}(P_{N}) =\{u_{1,N},...,u_{k,N}\}.
\end{equation*}Fix any $1\leq i_0\leq k$, as well as the polynomial $p_{i_0,N}$, which we symbolize as $p_N$ from now on for convenience. Consider the new polynomial family $P'_{N,h}=(p_N,h)^{\star} P_N$ that arises from the van der Corput operation. Here, $h$ ranges over the non-zero integers.

\begin{lemma}\label{form}
Assume that the family $P_N$ of degree $d$ is nice and let $(u_{1,N},...,u_{k,N})$ be its leading vector corresponding to $p_{1,N}$. For every choice of polynomial $p_N$ above and the value of $h\in\Z^{*}$, we have that each element of the leading vector of  $P'_{N,h}$ corresponding to the new polynomial $p_{1,N}(t+h)-p_N(t)$ has one the following forms:\begin{itemize}
    \item They are equal to one of the $u_{i,N}$ for some $2\leq i \leq k$.
    \item They have the form $d u_{1,N}h$.
    \item They are the sum $d u_{1,N}h+u_{i,N}$ for some $u_{i,N}$ with $i\neq 1$. 
\end{itemize}
\end{lemma}

\begin{proof}
 Without loss of generality, we will assume that we have taken $p_N=p_{k,N}$ (the case $p_N=p_{1,N}$ is very similar). We want to study the leading vector corresponding to the polynomial $p_{1,N}(t+h)-p_{k,N}(t)$. Therefore, it is sufficient to find the leading coefficients of the polynomials \begin{align*}
     \big(p_{1,N}(t+h)-p_{k,N}(t)\big) &-\big(p_{1,N}(t)-p_{k,N}(t) \big)\\
     \big(p_{1,N}(t+h)-p_{k,N}(t)\big)&-\big( p_{i,N}(t+h)-p_{k,N}(t)  \big)\\
      \big(p_{1,N}(t+h)-p_{k,N}(t)\big)&-\big( p_{i,N}(t)-p_{k,N}(t)  \big)
 \end{align*}for $2\leq i\leq k$. The leading coefficient of the first polynomial is always $dhu_{1,N}$ and that satisfies our required property. The leading coefficient of the second polynomial is always equal to the leading coefficient of $p_{1,N}(t+h)-p_{i,N}(t+h)$ and this is always equal to the leading coefficient of $p_{1,N}(t)-p_{i,N}(t)$ which belongs to the leading vector. Finally, the leading coefficient of the third polynomial is equal to the leading coefficient of $p_{1,N}(t+h)-p_{i,N}(t)$. Note that this polynomial can be rewritten as $$(p_{1,N}(t+h)-p_{1,N}(t)) +(p_{1,N}(t)-p_{i,N}(t)).$$ The leading coefficient of the first polynomial is equal to $dhu_{1,N}$ as we established above, while the second difference has leading coefficient $u_{i,N}$ (by definition). Therefore, the leading coefficient of their sum is either $dhu_{1,N}, u_{i,N}$ or their sum $dhu_{1,N}+u_{i,N}$, which concludes the proof.
\end{proof}
 
 Observe that the particular form each element of the new leading vector has does not depend on the value of $N$ (i.e. it cannot have the first form for one value of $N$ and then the second form for some other value of $N$). This follows from the fact that the type of the original family is independent of $N$, if $N$ is large enough. We will now use this lemma to study how the van der Corput operation affects the type of the original family.
 
\begin{corollary}\label{typecorollary}
Let $P_N,p_N$ be as above and let $d$ be the degree of the family $P_N$. Then, there exists a set of integers $Y$ with at most $O_{k,d}(1)$ elements such that, for every $h\notin Y$, the polynomial family $P'_{N,h}=(p_N,h)^{\star} P_N$ that arises from the van der Corput operation is nice and its type is independent\footnote{ The type depends only on which polynomial of the initial family we choose to be the polynomial $p_N$, as well as the type of the original family.} of the value of $h$. 
\end{corollary}

\begin{proof}
  We denote by $u_{ii,N}$ the leading coefficient of $p_{i,N}$, while $u_{ij,N}$ denotes the leading coefficient of $p_{i,N}-p_{j,N}$ for $i\neq j$. These are all good sequences by the definition of a nice family. Using Lemma \ref{form}, we can prove that the leading coefficients of all the polynomials in $P'_{N,h}$ and of their differences can take one of the following forms:\\
 i) they are equal to some $u_{ij,N}$ with $i\neq j$,\\
 ii) they have the form $ru_{ii,N}h$ for some $1\leq r\leq d$ or\\
 iii) they have the form $ru_{ii,N}h+u_{ij,N}$ for some $1\leq r\leq d$.
 
 We prove that these sequences are good for all except $O_{d,k}(1)$ values of $h$.
 For all values of $1\leq i,j\leq k$ ($i\neq j$) and $1\leq r\leq d$, we consider the set $A(i,j,r)$ of all possible sequences of the above three forms (not all of them appear as leading coefficients, but this does not affect our argument), where $h$ is some fixed non-zero integer. There are only finitely many such sets. Note that for $h\neq 0$, the sequences of the first two forms are always good. Now consider a sequence of the form $ru_{ii,N}h+u_{ij,N}$. There exist functions $f_1,f_2\in\mathcal{H}$, not converging to 0, such that $|u_{ii,N}/f_1(N)|=1+o_N(1)$ and $|u_{ij,N}/f_2(N)|=1+o_N(1)$.
 The function $rhf_1(t)+f_2(t)$ is obviously an element of $\mathcal{H}$. In addition, for our fixed $r$, the relation\begin{equation*}
     \lim\limits_{t\to+\infty }rhf_1(t)+f_2(t)=0
 \end{equation*}can hold only for at most one possible value of $h\in\Z$, which we call a "bad value". Then, if $h$ is not a bad value, we have \begin{equation*}
     \Big|\frac{ru_{ii,N}h+u_{ij,N}}{rhf_1(N)+f_2(N)}\Big|=1+o_N(1).
 \end{equation*}Indeed, this follows easily because the functions $f_1$ and $f_2$ are comparable, which also means that all the involved sequences are comparable. Thus, dividing the numerator and denominator of the above fraction by either $f_1(N)$ or $f_2(N)$, we easily get the result. In conclusion, the sequence $ru_{ii,N}h+u_{ij,N}$ is a good sequence for all non-bad values of $h$.
 
  Now, if we take all possible values of the $i,j,r$, we conclude that there are at most $O_{d,k}(1)$ bad values of $h$.

 We have shown that for every non-bad value of $h$, the family $P'_{N,h}$ is a nice polynomial family and, therefore, has a fixed type (independent of $N$). We show that its type does not depend on $h$. Therefore, consider two polynomials $q_1,q_2$ of $P'_{N,h}$ of the same degree. We consider some possible cases:\\
  a) If $q_1$ and $q_2$ have the form $p_{i,N}(t)-p_N(t)$, then whether or not their leading coefficients are equal depends only on the type of the original family and the choice of $p_N$ (and not on $h$).\\
 b) If $q_1$ has the form $p_{i,N}(t+h)-p_N(t)$, while $q_2$ has the form $p_{j,N}(t)-p_N(t)$, then their leading coefficients can be equal in only two possible cases: if the polynomial $p_N$ has degree strictly larger than the degree of both $p_{i,N}$ and $p_{j,N}$ (this depends only on the choice of $p_N$, not on $h$), or if the polynomials $p_{i,N}(t+h)$ and $p_{j,N}(t)$ have the same degree (bigger than or equal to the degree of $p_N$) and equal leading coefficients. In the second case, we must have that $p_{i,N}(t)$ and $p_{j,N}(t)$ have equal leading coefficients, which depends only on the type of the original family and not on $h$.\\
 c) If $q_1$ and $q_2$ both have the form $p_{i,N}(t+h)-p_N(t)$, then the result follows similarly as in the case a).

 The fact that the degrees of the polynomials of the new family and of their differences do not depend on $N$ and $h$ can also be established easily using the preceding arguments. We omit the details.
\end{proof}

\begin{proposition}\label{induction}
If $P_N=\{p_{1,N},...,p_{k,N}\}$ is an ordered polynomial family, then there exists a polynomial $p_N\in P_N$, such that for all, except at most one value of $h\in \Z$, the polynomial family $P'_{N,h}=(p_N,h)^{*} P_N$ has type strictly smaller than the type of $P_N$ and its leading polynomial is the polynomial $p_{1,N}(t+h)-p_N(t)$.
\end{proposition}

\begin{proof}
We describe the operation that reduces the type. At each step, we choose a polynomial $p_{N}\in P_N$ that has minimal degree in the family. For an $h\in\Z$, apply the van der Corput operation. This forms a polynomial family \begin{equation}\label{P}
    P'_N=\{ p_{1,N}(t+h)-p_{N}(t),...,p_{k,N}(t+h)-p_N(t),\    p_{1,N}(t)-p_{N}(t),...,p_{k,N}(t)-p_N(t)                                    \}
\end{equation} and choose $p_{1,N}(t+h)-p_{N}(t)$ to be the new leading polynomial.
We distinguish between some cases:

a) Assume that the polynomials $p_{1,N}$ and $p_{k,N}$ have distinct degrees. Then, choose $p_N=p_{k,N}$, which by the "ordered" assumption has minimal degree.
We notice that the polynomial $p_{1,N}(t+h)-p_N(t)$ has maximal degree in the polynomial family. We check that the type of the polynomial family is reduced. Indeed, if the degree of $p_{k,N}(t)$ is $d'$, then the number $w_d'$ is reduced, while all the numbers $w_i$ are left unchanged for $i>d'.$
    
b) Suppose the polynomials $p_{1,N}$ and $p_{k,N}$ have the same degree and not all leading coefficients in the family $P_N$ are equal. In particular, we may assume, without loss of generality, that this holds for the polynomials $p_{1,N}$ and $p_{k,N}$. Again, choose $p_{N}=p_{k,N}$. Then, the polynomial $p_{1,N}(t+h)-p_{N}(t)$ has maximal degree in the new polynomial family. In addition, the number $w_{d}$ is reduced, which means that the new family has smaller type than the original.

c) Finally, assume that all polynomials have the same degree and the same leading coefficient. We choose again $p_N=p_{k,N}$. The polynomial $p_{1,N}(t+h)-p_{N}(t)$ has maximal degree equal to $d-1$ in $P'_N$, except possibly for one value of $h\in \Z$ (to see this, we can work similarly as in the proof of Corollary \ref{typecorollary}).  Also, the family $P'_N$ has smaller type than $P_N$, since it has degree at most $d-1$.
\end{proof}

While for a given type $\bW$ there are infinite types smaller than $\bW$, it is straightforward to see that a decreasing sequence of types is eventually constant. Therefore, the type-reducing operation that we did above will eventually terminate to a type of degree 1, namely we will reduce our problem to the linear case, which we have already established.
To summarize all of the above, we have the following:

\begin{corollary}\label{all in one}
Let $P_N$ be a nice polynomial family of degree $d$, with $k$ polynomials and with type $\bW$. Then, there exists a $p_N\in P_N$, such that the family $P'_N=(p_N,h)^{*}P_N$ is nice and has (fixed) type smaller than $\bW$ for all, except at most $O_{d,k}(1)$ values of $h$.

\end{corollary}

\begin{definition}\label{non-degenerate}
We will call a van der Corput operation $(p_N,h)^{*}P_N$ {\em non-degenerate}, if the polynomial $p_N\in P_N$ is such, that the conditions of Corollary \ref{all in one} hold.
\end{definition}
Namely, the polynomial $p_N$ must be chosen, so that the resulting family has type independent of $N,h$, provided that $N$ is sufficiently large and $h$ takes values outside a set of at most $O_{d,k}(1)$ elements (here, this notation refers to the same asymptotic constant appearing in the statement of Corollary \ref{all in one}). In view of the above corollary, we deduce that there always exists a non-degenerate van der Corput operation. We will denote a non-degenerate van der Corput operation simply by $(p_N)^{*}P_N$ to indicate the independence on the parameter $h$.

We are now ready to finish the proof of Proposition \ref{PET}:
\begin{proof}[Proof of the higher degree case]

First of all, we shall explain how we will choose the parameters $t,s$. These depend crucially on how the van der Corput operations are used (and there are possibly many ways in which the successive van der Corput can be carried out), which may lead to ambiguity.

   Let  ${\bf W}=(d,w_d,...,w_{1})$ be the type of the given polynomial family. We say that a triplet $(d',k',\bW')$ can be reached by the triplet $(d,k,\bW)$ if there exists a sequence of non-degenerate van der Corput operations that produces the families $$P_{1,N}=(q_{1,N})^{*}P_N,...,P_{\ell,N}=(q_{\ell,N})^{*} P_{\ell-1,N}, $$ where the family $P_{\ell,N}$ consists of $k'$ polynomials, has degree $d'$ and type $\bW'$.
  
  Observe that the triplets that can be reached by the original triplet $(d,k,\bW)$ are finitely many in number, since there are only finitely many choices (depending on $d,k,\bW$) for each polynomial $q_{i,N}$ at each step. In particular, they all have degrees at most $d$, types strictly smaller than $\bW$ and the number $k'$ can be bounded by a function of $(d,k,\bW)$, since each van der Corput operation at most doubles the number of polynomials in a family and this operation can occur finitely many times as well. We also remark that we have already established our claim for all polynomial families of degree $d=1$ (this will work as the base case of our induction). 
 
  Let $S_{d,k,\bW}$ be the set of triplets that can be possibly reached by $(d,k,\bW)$, which is a finite set. We will use induction by considering that our claim holds for all triplets in $S_{d,k,\bW}$ and we will show that the claim holds for families corresponding to our original family $P_N$ that corresponds to the triplet $(d,k,\bW)$.
  
 Fix such a triplet $(d',k',\bW')$ and define $t(d',k',\bW'),s(d',k',\bW')$ to be the numbers appearing in the statement of Proposition \ref{PET}. Namely, if the nice ordered family $$Q_N=\{q_{1,N},...,q_{k',N}\}$$ has degree $d'$ and type $\bW'$, then   
  \begin{multline}\label{asdfghjkl}
    \sup_{|c_{n,N}| \leq 1}\bignorm{ \underset{0\leq n\leq L_N}{\E}\  c_{n,N} \prod_{i=1}^{k'} T^{\floor{q_{i,N}(n)}}f_{i,N}}_{L^2(\m)}^{2^{t(d',k',\bW')}}\ll_{d',k',\bW'} \\ 
     \frac{1}{M}+
  \sum_{{\bf h}\in     Y^{[[s(d',k',\bW')]]}}^{} \underset{{\bf m}\in [-M,M]^{t(d',k',\bW')}}{\E} \Big| \int \prod_{\underline{\e}\in [[s(d',k',\bW')]]}^{}  T^{\floor{A_{\ue,N}(\bm)}+ h_{\underline{\e}}}(\mathcal{C}^{|\underline{\e}|}f_{1,N})  \ d\m  \Big|+
    o_N(1),
\end{multline}where we are being vague on the dependence of the polynomials $A_{\ue,N}$ on the parameters $(d',k',\bW')$ and the family $Q_N$ in this relation, since this will not concern us temporarily. 
  
  The number $t(d',k',\bW')$ is the number of times we apply the van der Corput inequality in order to bound the left-hand side by the quantity on the right-hand side. Now, we define $$t_0=\max\limits_{(d',k',\bW')\in S_{d,k,\bW}} t(d',k',\bW'),$$ which, of course, is a parameter that depends only on $(d,k,\bW)$. Assume that the number $t_0$ corresponds to a family $Q_N$. Then, it is obvious that $Q_N$ can be reached by the original family $P_N$ in only one step. Indeed, if there was a another family in the sequence of van der Corput operations starting from $P_N$ to $Q_N$, then this family would have a strictly larger parameter $t( \cdot )$ associated to it than $t_0$. 
  
  Assume that the family $Q_N$ has the the triplet $(d',k',\bW')$ associated to it.  Although the parameter $t_0$ is well defined (and depends only on $(d,k,\bW)$), the parameter $s(d',k',\bW')$ may not be, because there may be another family $Q'_N$ which has the same value $t_0$ for the first parameter, but different for the second. In this case, we simply take $Q_N$ to be the one for which the parameter $s(d',k',\bW)$ is also maximized (denote this simply by $s$ from this point onward). Obviously, we have that $s$ depends only on $(d,k,\bW)$.

  For the family $Q_N$ constructed above, we can write $Q_N=(p_N)^{*}P_N$ for some $p_N\in P_N$. Without loss of generality, assume that $p_N=p_{k,N}$ (the case where $p_N=p_{1,N}$ is similar).

   We apply the van der Corput inequality to get \begin{multline}\label{25}
      \bignorm{ \underset{0\leq h\leq L_N}{\E}\  c_{n,N} \prod_{i=1}^{k} T^{\floor{p_{i,N}(n)}}f_{i,N}}_{L^2(\m)}^{2^{t_0+1}} \ll_{t_0}  \\ \frac{1}{M} +
      \underset{|m|\leq M}{\E}\Big|\underset{0\leq n\leq L_N}{\E}\ c_{n+m,N}\overline{c_{n,N}} \int  \prod_{i=1}^{k}\ T^{\floor{p_{i,N}(n+m)}}f_{i,N}\cdot        T^{\floor{p_{i,N}(n)}}\overline{f_{i,N}}\ d\m  \Big|^{2^{t_0 }} +o_N(1).
  \end{multline} 
  We compose with $T^{-\floor{p_{N}(n)}}$ in the above integral, so that \begin{align*}
      &\underset{0\leq n\leq L_N}{\E}\ c_{n+m,N}\overline{c_{n,N}}\int  \prod_{i=1}^{k}\ T^{\floor{p_{i,N}(n+m)}}f_{i,N}\cdot        T^{\floor{p_{i,N}(n)}}\overline{f_{i,N}}\ d  \m =\\
      &\underset{0\leq n\leq L_N}{\E}\ c_{n+m,N}\overline{c_{n,N}}\int \prod_{i=1}^{k}\ T^{\floor{p_{i,N}(n+m)}-\floor{p_N(n)}}f_{i,N}\cdot        T^{\floor{p_{i,N}(n)}-\floor{p_N(n)}{}}\overline{f_{i,N}}\ d\m=\\
      &\underset{0\leq n\leq L_N}{\E} c_{n+m,N}\overline{c_{n,N}}\int \prod_{i=1}^{k}\ T^{\floor{p_{i,N}(n+m)-p_N(n)}+e_{1,n,i,m,N}}f_{i,N}\cdot        T^{\floor{p_{i,N}(n)-p_N(n)}+e_{2,n,i,N}}\overline{f_{i,N}}\ d\m,
  \end{align*}where the numbers $e_{1,n,i,m,N} \text{ and }  e_{2,n,i,N}$ take values in the set $\{0,\pm 1\}$. We use the Cauchy-Schwarz inequality and then use Lemma \ref{errors} to bound the absolute value of the last quantity by a constant (depending only on $k$) multiple of the expression \begin{multline*}
      \sup_{|c_{n,N}|\leq 1}\bignorm{ \underset{0\leq n\leq L_N }{\E} c_{n,N} \big(\prod_{i=1}^{k-1} T^{\floor{p_{i,N}(n+m)-p_N(n)}}f'_{i,N}\cdot        T^{\floor{p_{i,N}(n)-p_N(n)}}\overline{f'_{i,N}}\big)  \\
      T^{\floor{p_{k,N}(n+m)-p_{k,N}(n)}} f'_{k,N}    }_{L^2{(\m)}}+o_N(1)
  \end{multline*}for some 1-bounded functions $f'_{1,N}=f_{1,N}, f'_{2,N},...,f'_{k,N}$. Recall that we chose $p_N=p_{k,N}$.
  The family of polynomials \begin{equation*}
      P'_{N,m}=\{ p_{1,N}(t+m)-p_{k,N}(t),...,p_{k,N}(t+m)-p_{k,N}(t),\    p_{1,N}(t)-p_{k,N}(t),...,p_{k-1,N}(t)-p_{k,N}(t)                                    \}
  \end{equation*}is nice and has (fixed) type $\bW'<\bW $ independent of $m$ for all, except at most $O_{d,k}(1)$ values of $m\in\N$ (it has the same triplet $(d',k',\bW')$ of parameters as the family $Q_N$ above). Let $Q$ be this finite set of "bad" values of $m$ and let \begin{equation*}
      \mathcal{S}(P'_{N,m})=\{u'_{1,m,N},...,u'_{k',m,N}\} 
  \end{equation*} be the leading vector of $P'_{N,m}$, where $k'\leq  2k-1$. For all $m\notin Q$, we use the induction hypothesis to deduce that
  \begin{multline}\label{uiop}
      \sup_{|c_{n,N}|\leq 1}\bignorm{\underset{0\leq n\leq L_N }{\E} c_{n,N} \big(\prod_{i=1}^{k-1} T^{\floor{p_{i,N}(n+m)-p_N(n)}}f'_{i,N}\cdot        T^{\floor{p_{i,N}(n)-p_N(n)}}\overline{f'_{i,N}}\big)\cdot \\  
      \  \ \ \  \ \ \ \ \ \ \ \ \  \ \ \ \  \  \ \ \ \ \ \ \ \  \ \ \  \  \ \ \ \  \ \ \  \ \ \  \  \  \ \ \ \ T^{\floor{p_{k,N}(t+m)-p_{k,N}(t)}} f'_{k,N}    }_{L^2{(\m)}}^{2^{t_0}} 
      \ll_{k,d,\bW'} \\  \frac{1}{M}+\ 
  \sum_{{\bf h}\in [[Y_0]]}^{} \underset{(m_1,...,m_t)\in [-M,M]^t}{\E}
  \Big| \int \prod_{\underline{\e}\in [[s]]}^{}  T^{\floor{\sum_{1\leq j\leq k'} \ p_{\underline{\e},j}(m_1,...,m_t)u'_{j,m,N}}+h_{\underline{\e}}}(\mathcal{C}^{|\underline{\e}|}f_{1,N})  \  d\m    \Big|
  +  o_N(1)
  \end{multline}for a finite set $Y_0$ that depends only on $d', \bW' \text{ and } k'$ (i.e. $d,\bW$ and $k$). We will now set the parameter $t=t(d,k,\bW)$ to be simply $t_0+1$.

  We also observe that our induction imposes that the polynomials \begin{equation*}
A_{\underline{\e},N,m}(m_1,...,m_t)  =    \sum_{1\leq j\leq k'} \ p_{\underline{\e},j}(m_1,...,m_t)u'_{j,m,N}
  \end{equation*}are non-constant and pairwise essentially distinct for any (non-zero) values of the leading vector $\{u'_{1,m,N},..., u'_{k',m,N}\}$ and that all the polynomials $p_{\underline{\e},j}$ are at most linear in each variable. In addition, we claim that \begin{equation}\label{e^c}
      A_{\underline{\e},N,m}(m_1,...,m_t)  +A_{\underline{\e'},N,m}(m_1,...,m_t)=A_{\ue+\ue',N}(m_1,...,m_t)\  \text{whenever } \ \ue+\ue'\in [[s]].
  \end{equation} (we have seen that all of the above are true in the linear case). These are the properties i)-iii) in Proposition \ref{PET}.

  All the $u'_{j,m,N}$ have the form described by Lemma \ref{form}. Therefore, we can write \begin{equation}\label{essential}
      p_{\underline{\e},j}(m_1,...,m_t)u'_{j,m,N}=p'_{1,\underline{\e},\ell}(m,m_1,...,m_t)u_{\ell ,N} +p'_{2,\underline{\e},\ell'}(m_1,...,m_t)u_{\ell',N}.
  \end{equation}In order to describe the form of the new polynomials $p'_{1,\underline{\e},\ell},p'_{2,\underline{\e},\ell'}$, we split into cases depending on the form of $u'_{j,m,N}$ (cf. Lemma \ref{form}):\\
  a) If $u'_{j,m,N}$ is equal to some $u_{\ell,N}$ for $1\leq \ell\leq k$, then we have $p'_{2,\underline{\e},\ell'}=0$ and $$p'_{1,\underline{\e},\ell}(m,m_1,...,m_t)=p_{\underline{\e},j}(m_1,...,m_t)$$ (thus $p'_{1,\underline{\e},\ell}(m_1,...,m_t)$ is constant as a polynomial in $m$).\\
  b) If  $u'_{j,m,N}$ is equal to $dmu_{1,N}$ ($\ell=1$), then we have again $p'_{2,\underline{\e},\ell'}=0$ and \begin{equation*}
      p'_{1,\underline{\e},1}(m,m_1,...,m_t)=dm p_{\underline{\e},j}(m_1,...,m_t).
  \end{equation*}
  c) In the final case that $ u'_{j,m,N}=dmu_{1,N} +u_{\ell',N} $ for some $\ell'\neq 1$, then we have $p'_{2,\underline{\e},\ell'}=p_{\underline{\e},j}(m_1,...,m_t)$ and 
  \begin{equation*}
      p'_{1,\underline{\e},1}(m,m_1,...,m_t)=dm p_{\underline{\e},j}(m_1,...,m_t).
  \end{equation*}

  Therefore, the new polynomials $p_{1,\underline{\e},\ell}$ and $p_{2,\underline{\e},\ell'}$ are at most linear in each of the variables $m_1,...,m_t$, as well as the new variable $m$. By grouping the terms corresponding to the same $u_{\ell,N}$, we can rewrite \begin{equation*}
      \sum_{1\leq r\leq k'}^{} \ p_{\underline{\e},r}(m_1,...,m_t)u'_{r,m,N} =\sum_{1\leq r\leq k}^{}  q_{\underline{\e},r}(m,m_1,....,m_t)u_{r,N} 
      \end{equation*}
     for some new polynomials $q_{\underline{\e},r}$. 
     
      \begin{claim}
        The new polynomials $\sum_{1\leq r\leq k}^{}  q_{\underline{\e},r}(m,m_1,....,m_t)u_{r,N} $ satisfy conditions i), ii), iii) and iv) of Proposition \ref{PET}, for any (non-zero) values of the $u_{r,N}$.
      \end{claim}

\begin{proof}[Proof of the Claim]The fact that they are non-constant is trivial, since otherwise one of  the polynomials \begin{equation*}
     \sum_{1\leq r\leq k'}^{} \ p_{\underline{\e},r}(m_1,...,m_t)u'_{r,m,N}
\end{equation*}would be constant, which is at odds with the induction hypothesis. Assume that condition ii) fails for two $\underline{\e_1}, \underline{\e_2} \in [[s]]$. Regarding these two polynomials as polynomials only in $  (m_1,...,m_t)$,  \eqref{essential} would give that the polynomials \begin{equation*}
          \sum_{1\leq r\leq k'}^{}  p_{\underline{\e_1},r}(m_1,...,m_t)u'_{r,m,N}\  \text{ and }\  \sum_{1\leq r\leq k'}^{} p_{\underline{\e_2},r}(m_1,...,m_t) u'_{r,m,N},
        \end{equation*}are not essentially distinct, which is false by the induction hypothesis. Therefore, we have established both i) and ii).

  Now, we want to prove an analogue of \eqref{e^c} for our new polynomials. But this follows trivially by \eqref{essential} (the new polynomials are just a rewritten form of the $A_{\ue,N}$). This establishes that the new polynomials satisfy condition iii) in the statement of Proposition \ref{PET}.
  
  Finally, we are going to prove that the new polynomials $q_{\ue,j}$ satisfy condition iv) of Proposition \ref{PET}. Fix an $\ue\in [[s]]$. We will assume that all $q_{\ue,j}$ are non-zero and we will show that they are linearly independent (if there are identically zero polynomials among the $q_{\ue,j}$, we proceed similarly by ignoring these polynomials). It suffices to show that if $a_1,...,a_k$ are real numbers, such that \begin{equation*}
      a_1q_{\underline{\e},1}(m,m_1,...,m_t)+...+a_kq_{\ue,k}(m,m_1,...,m_t)
  \end{equation*}is the zero polynomial, then all the numbers $a_i$ are zero. Recalling the form of the $q_{\ue, r}$, this becomes a linear combination of the form \begin{equation}\label{bigpoly}
      a_1P_{1,\ue}(m,m_1,...,m_t)+\sum_{i\in I_1}^{}b_i p_{\ue,i}(m_1,...,m_t)
  \end{equation}for some $I_1\subset \{1,2,...,k'\}$ and $b_i\in\{a_2,...,a_k\}$\footnote{Observe that each one of the numbers $a_2,..,a_k$ appears in the set $\{b_i,i\in I_1\}$ (maybe with multiplicity), because we have assumed that each polynomial $q_{\ue,\ell}, (\ell>1)$ is not the trivial polynomial (otherwise, we ignore it).}. In addition, the polynomial $P_{1,\ue}$ has the form \begin{equation*}
      dm\sum_{i\in I_2}^{}p_{\ue,i} +\sum_{i\in I_3}^{}p_{\ue,i}
  \end{equation*}for some $I_2,I_3\subset \{1,2,...,k'\} $ with $I_1\cap I_2=\emptyset$ and $I_1\cap I_3=\emptyset$. We argue by contradiction. For $m=0$, the polynomial in \eqref{bigpoly} must be identically zero and this easily yields that all the $b_i$ must be zero and that $a_1\sum_{i\in I_3}^{}p_{\ue,i}$ is also the zero polynomial. The first relation implies that $a_2=...=a_k=0$ by the induction hypothesis, while the second implies that either $a_1=0$ (in which case we are done), or $I_3=\emptyset$ (since the $p_{\ue,i}$ are linearly independent by the induction hypothesis). If $I_3=\emptyset$, then \eqref{bigpoly} implies that the polynomial \begin{equation*}
      a_1dm\sum_{i\in I_2}^{}p_{\ue,i}
  \end{equation*}is the zero polynomial. This implies that $a_1=0$ or $I_2=\emptyset.$ However, we cannot have $I_2=I_3=\emptyset$, because that would imply that the polynomial $q_{\ue,1}$ is identically zero, which is absurd (since we assumed that we have already discarded the zero polynomials among the $q_{\ue,i}$). Our claim follows.
        \end{proof}

 Combining all of the above we rewrite \eqref{uiop} as
  \begin{multline*}
      \sup_{|c_{n,N}|\leq 1}\bignorm{ \underset{0\leq n\leq L_N }{\E} c_{n,N} \big(\prod_{i=1}^{k-1} T^{\floor{p_{i,N}(n+m)-p_N(n)}}f_{i,N}\cdot        T^{\floor{p_{i,N}(n)-p_N(n)}}\overline{f_{i,N}}\big)\cdot \\ 
      \ \ \  \  \ \ \ \ \  \ \ \ \ \ \ \ \ \  \ \ \ \ \ \ \ \ \ \ \ \ \  \ \ \ \  \ \ \ \  \ \ \ \ \ \  \ \  \  \ \ \  \  \ T^{\floor{p_{k,N}(t+m)-p_{k,N}(t)}} f_{k,N}    }_{L^2{(\m)}}^{2^{t_0+1}} \ll_{d,k,\bW} \\
       \frac{1}{M}+
  \sum_{{\bf h}\in [[Y_0]]}^{} \underset{ |m_1|,...,|m_t|\leq M}{\E}\
  \Big| \int \prod_{\underline{\e}\in [[s]]}^{}  T^{\floor{\sum_{1\leq r\leq k}^{} q_{\underline{\e},r}(m,m_1,....,m_t)u_{r,N}}+ h_{\underline{\e}}} (\mathcal{C}^{|\underline{\e}|}f_{1,N})     d\m   \Big|+
    o_N(1).
  \end{multline*}
  We use the above bounds for all $-M\leq m\leq M$ in \eqref{25}. The possible error coming from the bad values of the set $Q$ can be absorbed by an $O_{d,k}({1}/ {M})$ term. Finally, we get \begin{multline*}
      \bignorm{ \underset{0\leq h\leq L_N}{\E} c_{n,N} \prod_{i=1}^{k} T^{\floor{p_{i,N}(n)}}f_{i,N}}_{L^2(\m)}^{2^{t}} \ll_{d,k,\bW} \\
     \frac{1}{M} +\ 
  \sum_{{\bf h}\in [[Y_0]]}^{} \underset{|m|,|m_1|,...,|m_t|\leq M}{\E}
  \Big| \int \prod_{\underline{\e}\in [[s]]}^{}  T^{\floor{\sum_{1\leq r\leq k}^{}  q_{\underline{\e},r}(m,m_1,....,m_t)u_{r,N}}+ h_{\underline{\e}}} (\mathcal{C}^{|\underline{\e}|}f_{1,N}) d\m       \Big|+
    o_N(1),
  \end{multline*} which is what we wanted to show.
    \end{proof}

\section{The sub-linear plus polynomial case}\label{sublinearsection}

In this section, we establish a particular case of Proposition \ref{factors}, which we shall also use in the general case in the next section. Let $\mathcal{S}$ denote the subset of $\mathcal{H}$ that contains the functions with sub-linear growth rate and $\mathcal{P}\subseteq \mathcal{H}$ denotes the collection of polynomials with real coefficients. Then, we let $\mathcal{S+P}$ denote the collection of functions that can be written as a sum of a function in $\mathcal{S}$ and a function in $\mathcal{P}$ (or equivalently, linear combinations of functions in $\mathcal{S}$ and $\mathcal{P}$).  

Let $a_1,...,a_k$ be a collection of functions in $\mathcal{S+P}$. Then, we can write $a_i=u_i+p_i$, where $u_i\in\mathcal{S}$ and $p_i$ is a polynomial. We will also define the {\em degree} and {\em type} of the collection $a_1,a_2,...,a_k$ using a similar notion to the degree and type of a polynomial family defined in the previous section. More precisely, since we do not impose that the polynomials $p_1,...,p_k$ are essentially distinct, we choose a maximal subset of the polynomials $p_i$ consisting of non-constant and essentially distinct polynomials and we define the degree and type of the collection $a_1,...,a_k$ to be the degree and type of this new subfamily of polynomials, respectively. Similarly, we define the leading vector of $a_1,...,a_k$ as the leading vector of the maximal subfamily that we defined above. We can always choose this maximal subset to contain the polynomial $p_1$. We define the cardinality of this new maximal subset to be the {\em size} of the collection $a_1,...,a_k$.

\begin{proposition}\label{sublinearseminorm}
Let $M$ be a positive integer and let $a_1,...,a_k$ be a collection of functions in $\mathcal{S+P}$ with degree $d$, type $\bW$ and size $k'\leq k$. Let $(c_1,...,c_{k'})$ be the leading vector of the family $\{a_1,...,a_k\}$. In addition, assume that $a_1(t)\succ \log t$ and $a_1(t)-a_j(t)\succ \log t$ for $j\neq 1$. Then, there exist positive integer $s,t$, a finite set $Y$ of integers and real polynomials $p_{\ue,j}$ in $t$ variables, where $\ue\in [[s]]$ and $1\leq j\leq k$, all depending only on $d,k',\bW$, such that, for any measure preserving system $(X,\m,T)$ and function $f_1\in L^{\infty}(\m)$ bounded by 1, we have \begin{multline}\label{61}
     \sup_{\norm{f_2}_{\infty},...,\norm{f_k}_{\infty}\leq 1 } \sup_{|c_n|\leq 1  }\bignorm{   \underset{1\leq n\leq N}{\E} c_n\ T^{\floor{a_1(n)}}f_1\cdot...\cdot T^{\floor{a_k(n)}}f_k          }_{L^2(\m)}^{2^t}\ll_{d,k,k',\bW} \\
    \frac{1}{M} +\sum_{{\bf h}\in Y^{[[s]]} }^{} \underset{m\in [-M,M]^t}{\E} \nnorm{  \prod_{\ue\in [[s]]}^{} T^{\floor{A_{\ue}(\bm)} +h_{\ue}}f_1 }_{2k+1}+o_N(1)
\end{multline}where \begin{equation*}
    A_{\ue}(\bm)=\sum_{j=1}^{k'} p_{\ue,j}(\bm)c_j.
\end{equation*} are pairwise essentially distinct polynomials.
\end{proposition}

 Observe that the iterates inside the seminorm in \eqref{61} are real polynomials in several variables. We can take $M\to+\infty$ and expand these seminorms to arrive at an iterated limit of polynomial averages. It is possible to bound these averages by a suitable seminorm of the function $f_1$ using the results in \cite{Leibmanseveral} and get a simpler bound in \eqref{61}. This necessitates that we substitute the $O_{d,k,k',\bW}(1)$ implicit constant by an $O_{a_1,...,a_k}(1)$ constant and this is insufficient for our purposes in the next section, where we will have to apply Proposition \ref{sublinearseminorm} for several collections of functions simultaneously. However, in view of the above discussion, we deduce the following:
 \begin{corollary}\label{sublinearcorollary}
 Let $a_1,...,a_k$ be a collection of functions in $\mathcal{S}+\mathcal{P}$ such that  $a_1(t)\succ \log t$ and $a_1(t)-a_j(t)\succ \log t$ for $j\neq 1$. Then, there exists a positive integer $s$ such that, for any measure preserving system $(X,\m,T)$ and 1-bounded function $f_1\perp Z_s(X)$, we have \begin{equation*}
      \lim\limits_{N\to+\infty} \ \sup_{\norm{f_2}_{\infty},...,\norm{f_k}_{\infty}\leq 1 } \ \sup_{|c_n|\leq 1  }\bignorm{   \underset{1\leq n\leq N}{\E} c_n\ T^{\floor{a_1(n)}}f_1\cdot...\cdot T^{\floor{a_k(n)}}f_k          }_{L^2(\m)}=0.
 \end{equation*}
 \end{corollary}

  We analyze the conditions imposed on the functions $a_1,...,a_k$ more closely: write each function $a_i$ in the form $a_i(t)=u_i(t)+p_i(t)$, where $u_i\in \mathcal{S}$ and $p_i\in \mathcal{P}$. The condition $a_1(t)\succ \log t $ implies that either $u_1(t)\succ \log t$ or $p_1(t)$ is a non-constant polynomial. Similarly, the second condition implies that either $u_1(t)-u_i(t)\succ \log t $ or $p_1(t)-p_j(t)$ is a non-constant polynomial. 

Furthermore, we can make one more reduction. Writing again $a_i(t)=u_i(t)+p_i(t)$ as above and using the same argument as in Section \ref{sectionfactors} (see the discussion following the statement of Proposition \ref{factors}), we may assume that the function $u_1$ has the largest growth rate among the functions $u_i$.

In order to establish the main result of this section, we will also use the following proposition, which is special case of Proposition \ref{sublinearseminorm}. 
\begin{proposition}\label{sublinearonly}
Let $ a_1,..., a_k$ be sub-linear functions in $\mathcal{H}$ and assume that all the functions $a_1,a_1-a_2,...,a_1-a_k$ dominate $\log t$. Then, for any measure preserving system $(X,\m,T)$ and function $f_1\in L^{\infty}(\m)$ bounded by 1, we have \begin{equation}\label{62}
    \limsup\limits_{N\to+\infty} \sup_{\norm{f_2}_{\infty},...,\norm{f_k}_{\infty}\leq 1 } \sup_{|c_n|\leq 1  }\ \bignorm{   \underset{1\leq n\leq N}{\E} c_n\  T^{\floor{a_1(n)}}f_1\cdot...\cdot T^{\floor{a_k(n)}}f_k          }_{L^2(\m)}\ll_k \  \nnorm{f_1}_{2k}.
\end{equation}
\end{proposition}
\begin{remark*}
The proof that Proposition \ref{sublinearonly} implies Proposition \ref{sublinearseminorm} corresponds to Step 1 in example b) of Section \ref{sectionfactors}, while the proof of Proposition \ref{sublinearonly} corresponds to step 2 of the same example.
\end{remark*}

\begin{proof}[Proof that Proposition \ref{sublinearonly} implies Proposition \ref{sublinearseminorm}]
First of all, we write each $a_i(t)$ in the form $u_i(t)+p_i(t)$ as we discussed above. Our main tool will be to use Lemma \ref{mainlemma} in order to reduce our problem to studying averages on small intervals, where the sublinear functions $u_i$ will have a constant integer part.

Suppose that not all of the polynomials $p_1(t),...,p_k(t)$ are constant, since the other case follows from Proposition \ref{sublinearonly} (that means the family has degree $\geq 1$). We can assume, without loss of generality, that $p_i(0)=0$ for all $i$ (the constant terms can be absorbed by the functions $u_i$).
Therefore, let $L(t)\in \mathcal{H}$ be a sub-linear function to be chosen later. In addition, we choose functions $f_{2,N},...,f_{k,N}$ so that the average in the left-hand side of \eqref{61} is $1/N$ close to the supremum. We want to bound \begin{equation*}
     \underset{1\leq r\leq R}{\E}\  \sup_{|c_{n,r}|\leq 1  }\ \bignorm{   \underset{r\leq n\leq r+L(r)}{\E}\ c_{n,r} \ T^{\floor{u_1(n)+p_{1}(n) }}f_1\cdot...\cdot T^{\floor{u_k(n)+p_k(n)}}f_{k,R}        }_{L^2(\m)}^{2^t}
\end{equation*}for some integer parameter $t$, which we will choose later to depend only on the quantities $d,k',\bW$ (thus, when applying Lemma \ref{errors} below to remove the error terms in the iterates, we will always have that the implicit constant depends only on $d,k',\bW$).

Recall that we have reduced our problem to the case that the function $u_1$ has the largest growth rate among the functions $u_i$.
Now, we want to choose the sub-linear function $L(t)\in \mathcal{H}$ so that the functions $u_i(n)$ restricted to the interval $[r,r+L(r)]$ become very close to the value $u_i(r)$. To achieve this, it suffices to take $L(t)\in\mathcal{H}$ such that  \begin{equation*}
   1\prec L(t)\prec (u_1'(t))^{-1}. 
\end{equation*}To see that such a function exists, we only need to show that $(u_1'(t))^{-1}\succ 1$ which follows easily from the fact that $u_1(t)\prec t$. Observe that for every $i\in\{1,2,...,k\}$ we must have $L(t)\prec (u_i'(t))^{-1}$, since $u_1$ has maximal growth among the functions $u_i$. For every $n\in [r,r+L(r)]$, we observe that \begin{equation*}
    |u_i(n)-u_i(r)|\leq (n-r)\max_{x\in [r,r+L(r)]}|u_i'(x)|.
\end{equation*}Since $|u_i'(t)|\searrow 0$, we have that for $r$ large enough \begin{equation*}
    |u_i(n)-u_i(r)|\leq L(r)u_i'(r)=o_r(1), \ \ \ \ \ \  n\in [r,r+L(r)].
\end{equation*}Therefore, for $r$ sufficiently large we have \begin{equation*}
    \floor{u_i(n)+p_i(n)}=\floor{u_i(r)}+\floor{p_i(n)}+e_{i,n}, \ \ \ \ \ \ \  n\in[r,r+L(r)],
\end{equation*}where $e_{i,n}\in\{0,\pm 1,\pm 2\}$. Therefore, our original problem reduces to bounding the quantity \begin{equation}\label{don}
      \underset{1\leq r\leq R}{\E}\  \sup_{|c_{n,r}|\leq 1  }\bignorm{   \underset{r\leq n\leq r+L(r)}{\E}\ c_{n,r}\ T^{\floor{u_1(r)}+\floor{p_{1}(n) }+e_{1,n}}f_1\cdot...\cdot T^{\floor{u_k(r)}+\floor{p_k(n)}+e_{k,n}}f_{k,R}        }_{L^2(\m)}^{2^t}.
\end{equation}Using Lemma \ref{errors}, we may reduce to the case that the error terms $e_{i,n}$ in the iterates are all equal to zero.
 
 Let $S$ be the set of those $i\in\{1,...,k\}$ for which the polynomial $p_i(t)$ is equal to the polynomial $p_1(t)$. Reordering, if necessary, we may assume that $S=\{1,...,k_0\}$ for some $k_0\leq k$. Note that the original condition then implies that $u_1(t)-u_i(t)\succ \log t$ for each $2\leq i\leq k_0$.
  We rewrite \eqref{don} as \begin{multline}\label{buf}
     \underset{1\leq r\leq R}{\E}\  \sup_{|c_{n,r}|\leq 1  }\bignorm{   \underset{r\leq n\leq r+L(r)}{\E}\ c_{n,r}\  
   T^{\floor{p_1(n)}} (\prod_{i=1}^{k_0} T^{\floor{u_i(r) }}f_{i,R} )
     \prod_{i=k_0+1}^{k} T^{\floor{u_i(r)}+\floor{p_i(n)}}f_{i,R}        }_{L^2(\m)}^{2^t}=\\
       \underset{1\leq r\leq R}{\E}\  \sup_{|c_{h,r}|\leq 1  }\bignorm{   \underset{0\leq h\leq L(r)}{\E}\ c_{h,r} \
   T^{\floor{p_1(r+h)}} (\prod_{i=1}^{k_0} T^{\floor{u_i(r) }}f_{i,R} )
     \prod_{i=k_0+1}^{k} T^{\floor{u_i(r)}+\floor{p_i(r+h)}}f_{i,R}        }_{L^2(\m)}^{2^t}\leq \\
 \underset{1\leq r\leq R}{\E}\  \sup_{\norm{f_{k_0+1}}_{\infty},...,\norm{f_{k}}_{\infty}\leq 1   }\    \sup_{|c_{h,r}|\leq 1  }\bignorm{   \underset{0\leq h\leq L(r)}{\E}\ c_{h,r} \
   T^{\floor{p_1(r+h)}} (\prod_{i=1}^{k_0} T^{\floor{u_i(r) }}f_{i,R} )
     \prod_{i=k_0+1}^{k} T^{\floor{p_i(r+h)}}f_{i}        }_{L^2(\m)}^{2^t},
       \end{multline}where $f_{1,R}=f_1$. We also write $F_{r,R}:= \prod_{i=1}^{k_0} T^{\floor{u_i(r) }}f_{i,R}$ for brevity.
       
       We can assume that the polynomials $p_i(r+h) $ are non-constant (otherwise, we just ignore the corresponding iterate in the last average). In addition, we may assume that they are pairwise essentially distinct, because if two polynomials are equal, we can combine both of these iterates into a single iterate (this operation does not change the type or leading vector of the given collection of functions).
       Note that under these assumptions the family of polynomials \begin{equation*}
          P_r= \{p_1(r+h),p_{k_0+1}(r+h),...,p_k(r+h) \}
       \end{equation*}is a nice family of polynomials\footnote{There is the possibility that the polynomial $p_{1}(n)$ is constant (and so is the polynomial $p_{1}(r+h)$) or that it does not have maximal degree (which would prevent the use of Proposition \ref{PET}, which was stated for ordered polynomial families). However, since we have assumed that not all of the polynomials $p_i$ are constant, then we can use the same argument as in Section \ref{sectionfactors} after Proposition \ref{factors} (where we reduced our problem to the case that the first function has maximal growth rate) to replace the polynomial $p_1(r+h)$ by $p_1(r+h)-p_i(r+h)$ for a non-constant polynomial $p_i(r+h)$ among $p_{k_0+1}(r+h),...,p_k(r+h)$. } in the variable $h$ (the leading coefficients of the polynomials and their pairwise differences are all constant sequences) and has type and leading vector equal to that of the original collection $\{p_1,...,p_k\}$. Therefore, we can apply Proposition \ref{PET}: there exist positive integers $t_0$ and $s$, a finite set $Y$ of integers and polynomials $p_{\ue,j}$ where $\ue\in [[s]]$ and $1\leq j\leq k$ such that \begin{multline}
       \sup_{\norm{f_{k_0+1}}_{\infty},...,\norm{f_{k}}_{\infty}\leq 1   }\  \sup_{|c_{h,r}|\leq 1  }  \bignorm{   \underset{0\leq h\leq L(r)}{\E}\ c_{h,r}\ 
   T^{\floor{p_1(r+h)}} F_{r,R}
     \prod_{i=k_0+1}^{k} T^{\floor{p_i(r+h)}}f_{i}        }_{L^2(\m)}^{2^{t_0}}\ll_{d,k',\bW}\\
      \frac{1}{M}+
  \sum_{{\bf h}\in     Y^{[[s]]}}^{} \underset{{\bf m}\in [-M,M]^{t_0}}{\E} \Big| \int \prod_{\underline{\e}\in [[s]]}^{}  T^{\floor{A_{\ue}(\bm)}+ h_{\underline{\e}}}(\mathcal{C}^{|\underline{\e}|}F_{r,R})  \ d\m  \Big|+
    o_r(1),
\end{multline}where \begin{equation*}
    A_{\ue}(\bm)=\sum_{1\leq j\leq k'} \ p_{\underline{\e},j}({\bf m})c_{j}
       \end{equation*}and $(c_1,...,c_{k'})$ is the leading vector of the initial family (here we have $k'\leq k-k_0+1$). 

Using this in \eqref{buf} with $t={t_0}$ (which depends only on $d,k',\bW$ as we claimed in the beginning), we deduce that our original average is bounded by $O_{d,k,k',\bW}(1)$ times \begin{equation*}
    \frac{1}{M} +\underset{1\leq r\leq R}{\E} \ \sum_{{\bf h}\in     Y^{[[s]]}}^{} \ \underset{{\bf m}\in [-M,M]^{t_0}}{\E} \Big| \int \prod_{\underline{\e}\in [[s]]}^{}  T^{\floor{A_{\ue}(\bm)}+ h_{\underline{\e}}}(\mathcal{C}^{|\underline{\e}|}F_{r,R})  \ d\m  \Big|+
    o_R(1).
\end{equation*}Using the definition of $F_{r,R}$, we rewrite this as \begin{multline*}
     \frac{1}{M} +\underset{1\leq r\leq R}{\E}\ \sum_{{\bf h}\in     Y^{[[s]]}}^{} \ \underset{{\bf m}\in [-M,M]^{t_0}}{\E} \Big| \int
     T^{\floor{u_1(r)}}\big(\prod_{\underline{\e}\in [[s]]}^{}  T^{\floor{A_{\ue}(\bm)}+ h_{\underline{\e}}}(\mathcal{C}^{|\underline{\e}|}f_{1})\big)\\
     \prod_{i=2}^{k_0} T^{\floor{u_i(r)}}\big(\prod_{\underline{\e}\in [[s]]}^{}  T^{\floor{A_{\ue}(\bm)}+ h_{\underline{\e}}}(\mathcal{C}^{|\underline{\e}|}f_{i,R})\big)
    \ d\m  \Big|+
    o_R(1).
\end{multline*}

Now, we consider two cases:\\

\underline{Case 1}: Firstly, assume that $k_0=1$. Then, the above quantity can be rewritten as \begin{multline*}
    \frac{1}{M} +\underset{1\leq r\leq R}{\E} \sum_{{\bf h}\in     Y^{[[s]]}}^{} \underset{{\bf m}\in [-M,M]^{t_0}}{\E} \Big| \int
     T^{\floor{u_1(r)}}(\prod_{\underline{\e}\in [[s]]}^{}  T^{\floor{A_{\ue}(\bm)}+ h_{\underline{\e}}}(\mathcal{C}^{|\underline{\e}|}f_{1}))
    \ d\m  \Big|+
    o_R(1)=\\
    \frac{1}{M} + \sum_{{\bf h}\in     Y^{[[s]]}}^{} \underset{{\bf m}\in [-M,M]^{t_0}}{\E} \Big| \int
     \prod_{\underline{\e}\in [[s]]}^{}  T^{\floor{A_{\ue}(\bm)}+ h_{\underline{\e}}}(\mathcal{C}^{|\underline{\e}|}f_{1}))
    \ d\m  \Big|+
    o_R(1).
\end{multline*} The result follows immediately, since \begin{equation*}
     \Big| \int
     \prod_{\underline{\e}\in [[s]]}^{}  T^{\floor{A_{\ue}(\bm)}+ h_{\underline{\e}}}(\mathcal{C}^{|\underline{\e}|}f_{1}))
    \ d\m  \Big|\leq \nnorm{
     \prod_{\underline{\e}\in [[s]]}^{}  T^{\floor{A_{\ue}(\bm)}+ h_{\underline{\e}}}(\mathcal{C}^{|\underline{\e}|}f_{1}))
    }_{2k+1}.
\end{equation*}

\underline{Case 2}: Assume that $k_0>1$ and we want to bound \begin{multline}\label{ssa}
    \frac{1}{M} +\underset{1\leq r\leq R}{\E} \sum_{{\bf h}\in     Y^{[[s]]}}^{} \underset{{\bf m}\in [-M,M]^{t_0}}{\E} \Big| \int
     T^{\floor{u_1(r)}}(\prod_{\underline{\e}\in [[s]]}^{}  T^{\floor{A_{\ue}(\bm)}+ h_{\underline{\e}}}(\mathcal{C}^{|\underline{\e}|}f_{1}))\\
     \prod_{i=2}^{k_0} T^{\floor{u_i(r)}}(\prod_{\underline{\e}\in [[s]]}^{}  T^{\floor{A_{\ue}(\bm)}+ h_{\underline{\e}}}(\mathcal{C}^{|\underline{\e}|}f_{i,R}))
    \ d\m  \Big|+
    o_R(1).
\end{multline}Our original hypothesis implies that the functions $u_1-u_i$ (where $2\leq i\leq k_0$) dominate $\log t$. Since $u_i$ was assumed in the beginning to have the biggest growth rate among the functions $u_i$, we must also have $u_1(t)\succ \log t$.

We take the limit as $R\to+\infty$ and rewrite the quantity in \eqref{ssa} as \begin{multline*}
    \frac{1}{M}+\sum_{{\bf h}\in     Y^{[[s]]}}^{} \underset{{\bf m}\in [-M,M]^t}{\E} \limsup\limits_{R\to+\infty} \underset{1\leq r\leq R}{\E}\Big| \int
     T^{\floor{u_1(r)}}(\prod_{\underline{\e}\in [[s]]}^{}  T^{\floor{A_{\ue}(\bm)}+ h_{\underline{\e}}}(\mathcal{C}^{|\underline{\e}|}f_{1}))\\
     \prod_{i=2}^{k_0} T^{\floor{u_i(r)}}(\prod_{\underline{\e}\in [[s]]}^{}  T^{\floor{A_{\ue}(\bm)}+ h_{\underline{\e}}}(\mathcal{C}^{|\underline{\e}|}f_{i,R}))
    \ d\m  \Big|.
\end{multline*}Applying the Cauchy-Schwarz inequality, we deduce that \begin{multline*}
     \underset{1\leq r\leq R}{\E}\Big| \int
     T^{\floor{u_1(r)}}(\prod_{\underline{\e}\in [[s]]}^{}  T^{\floor{A_{\ue}(\bm)}+ h_{\underline{\e}}}(\mathcal{C}^{|\underline{\e}|}f_{1}))
     \prod_{i=2}^{k_0} T^{\floor{u_i(r)}}(\prod_{\underline{\e}\in [[s]]}^{}  T^{\floor{A_{\ue}(\bm)}+ h_{\underline{\e}}}(\mathcal{C}^{|\underline{\e}|}f_{i,R}))
    \ d\m  \Big|\leq\\
    \Big( \underset{1\leq r\leq R}{\E}    \int
     S^{\floor{u_1(r)}}(\prod_{\underline{\e}\in [[s]]}^{}  S^{\floor{A_{\ue}(\bm)}+ h_{\underline{\e}}}(\mathcal{C}^{|\underline{\e}|}F_{1}))\\
     \prod_{i=2}^{k_0} S^{\floor{u_i(r)}}(\prod_{\underline{\e}\in [[s]]}^{}  S^{\floor{A_{\ue}(\bm)}+ h_{\underline{\e}}}(\mathcal{C}^{|\underline{\e}|}F_{i,R}))
    \ d(\m\times \m)                       \Big)^{1/2},
\end{multline*}where $S=T\times T$, $F_1=\overline{f_1}\otimes f_1$ and $F_{i,R}=\overline{f_{i,R}}\otimes f_{i,R}$. A final application of the Cauchy-Schwarz inequality bounds the last quantity by  \begin{multline*}
    \bignorm{\underset{1\leq r\leq R}{\E}   S^{\floor{u_1(r)}}(\prod_{\underline{\e}\in [[s]]}^{}  S^{\floor{A_{\ue}(\bm)}+ h_{\underline{\e}}}(\mathcal{C}^{|\underline{\e}|}F_{1}))
     \prod_{i=2}^{k_0} S^{\floor{u_i(r)}}(\prod_{\underline{\e}\in [[s]]}^{}  S^{\floor{A_{\ue}(\bm)}+ h_{\underline{\e}}}(\mathcal{C}^{|\underline{\e}|}F_{i,R}))  }_{L^2(\m\times \m)}^{1/2}.
\end{multline*}Applying Proposition \ref{sublinearonly}, we deduce that the $\limsup$ of this last average is bounded by $O_{k_0}(1)$ (which is $O_k(1)$) times \begin{equation*}
    \nnorm{\prod_{\underline{\e}\in [[s]]}^{}  S^{\floor{A_{\ue}(\bm)}+ h_{\underline{\e}}}(\mathcal{C}^{|\underline{\e}|}F_{1})}_{2k_0,T\times T}\leq \nnorm{\prod_{\underline{\e}\in [[s]]}^{}  S^{\floor{A_{\ue}(\bm)}+ h_{\underline{\e}}}(\mathcal{C}^{|\underline{\e}|}F_{1})}_{2k,T\times T}.
\end{equation*}Our original problem reduces to bounding \begin{equation*}
    \frac{1}{M}+\sum_{{\bf h}\in     Y^{[[s]]}}^{} \underset{{\bf m}\in [-M,M]^t}{\E} \nnorm{\prod_{\underline{\e}\in [[s]]}^{}  S^{\floor{A_{\ue}(\bm)}+ h_{\underline{\e}}}(\mathcal{C}^{|\underline{\e}|}F_{1})}_{2k,T\times T}^{1/2}+o_R(1),
\end{equation*}which is smaller than \begin{equation*}
      \frac{1}{M}+\sum_{{\bf h}\in     Y^{[[s]]}}^{} \underset{{\bf m}\in [-M,M]^t}{\E} \nnorm{\prod_{\underline{\e}\in [[s]]}^{}  T^{\floor{A_{\ue}(\bm)}+ h_{\underline{\e}}}(\mathcal{C}^{|\underline{\e}|}f_{1})}_{2k+1,T}+o_R(1)
\end{equation*}and the conclusion follows.
\end{proof}

\begin{proof}[Proof of Proposition \ref{sublinearonly}]

Using the arguments after the statement of Proposition \ref{factors}, we may reduce to the case that $a_1(t)$ has maximal growth rate among $a_1,...,a_k$.

We induct on $k$. In the base case of the induction, we want to show that \begin{equation*}
    \limsup\limits_{N\to+\infty} \ \sup_{|c_n|\leq 1} \ \bignorm{\underset{1\leq n\leq N}{\E} c_n T^{\floor{a_1(n)}} f_1   }_{L^2(\m)}\ll \nnorm{f_1}_2.
\end{equation*}Due to Lemma \ref{mainlemma}, it suffices to show that \begin{equation}\label{sngl}
    \limsup\limits_{N\to+\infty} \sup_{|c_{n,N}|\leq 1}\bignorm{\underset{N\leq n\leq N+L(N)}{\E} c_{n,N} T^{\floor{a_1(n)}} f_1   }_{L^2(\m)}\ll \nnorm{f_1}_2
\end{equation}for some suitable sub-linear function $L(t)\in\mathcal{H}$. Since $a_1(t)\succ \log t$, we conclude that \begin{equation*}
    |a'_1(t)|^{-1}\prec |a_1''(t)|^{-1/2}
\end{equation*}by Proposition \ref{growth}. We choose the function $L(t)$ to satisfy \begin{equation*}
     |a'_1(t)|^{-1}\prec L(t) \prec |a_1''(t)|^{-1/2}.
\end{equation*} Therefore, for every $n\in [N,N+L(N)]$, we can write \begin{equation*}
    a_1(n)=a_1(N)+(n-N)a_1'(N)+o_N(1),
\end{equation*}which in turn implies that, for $N$ sufficiently large, we can write \begin{equation*}
    \floor{a_1(n)}=\floor{a_1(N)+(n-N)a_1'(N)}+e_{n,N},
\end{equation*}where $e_{n,N}\in \{0,\pm 1\}$. Substituting this in \eqref{sngl}, we want to prove that \begin{equation*}
    \limsup\limits_{N\to+\infty} \sup_{|c_{n,N}|\leq 1}\bignorm{\underset{N\leq n\leq N+ L(N)}{\E} c_{n,N} T^{\floor{a_1(N)+(n-N)a_1'(N)}+e_{n,N}} f_1   }_{L^2(\m)}\ll \nnorm{f_1}_2.
\end{equation*}
Using Lemma \ref{errors}, we can reduce our problem to \begin{equation*}
      \limsup\limits_{N\to+\infty} \sup_{|c_{h,N}|\leq 1}\bignorm{\underset{0\leq h\leq L(N)}{\E} c_{h,N} T ^{\floor{a_1(N)+ha_1'(N)}} f_1   }_{L^2(\m)}\ll \nnorm{f_1}_2.
\end{equation*}This bound can be proven using the change of variables trick that we have seen in the first example in Section \ref{sectionfactors}. However, we will establish our assertion with a slightly quicker argument below. 

We shall apply the van der Corput inequality. We fix a positive integer $M$ and  choose the quantity $M_N= \floor{|M/a_1'(N)|} $. It is easy to check that $M_N\prec L(N)$, since $L(N)|a'_1(N)|\to+\infty$. Therefore, we can apply the van der Corput inequality to deduce that \begin{multline*}
    \bignorm{\underset{0\leq h\leq L(N)}{\E} c_{h,N} T^{\floor{a_1(N)+ha_1'(N)}} f_1   }_{L^2(\m)}^2\ll \\
    \frac{1}{M_N} +\underset{|m|\leq M_N}{\E}\Big|\underset{0\leq h \leq L(N)}{\E} \overline{c_{h,N}}c_{h+m,N}\int T^{\floor{a_1(N)+ha_1'(N)}}\overline{f_1}\cdot T^{\floor{a_1(N)+(h+m)a_1'(N)}}f_1 \ d\m  \Big| +o_N(1),
\end{multline*}where the implied constant is absolute (and does not depend on $M$). We write \begin{equation*}
    \floor{a_1(N)+(h+m)a_1'(N)}=\floor{a_1(N)+ha_1'(N)}+\floor{ma_1'(N)}+e_{m,h,N},
\end{equation*}where $e_{m,h,N}\in \{0,\pm 1\}$. We rewrite the double average in the middle as \begin{multline*}
    \underset{|m|\leq M_N}{\E}\Big|\underset{0\leq h \leq L(N)}{\E} \overline{c_{h,N}}c_{h+m,N}\int \overline{f_1}\cdot T^{\floor{ma_1'(N)}+e_{m,h,N} }f_1 \ d\m  \Big| \leq \\
    \sum_{z\in \{0,\pm 1\}}^{} \underset{|m|\leq M_N}{\E} \Big|\int \overline{f_1}\cdot T^{\floor{ma_1'(N)} +z}f_1\ d\m \Big|.
\end{multline*}However, note that $|ma_1'(N)|\leq M_N|a_1'(N)|\leq M$.
Thus, for any $z\in \{0,\pm 1\}$, we have \begin{equation*}
    \underset{|m|\leq M_N}{\E} \Big|\int \overline{f_1} \cdot T^{\floor{ma_1'(N)} +z}f_1\ d\m\Big|=\frac{2M+1}{2M_N+1}\underset{|m'|\leq M}{\E}\  p_N(m') \Big|\int \overline{f_1}\cdot T^{m' +z}f_1\ d\m\Big|,
\end{equation*}where $p_N(m')=\#\{ m\in \N\colon \floor{ma_1'(N)}=m' \}$. Since $a_1'(N)\to 0$, we can easily see that for $N$ large enough, we must have \begin{equation*}
    p_N(m')\leq \Big|\frac{1}{a_1'(N)}\Big|.
\end{equation*}Therefore, we have \begin{multline*}
    \frac{2M+1}{2M_N+1}\underset{|m'|\leq M}{\E}\  p_N(m') \Big|\int \overline{f_1}\cdot T^{m' +z}f_1\ d\m\Big|\leq \frac{(2M+1)}{(2M_N+1)|a_1'(N)|}\ \underset{|m'|\leq M}{\E}\  \Big|\int \overline{f_1}\cdot T^{m' +z}f_1\ d\m\Big|\ll\\
    \underset{|m'|\leq M}{\E}\  \Big|\int \overline{f_1}\cdot T^{m' +z}f_1\ d\m\Big|.
    \end{multline*}
Thus, the square of our original average is $\ll$ \begin{equation*}
\sum_{z\in\{0,\pm 1\}}  \underset{|m'|\leq M}{\E}\  \Big|\int \overline{f_1}\cdot T^{m' +z}f_1\ d\m\Big|+o_N(1)
\end{equation*}for some implied constant that does not depend on the original integer $M$. Therefore, we take first $N\to +\infty$ and then $M\to+\infty$ and use the Cauchy-Schwarz inequality to easily reach the conclusion. This establishes the base case  of the induction.

Now assume the claim has been established for all positive integers less than or equal to $ k-1$ (for some $k\geq 2$). We prove that it holds for $k$ as well.
Since we have assumed that $a_1$ has maximal growth rate, we may reorder the given functions so that we have $a_1(t)\gg \dots \gg a_k(t)$.
Let $k_0\leq k$ be the largest integer, such that the function $a_{k_0}$ has the same growth rate as $a_1(t)$. This means that all the functions $a_1,...,a_{k_0}$ have the same growth rate.
We rewrite our average in \eqref{62} as \begin{multline}
    \sup_{\norm{f_2}_{\infty},...,\norm{f_k}_{\infty}\leq 1 }\ \sup_{|c_n|\leq 1  }\bignorm{ \underset{1\leq n\leq N}{\E} c_n \
    \prod_{i=1}^{k_0} T^{\floor{(a_i(n) -a_{k_0}(n)) +a_{k_0}(n)}}f_i\cdot 
   \prod_{i=k_0+1}^{k} T^{\floor{a_i(n)}}f_i          }_{L^2(\m)}=\\
    \sup_{\norm{f_2}_{\infty},...,\norm{f_k}_{\infty}\leq 1 }\ \sup_{|c_n|\leq 1  }\bignorm{ \underset{1\leq n\leq N}{\E} c_n \
   T^{\floor{a_{k_0}(n)}}( \prod_{i=1}^{k_0} T^{\floor{(a_i(n) -a_{k_0}(n))} +e_{i,n}}f_i)\ 
   \prod_{i=k_0+1}^{k} T^{\floor{a_i(n)}}f_i          }_{L^2(\m)}
\end{multline}for some $e_{i,n}\in\{0,\pm 1\}$. Using Lemma \ref{errors}, we may reduce our problem to the case that all the error terms $e_{i,n}$ are zero. Note that the function $a_{k_0}(n)$ dominates each one of the functions $a_1-a_{k_0},...,a_{k_0-1}-a_{k_0}$, as well as the functions $a_i, i\geq k_0$. 
Now, we choose sequences of functions $f_{2,N},...,f_{k,N}$ so that the above average is $1/N$ close to the supremum (we also write $f_{1,N}=f_1$). In addition, we invoke Lemma \ref{mainlemma}
to deduce that it is sufficient to show that \begin{multline}
    \limsup\limits_{R\to+\infty}\ \underset{1\leq r\leq R}{\E}\ \sup_{|c_{n,r}|\leq 1}\  \bignorm{ \underset{r\leq n\leq r+L(r)}{\E} c_{n,r}\ T^{\floor{a_{k_0}(n)}}( \prod_{i=1}^{k_0} T^{\floor{(a_i(n) -a_{k_0}(n))}}f_{i,R})\\
   \prod_{i=k_0+1}^{k} T^{\floor{a_i(n)}}f_{i,R}        }_{L^2(\m)}
   \ll_k \nnorm{f_1}_{2k}
\end{multline}for a sub-linear function $L(t)\in\mathcal{H}$ that we shall choose momentarily. Namely, we choose the function $L\in\mathcal{H}$ to satisfy \begin{equation*}
     |a_{k_0}'(t)|^{-1} \prec L(t)\prec |a_{k_0}''(t)|^{-1/2}
 \end{equation*}and \begin{equation*}
     L(t)\prec (\psi'(t))^{-1}
 \end{equation*}for all the functions $\psi$ of the set $\mathcal{A}=\{a_1-a_{k_0},...,a_{k_0-1}-a_{k_0},a_{k_0+1},...,a_k\}$. To see that such a function exists, we only need to prove that for any function $\psi\in\mathcal{A}$, we have \begin{equation*}
      (a_{k_0}'(t))^{-1}\prec (\psi'(t))^{-1}
 \end{equation*}and \begin{equation*}
       (a_{k_0}'(t))^{-1} \prec |a_{k_0}''(t)|^{-1/2}.
 \end{equation*}The first relation follows easily from the fact that $a_{k_0}$ dominates all functions in $\mathcal{A}$ and L' Hospital's rule. The second relation follows from Proposition \ref{growth}, since $\log t\prec a_{k_0}(t)\prec t$.
 
 Using similar approximations as in the proof of Proposition \ref{sublinearseminorm}, we deduce that for $r$ sufficiently large, we can write \begin{equation*}
     \floor{\psi(n)}=\floor{\psi(r)}+e_{\psi,n} \ \ \text{    for    }\ \  n\in [r,r+L(r)]
 \end{equation*}for every $\psi\in\mathcal{A}$, where $e_{\psi,n}\in \{0,\pm 1\}$. In addition, we can write \begin{equation*}
     \floor{a_{k_0}(n)}=\floor{a_{k_0}(r)+(n-r)a'_{k_0}(r)}+e_{a_{k_0},n}\ \  \text{    for    } \ \ \ n\in [r,r+L(r)],
 \end{equation*}where $e_{a_{k_0},n}\in \{0,\pm 1\}$. Using the argument Lemma \ref{errors} once more to remove the error terms, our original problem reduces to showing \begin{multline}
     \limsup\limits_{R\to+\infty}  \sup_{\norm{f_2}_{\infty},...,\norm{f_k}_{\infty}\leq 1 }\underset{1\leq r\leq R}{\E}  \ \sup_{|c_{h,r}|\leq 1}\bignorm{\underset{0\leq h\leq L(r)}{\E}       c_{h,r}T^{\floor{a_{k_0}(r)+ha'_{k_0}(r) }}( \prod_{i=1}^{k_0} T^{\floor{(a_i(r) -a_{k_0}(r))}}f_{i})\\
   \prod_{i=k_0+1}^{k} T^{\floor{a_i(r)}}f_{i}        }_{L^2(\m)}\ll_k 
   \nnorm{f_1}_{2k}.
 \end{multline}Since the functions $f_i$ are bounded by 1, the last relation follows if we prove that \begin{multline*}
      \limsup\limits_{R\to+\infty}  \sup_{\norm{f_2}_{\infty},...,\norm{f_k}_{\infty}\leq 1 } \underset{1\leq r\leq R}{\E} \ \sup_{|c_{h,r}|\leq 1}\bignorm{\underset{0\leq h\leq L(r)}{\E}       c_{h,r}T^{\floor{a_{k_0}(r)+ha'_{k_0}(r) }}( \prod_{i=1}^{k_0} T^{\floor{(a_i(r) -a_{k_0}(r))}}f_{i})   }_{L^2(\m)}\\
      \ll_k 
   \nnorm{f_1}_{2k}.
 \end{multline*}We choose functions $f_{2,R},...,f_{k_0,R}$ so that the corresponding average is $1/R$ close to the supremum. 
Write $F_{r,R}:= \prod_{i=1}^{k_0} T^{\floor{(a_i(r) -a_{k_0}(r))}}f_{i,R}$. We also fix a positive integer $M$. Repeating the same argument as in the base case, we can show that \begin{multline}\label{label}
    \sup_{|c_{h,r}|\leq 1}\bignorm{\underset{0\leq h\leq L(r)}{\E}       c_{h,r}\ T^{\floor{a_{k_0}(r)+ha'_{k_0}(r) }}F_{r,R}   }_{L^2(\m)}^2\ll\\
    \frac{1}{M}+ \sum_{z\in \{0,\pm 1\} }^{} \underset{|m|\leq M}{\E}\  \Big|\int F_{r,R}\ 
    \cdot T^{m+z}F_{r,R}\ d \m\Big|+o_r(1).
\end{multline}Therefore, we have \begin{multline}\label{sw}
   \sup_{\norm{f_2}_{\infty},...,\norm{f_k}_{\infty}\leq 1 } \underset{1\leq r\leq R}{\E}   \ \sup_{|c_{h,r}|\leq 1}\bignorm{\underset{0\leq h\leq L(r)}{\E}       c_{h,r}\ T^{\floor{a_{k_0}(r)+ha'_{k_0}(r) }}( \prod_{i=1}^{k_0} T^{\floor{(a_i(r) -a_{k_0}(r))}}f_{i})   }_{L^2(\m)}^2\ll \\
    \frac{1}{M}+\underset{1\leq r\leq R}{\E} \sum_{z\in \{0,\pm 1\} }^{} \underset{|m|\leq M}{\E}\  \Big|\int F_{r,R}
    \cdot T^{m+z}F_{r,R}\ d \m\Big|+    O_R(1)
\end{multline}and we want to bound this last quantity by $O_k(1)$ times $\nnorm{f_1}_{2k}^2$.

For a fixed $m\in [-M,M]$ and $z\in\{0,\pm 1\}$, we apply the Cauchy-Schwarz inequality to get \begin{multline*}
    \underset{1\leq r\leq R}{\E}\Big|\int F_{r,R} \cdot T^{m+z}F_{r,R}\ d \m\Big|\leq \Big(    \underset{1\leq r\leq R}{\E}\Big|\int F_{r,R} \cdot T^{m+z}F_{r,R}\ d \m\Big|^2     \Big)^{1/2}=\\
    \Big(  \int \underset{1\leq r\leq R}{\E} (\overline{F_{r,R}}\otimes F_{r,R}) \   (T\times T)^{m+z} (F_{r,R}\otimes \overline{F_{r,R}}) \ d (\m\times \m)    \Big)^{1/2}=\\
    \Big(\int \underset{1\leq r\leq R}{\E} \ \prod_{i=1}^{k_0}  (T\times T)^{\floor{a_i(r)-a_{k_0}(r)}}\big(( \overline{f_{i,R}}\otimes f_{i,R}) \cdot (T\times T)^{m+z} (f_{i,R} \otimes \overline{f_{i,R}}) \big)            \ d(\m\times \m)      \Big)^{1/2}\leq \\
    \bignorm{ \underset{1\leq r\leq R}{\E} \ \prod_{i=1}^{k_0-1}  (T\times T)^{\floor{a_i(r)-a_{k_0}(r)}}\big( (\overline{f_{i,R}}\otimes f_{i,R}) \cdot (T\times T)^{m+z} (f_{i,R} \otimes \overline{f_{i,R}}) \big)   }_{L^2(\m\times \m)}^{1/2}
    \end{multline*}where $f_{1,R}=f_1$. Note that the functions $a_1-a_{k_0},...,a_{k-1}-a_{k_0}$ satisfy the hypotheses of Proposition \ref{sublinearonly}. Therefore, we can apply the induction hypothesis (for $k_0-1<k$) to conclude that \begin{multline*}
         \bignorm{ \underset{1\leq r\leq R}{\E} \ \prod_{i=1}^{k_0-1}  (T\times T)^{\floor{a_i(r)-a_{k_0}(r)}}\big( (\overline{f_{i,R}}\otimes f_{i,R}) \cdot (T\times T)^{m+z} (f_{i,R} \otimes \overline{f_{i,R}}) \big)   }_{L^2(\m\times \m)}^{1/2} \ll_{k_0}\\
         \nnorm{(\overline{f_{1}}\otimes f_{1}) \cdot (T\times T)^{m+z} (f_{1} \otimes \overline{f_{1}})}_{2k_0-2,T\times T}^{1/2}
    \end{multline*}and the last quantity is smaller than $ \nnorm{\overline{f_1} \cdot T^{m+z} f_1}_{2k_0-1,T}$. Putting this in \eqref{sw}, we get \begin{multline*}
        \sup_{\norm{f_2}_{\infty},...,\norm{f_k}_{\infty}\leq 1 }  \underset{1\leq r\leq R}{\E}  \ \sup_{|c_{h,r}|\leq 1}\bignorm{\underset{0\leq h\leq L(r)}{\E}       c_{h,r}\ T^{\floor{a_{k_0}(r)+ha'_{k_0}(r) }}( \prod_{i=1}^{k_0} T^{\floor{(a_i(r) -a_{k_0}(r))}}f_{i})   }_{L^2(\m)}^2\ll_k\\
         \frac{1}{M}+\sum_{z\in \{0,\pm 1\}}^{}\ \underset{|m|\leq M}{\E}\nnorm{\overline{f_1} \cdot T^{m+z} f_1}_{2k_0-1,T}   +o_R(1)\leq\\  \frac{1}{M}+\sum_{z\in \{0,\pm 1\}}^{}\ \underset{|m|\leq M}{\E}\nnorm{\overline{f_1} \cdot T^{m+z} f_1}_{2k-1,T}   +o_R(1),
    \end{multline*}since $k_0\leq k$. Taking $R\to+\infty$ and then $M\to+\infty$, we get that it suffices to show that \begin{equation*}
     \limsup\limits_{M\to+\infty}   \underset{|m|\leq M}{\E}\nnorm{\overline{f_1} \cdot T^{m+z} f_1}_{2k-1} \leq \nnorm{f_1}_{2k}^2
    \end{equation*}for any $z\in \{0,\pm 1\}$. This follows easily by raising to the $2^{2k-1}$-th power and using the power mean inequality, as well as 
    the definition of the Host-Kra seminorms.
\end{proof}

\section{The general case of Proposition \ref{factors} }\label{reductionestimates}

In this section we aim to prove the general case of Proposition \ref{factors}. We maintain the notation of Proposition \ref{factors} and we also assume that at least on of the functions $a_1,...,a_k$ has super-linear growth. We also consider the set of functions \begin{equation*}
    S=\{a_1(t),a_1(t)-a_2(t),...,a_1(t)-a_k(t)  \}
\end{equation*} Functions in $S$ dominate $\log t$ by our hypothesis. Finally, we assume that not every one of the involved functions has the form $p(t)+g(t)$, where $p\in\R[t]$ and $g\in \mathcal{H}$ is sub-fractional, since this case was covered in the previous section (it follows from Corollary \ref{sublinearcorollary}). In particular, we assume that this holds for the function $a_1$.

We will use the following decomposition result from \cite{Richter}.

\begin{lemma}\cite[Lemma A.3]{Richter}\label{decomposition}
Let $a_1,...,a_k\in\mathcal{H}$ have polynomial growth. Then, there exist a natural number $m$, functions $g_1,...,g_m\in \mathcal{H}$, real numbers $c_{i,j}$, where $1\leq i\leq k$ and $1\leq j\leq m$, and real polynomials $p_1,...,p_k$ such that: \begin{enumerate}
    \item $g_1\prec g_2\prec...\prec g_m$,
    \item $t^{l_i}\prec g_i(t)\prec t^{l_i+1}$ for some $l_i\in \Z^{+}$ (i.e. they are strongly non-polynomial) and
    \item for all $i\in \{1,2,...,k\}$ we have \begin{equation*}
        a_i(t)=\sum_{j=1}^{m} c_{i,j}g_j(t)+p_i(t)+o_t(1).
    \end{equation*}
\end{enumerate}
\end{lemma}

Note that the functions $g_j$ do not necessarily belong in the set of linear combinations of the $a_1,...,a_k$. The proof of this lemma can be found in the appendix of \cite{Richter}. As an example, if we have the pair $\{t+t^{3/2}, t^2+t^{5/2}\}$, then the functions in the above decomposition are $\{g_1,g_2,p_1,p_2\}=\{t^{3/2},t^{5/2},t,t^2\}$.

Returning to our original problem, we split the given family of functions into two sets \begin{equation*}
    J_1=\{a_i\colon a_i(t)\ll t^{\delta} \text{ for all } \delta>0\} \text{ and } J_2=\{a_i\colon \exists\ \delta>0 \text{ with } \ a_i(t)\gg t^{\delta} \}.
\end{equation*}
We do the same for the set $S$ of differences:\begin{equation*}
     S_1=\{f\in S\colon f(t)\ll t^{\delta} \text{ for all } \delta>0\} \text{ and } S_2=\{f\in S\colon \exists\ \delta>0 \text{ with } \ f(t)\gg t^{\delta} \}.
\end{equation*}Observe that the function $a_1$ belongs to the sets $J_2$ and $S_2$ due to our assumption in the beginning of this section.

We will see that the slow-growing functions in sets $J_1$ and $S_1$ will be approximately equal to a constant, when we consider averages on small intervals. For the remaining functions, we will use the Taylor expansion to approximate them.
We split the proof into several steps. Steps 1 through 4 of this proof correspond to step 1 in example a) of section \ref{sectionfactors}, while steps 5 and 6 of the proof correspond to step 2 of the same example. The remaining two steps correspond to step 3 of example a). In Step 8, we will also use the results of the special case of the previous section.

\subsection{Step 1: Introducing a double averaging}

Let $L(t)\in\mathcal{H}$ be a sub-linear function to be specified later. 
We can consider a priori functions that satisfy $L(t)\prec t^{1-\e}$ for some $\e>0$ (i.e. we exclude functions like $t/\log t$ ). Invoking Lemma \ref{mainlemma}, we see that it is sufficient to prove that\begin{equation}\label{xef}
  \limsup\limits_{R\to\infty}  \underset{1\leq r\leq R}{\E} \sup_{|c_{r,n}|\leq 1} \bignorm{\underset{r\leq n\leq r+L(r)}{\E} c_{r,n}\ T^{\floor{a_1(n)}} f_1\cdot ...\cdot T^{\floor{a_k(n)}}f_{k,R}   }_{L^2(\m)}^{2^t}=  0
\end{equation}for any sequences of 1-bounded functions $f_{2,R},...,f_{k,R}$ and some positive integer parameter $t$, which will depend only on the original functions $a_1,...,a_k$. Therefore, when applying Lemma \ref{errors}, we can always assume that the implicit constant (which depends on the exponent $2^t$) is an $O_{a_1,...,a_k}(1)$ constant.

 We observe that \eqref{xef} follows if we show that\begin{equation}\label{R}
   \sup_{||f_2||_{\infty}\leq 1,...,||f_k||_{\infty} \leq 1}  \underset{1\leq r\leq R}{\E} \sup_{|c_{r,n}|\leq 1} \bignorm{\underset{r\leq n\leq r+L(r)}{\E} c_{r,n}\ T^{\floor{a_1(n)}} f_1\cdot   ...\cdot T^{\floor{a_k(n)}}f_{k}   }_{L^2(\m)}^{2^t}
\end{equation}goes to 0, as $R\to+\infty$.

\subsection{Step 2: Eliminating the small functions of $J_1$}

While in example a) of Section 4 we used the Taylor expansion right at the beginning, it is more convenient to reverse our steps a bit in the proof.

Assume that the function $a_i$ belongs to the set $J_1$ (namely, it is a sub-fractional function). Then, for any $n\in [r,r+L(r)]$, we have\begin{equation*}
    |a_i(n)-a_i(r)|=|n-r||a'_i(\xi)|
\end{equation*} for some $\xi \in [r,n]$. Since $|a'_i(t)|\searrow 0$, we get \begin{equation*}
    |a_i(n)-a_i(r)|\leq |L(r)||a'_i(r)|,
\end{equation*}which is $o_r(1)$. Note that we already assumed that we will eventually choose $L\in\mathcal{H}$ such that $L(t)\ll t^{1-\e}$, which makes the previous statements valid  (see the discussion at the end of the Appendix). Thus, if $r$ is sufficiently large and $n\in [r,r+L(r)]$, we can write $\floor{a(n)}=\floor{a(r)}+\e_{r,n}$, where $\e_{r,n}\in \{0,1\}$. Using the argument in Lemma \ref{errors}, we absorb the error terms $\e_{r,n}$ in the supremum outside of the averages in \eqref{R}. 

The iterate corresponding to the function $f_{i}$ has now become constant and we can ignore it. In conclusion, we have reduced our problem to the case that the set $J_1$ is empty.

\subsection{Step 3: Concatenating the functions of the set $S_1$}

Assume that the function $a_1-a_i$ belongs to $S_1$. Then, mimicking the arguments of the previous step, we can write $a_i=a_1+(a_i-a_1)$ where the function $a_i-a_1$ is asymptotically a constant in the interval $[r,r+L(r)]$. Then, we can combine the product of all such terms
\begin{equation*}
    T^{\floor{a_1(n)}}f_1 \prod_{a_1-a_i\in S_1} T^{\floor{a_i(n)}}f_{i}
\end{equation*}into one iterate $T^{\floor{a_1(n)}} \tilde{f}_{r}$ (we use again the argument in Lemma \ref{errors} to remove the error terms), where  \begin{equation}\label{fapeiro}
   \tilde{f}_r= f_1\cdot  T^{\floor{\theta_1(r)}}h_{1}\cdot...
    \cdot T^{\floor{\theta_{\ell}(r)}}h_{\ell},
    \end{equation}where $h_1,...,h_{\ell}$ are functions in $L^{\infty}(\m)$ and the functions $\theta_1,...,\theta_{\ell} \in \mathcal{H}$ are sub-linear functions that satisfy \begin{equation*}
        \log t\prec \theta_{i}(t)\prec t^{\delta} 
    \end{equation*}for all $\delta>0$. In addition, the assumption that the pairwise differences of the functions $a_1,...,a_k$ dominate $\log t$ implies that \begin{equation*}
        \log t\prec \theta_{i}(t)-\theta_j(t) \
    \end{equation*}for $i\neq j$.

Now the original problem reduces to the following: If all the functions $a_1,...,a_k$ are such that the sets $J_1$ and  $S_1$ are empty, then show that the averages \begin{multline}
     \sup_{||f_2||_{\infty},...,||f_k||_{\infty} \leq 1}\sup_{ ||h_1||_{\infty},...,||h_{\ell}||_{\infty}\leq 1}\\
     \underset{1\leq r \leq R}{\E}\ \sup_{|c_{r,n}|\leq 1} 
     \bignorm{\underset{{r\leq n\leq r+L(r)}}{\E}\ c_{r,n}\  T^{\floor{a_1(n)}} \tilde{f}_{r}\cdot ...\cdot T^{\floor{a_k(n)}}f_{k}   }_{L^2(\m)}^{2^t}
\end{multline} go to 0 as $R\to+\infty$, where the function $\tilde{f}_r$ is the function in \eqref{fapeiro}.

We can repeat the same argument of this step to reduce to the case where $a_i(t)-a_j(t)\gg t^{\delta}$ for some $\delta>0$. Indeed, if the difference $a_i-a_j$ is sub-fractional, we can combine the iterates corresponding to these two functions into a single iterate of the form $T^{\floor{a_i(n)}}g_r$ for some function $g_r$. In order to replace $g_r$ by a function that does not depend on $r$, we move the supremum of the $f_2,...,f_k$ inside the outer average. In conclusion, it suffices to show that \begin{multline}\label{no logs}
    \sup_{ ||h_1||_{\infty},...,||h_{\ell}||_{\infty}\leq 1}\ 
     \underset{1\leq r \leq R}{\E} \ \sup_{||f_2||_{\infty},...,||f_k||_{\infty}\leq 1}\ \sup_{|c_{r,n}|\leq 1} 
     \bignorm{\underset{{r\leq n\leq r+L(r)}}{\E}c_{r,n}\  T^{\floor{a_1(n)}} \tilde{f}_{r}\cdot ...\cdot T^{\floor{a_k(n)}}f_{k}   }_{L^2(\m)}^{2^t}
\end{multline}goes to 0 as $R\to+\infty$, where $\tilde{f}_r$ is the function in \eqref{fapeiro} and all differences $a_i-a_j$ dominate some fractional power\footnote{Since our functions $a_1,...,a_k$ dominate a fractional power, we can now use the fact that the classes $S(a_i,k)$ (defined and studied in the Appendix) can be well defined in order to approximate all of them by polynomials.}.
Recall that the functions $\theta_i$ satisfy \begin{equation*}
     \log t\prec \theta_{i}(t)\prec t^{\delta} \ \text{ for every }\ \delta>0
\end{equation*}and \begin{equation*}
      \log t\prec \theta_{i}(t)-\theta_j(t).
\end{equation*}

\subsection{Step 4: Approximating by polynomials}

In this step, we will use the Taylor expansion to replace the functions $a_i$ by polynomials in the intervals $[r,r+L(r)]$.
First of all, we can use Lemma \ref{decomposition} in order to write \begin{equation}\label{expansionform}
        a_i(t)=\sum_{j=1}^{m}c_{i,j}g_j(t)+q_i(t)+o_t(1),
\end{equation}where $g_1\prec g_2\prec...\prec g_m$ are strongly non-polynomial functions and $q_i(t)$ are real polynomials. We immediately conclude that the function $g_m$ cannot be sub-fractional. Indeed, if that was the case, then all the functions $a_i$ would be a sum of a polynomial plus a sub-fractional function, which is at odds with our initial assumption.

The $o_t(1)$ terms can be eliminated by using an argument similar to the proof of Lemma \ref{errors}. In addition, we may assume that $c_{1,m}\neq 0$ (and thus $g_m$ exists in the expansion of $a_1$). This can be proven by an argument similar to the one in the beginning of Section \ref{sectionfactors} (the same reasoning we used to reduce our problem to the case that $a_1$ has maximal growth rate). Of course, by assuming  this new property, we abandon the assumption that $a_1$ has maximal growth rate.

We define \begin{equation*}
    \F=\{g_1,...,g_m\}
\end{equation*} and let $\A=\{g_1,...,g_l\}\subseteq\F$ be the set of functions that satisfy $g_i(t)\ll t^{\delta}$ for all $\delta>0$ (i.e the sub-fractional functions). We have that $g_m\not\in \A$.

By the reductions in steps 2 and 3, we have that $a_i(t)\gg t^{\delta_i}$ for some $\delta_i>0$ and a similar relation holds for the differences $a_i-a_j$. Therefore, we have the following property:

\begin{equation}\label{Papeiro}\tag{{\bf P}}
\text{ If}\  i_1\neq i_2, \text{ we have either}\  c_{i_1,j}\neq c_{i_2,j}\ \text{ for some}\  j>l, \text{ or}\  q_{i_1}(t)-q_{i_2}(t)\ \text{is non-constant.}
    \end{equation}

Now every function $g\in\A$ satisfies \begin{equation*}
    \max_{n\in [r,r+L(r)]}|g(n)-g(r)|=o_r(1)
\end{equation*}by the arguments in the preceding steps. We can use the argument in Lemma \ref{errors} to remove the error term $o_r(1)$ and then substitute each function $g\in \A$ in the interval $[r,r+L(r)]$ by a constant (namely, the value of the function $g$ at $r$). These constants can be absorbed by the supremum of the $f_2,...,f_k$ and the use of Lemma \ref{errors}. Therefore, we may assume that all functions $g_1,...,g_m$ dominate some fractional power $t^{\delta}$ (equivalently $\A=\emptyset$) and that property \eqref{Papeiro} above holds with $l=0$.

Since the functions $g_1,...,g_m$ dominate some fractional power, the classes \begin{equation*}
    S(g_i,n)=\{f\in\mathcal{H}, (g_i^{(n)}(t))^{-1/n}\preceq f(t)\prec (g_i^{(n+1)}(t))^{-1/(n+1)} \}
\end{equation*}are well defined for $n$ large enough. We remind the reader that these classes and their properties are all studied in the Appendix and we will use them freely from this point onward.

 Let $d$ be a natural number and  for every function $g\in \F$, we consider the natural number $k_g$, such that the function $|g_m^{(d)}(t)|^{-\frac{1}{d}}$ belongs to the class $S(g,k_g)$. This class always exists,  if we pick our number $d$ to be sufficiently large. We immediately deduce that $k_g\leq d$ for every $g\in \F$, while $k_{g_m}=d$. 
 
 Let $q$ be a positive real number (but not an integer), such that $t^q$ dominates all functions $g_1,...,g_m$ and the polynomials $q_1,...,q_k$. In particular, this implies that, for all $1\leq i\leq m$, all derivatives of $g_i$ of order bigger than $q$ go to 0 (as $t\to +\infty$). This is a consequence of Proposition \ref{prop:basic}. We make the additional assumption that our integers $k_g$ are very large compared to $q$, which can be attained if we take our initial number $d$ to be sufficiently large. The inequality $k_g\geq 10q$ will suffice for our purposes.

\begin{definition}
We say that two functions $f\ll g$ of $\mathcal{H}$ have the property $\mathcal{Q}$, if they have the same growth rate, or if the ratio \begin{equation*}
    \frac{g(t)}{f(t)}
\end{equation*}dominates some fractional power $ t^{\delta},\  \delta>0$.
\end{definition}

 We consider two possible cases:\\
a) Assume that for every $g\in \F\setminus\{g_m\}$, the functions $|g_m^{(d)}(t)|^{-\frac{1}{d}}$ and $|g^{(k_g)}(t)|^{-\frac{1}{k_g}}$ have the property\footnote{An example of functions that fall in this case is the pair $(t^{3/2},t\log t)$, if we consider their second derivatives. We can easily check that the ratio of the second derivatives of these two functions {raised to the $-\frac{1}{2}$-th power} grows like the function $t^{1/4}$.} $\mathcal{Q}$. Then, our selection will be the classes $S(g,k_g)$ as they stand. Furthermore, we choose $L(t)\in\mathcal{H}$ to be any function that belongs to the intersection of the classes $S(g_i,k_{g_i})$ (which is non-empty by definition).
In this case, we call the function $g_m$ our "special" function. Note that \begin{equation*}
    |g^{(k_g)}(t)|^{-\frac{1}{k_g}}\preceq |g_m^{(d)}(t)|^{-\frac{1}{d}}
\end{equation*}for $g\neq g_m$ in this case.\\
b) Assume that the above case does not hold\footnote{An example of functions that fall in this second case is the pair $(t\log t,t\log\log t)$, if we again consider their second derivatives. A simple computation yields that the growth rate of the ratio of the involved functions grows like the function $\sqrt{\log t}$ and, thus, they fail property $\mathcal{Q}$.}. Then, among all the functions $|g^{(k_g)}(t)|^{-\frac{1}{k_g}}$ for which the property $\mathcal{Q}$ fails (in relation to $|g_m^{(d)}(t)|^{-\frac{1}{d}}$), we choose a function $g$ for which $|g^{(k_g)}(t)|^{-\frac{1}{k_g}}$ has minimal growth rate. Then, we choose a function $L\in\mathcal{H}$ with the following properties:

i) If a function $\tilde{g}$ is such, that $|\tilde{g}^{(k_{\tilde{g}})}(t)|^{-\frac{1}{k_{\tilde{g}}}}$ fails to satisfy property $\mathcal{Q}$ in relation to $|g_m^{(d)}(t)|^{-\frac{1}{d}}$ and has different growth rate than $g$, then  we have \begin{equation*}
 |{(\tilde{g})}^{(k_{\tilde{g}}-1)}(t)|^{-\frac{1}{k_{\tilde{g}}-1}}   \prec L(t)\prec |(\tilde{g})^{(k_{\tilde{g}})}(t)|^{-\frac{1}{k_{\tilde{g}}}}.
\end{equation*}
Namely, we have $L(t)\in S(\tilde{g},k_{\tilde{g}}-1)$.

ii) If the function $\tilde{g}$ has the same growth rate as $g$, then we have $k_g=k_{\tilde{g}}$ and the classes $S(g,k_g)$ and
$S(\tilde{g},k_{\tilde{g}})$ coincide. In this case, we leave the integer $k_{\tilde{g}}$ as is and we will have $L(t)\in S(\tilde{g},k_{\tilde{g}})$.

iii) The third case is when the function $\tilde{g}$ satisfies property $\mathcal{Q}$ in relation to $|g_m^{(d)}(t)|^{-\frac{1}{d}}$. Then, we leave the the integer $S(\tilde{g},k_{\tilde{g}})$ as is and take $L(t)\in S(\tilde{g},k_{\tilde{g}})$.

The existence of such a function $L(t)$ follows by our minimality assumption on $|g^{(k_g)}(t)|^{-\frac{1}{k_g}}$. In this case, $g$ is our "special" function.

We denote by $k'_g$ the new integers that appear after the above procedure.

\begin{claim}
    For the choice we have made above, the function $|z^{(k'_z)}(t)|^{-\frac{1}{k'_z}}$ satisfies property $(\mathcal{Q})$ in relation to our special function, for any $z\in \F$. 
\end{claim}

\begin{proof}
  
  If we are in case a) above, the functions $ |g_m^{(d)}(t)|^{-\frac{1}{d}}$ and $|g^{(k_g)}(t)|^{-\frac{1}{k_g}}$ have the same growth rate or their ratio dominates a fractional power (for any $g\in \F$) and we are done.
  
  In the second case, we have a special function $g$ ($k_g=k'_g$). We consider functions $z\neq g$ such that $|g^{(k_g)}(t)|^{-\frac{1}{k_g}}$
and $|z^{(k'_z)}(t)|^{-\frac{1}{k'_z}}$ have different growth rates (because otherwise the claim is trivial). Then there are two possibilities:

$\bullet$ If the original function $|g_m^{(d)}(t)|^{-\frac{1}{d}}$ and $|z^{(k_z)}(t)|^{-\frac{1}{k_z}}$ had a ratio dominating a fractional power, then the claim follows (in this case, we must have $k'_z=k_z$).

$\bullet$ If the original function $|z^{(k_z)}(t)|^{-\frac{1}{k_z}}$ failed property $\mathcal{Q}$ in relation to $|g_m^{(d)}(t)|^{-\frac{1}{d}}$, then we have \begin{equation*}
|g^{(k_g)}(t)|^{-\frac{1}{k_g}}  \prec  |z^{(k_z)}(t)|^{-\frac{1}{k_z}} \ \ \text{          (due to minimality) }
\end{equation*}and thus $L(t)\in S(z,k_z-1)$. We easily see that the functions $|g^{(k_g)}(t)|^{-\frac{1}{k_g}}$ and $|z^{(k_z-1)}(t)|^{-\frac{1}{k_z-1}}$ differ by a fractional power. Indeed, we have a "gain" of some power $t^{\delta}$ when passing from $S(z,k_z-1)$ to $S(z,k_z)$ due to \eqref{20000}. Therefore, if the functions $|z^{(k_z)}(t)|^{-\frac{1}{k_z}}$ and $|g^{(k_g)}(t)|^{-\frac{1}{k_g}}$ were "close", then $|z^{(k_z-1)}(t)|^{-\frac{1}{k_z-1}}$ and $|g^{(k_g)}(t)|^{-\frac{1}{k_g}}$ differ by a fractional power.
\end{proof}

For convenience, we will use the same notation $S(g,k_g)$ for the new classes that have been chosen after the above operation (that is we replace $k'_g$ by $k_g$).
\begin{remark*}
The above proof also implies that the growth rate of $ |g^{(k_g)}(t)|^{-\frac{1}{k_g}}$ is maximized when $g$ is the special function.
\end{remark*}

 We denote by ${\tilde{g}}$ the special function given by our above arguments. For any function $g\in \F$, we use the Taylor expansion around the point $r$ to obtain
\begin{equation}\label{expansion}
    g(r+h)=g(r) +\cdots+\frac{g^{(k_g)}(r)h^{k_g}}{k_g!} + \frac{g^{(k_g+1)}(\xi_m)h^{k_g+1}}{(k_g+1)!} \ \text{ for some } \xi_m\in [r,r+m],
\end{equation}for all $0\leq h\leq L(r)$. We observe that the last term is $o_r(1)$ while the second to last term in the above expansion diverges when $h=L(r)$ (see the discussion after the proof of Proposition \ref{growth}). Therefore, we have \begin{equation*}
    g(r+h) =p_{r,g}(h) +o_r(1)
\end{equation*}where $p_{r,g}$ is a polynomial.

\subsection{Step 5: The change of variables}

In this step, we do a change of variables trick. Our purpose is to rewrite the above polynomials in such a way, that the leading coefficients are good sequences in order to be able to apply Proposition \ref{PET}. All the work we did in the previous step (namely, making sure that our functions satisfied Property $\mathcal{Q}$) will ensure that the leading coefficients of our polynomials will be good sequences that either converge to a (non-zero) real number, or their growth rate is larger than some fractional power.
A similar trick is also used in \cite{FraHardy1}.

Assume that ${\tilde{g}}$ is our special function with the polynomial expansion 
\begin{equation*}
     {\tilde{g}}(r+h)={\tilde{g}}(r) +\cdots +\frac{{\tilde{g}}^{(k_{{\tilde{g}}})}(r)h^{k_{{\tilde{g}}}}}{k_{{\tilde{g}}}!} + o_r(1).
\end{equation*}
Every $0\leq h \leq L(r)$ can be written as \begin{equation*}
    h=w \bigfloor{ \Big|   \frac{k_{{\tilde{g}}}!}{{\tilde{g}}^{(k_{{\tilde{g}}})}(r)} \Big|^{\frac{1}{k_{{\tilde{g}}}}}} +v
\end{equation*}
for some integers $w,v$, where \begin{equation*}
    0\leq w\leq \frac{L(r)}{\bigfloor{ \Big|   \frac{k_{{\tilde{g}}}!}{{\tilde{g}}^{(k_{{\tilde{g}}})}(r)} \Big|^{\frac{1}{k_{{\tilde{g}}}}}}} = D_r  
\end{equation*}
and \begin{equation*}
    0\leq v\leq \bigfloor{ \Big|   \frac{k_{{\tilde{g}}}!}{{\tilde{g}}^{(k_{{\tilde{g}}})}(r)} \Big|^{\frac{1}{k_{{\tilde{g}}}}}}-1.
\end{equation*}Note that $D_r \succ 1$, because $L(t)\in S({\tilde{g}},k_{{\tilde{g}}})$.
We denote by $u(r)$ the function inside the integer part above, namely, we define \begin{equation*}
    u(r):=\Big|   \frac{k_{{\tilde{g}}}!}{{\tilde{g}}^{(k_{{\tilde{g}}})}(r)} \Big|^{\frac{1}{k_{{\tilde{g}}}}},
\end{equation*}which is a (sub-linear) function in $\mathcal{H}$. In addition, since we have chosen the numbers $k_g$ to be sufficiently large, we can ensure that the function $u$ dominates some fractional power (this follows by statement ii) of Lemma \ref{basic}).

We observe that (recall that $\tilde{f}_r$ is given by \eqref{fapeiro})\begin{multline}\label{compl}
   \sup_{\norm{f_2}_{\infty},...,\norm{f_k}_{\infty}\leq 1}    \bignorm{\underset{0\leq h\leq L(r)}{\E}\ c_{h,r}\  T^{\floor{a_1(r+h)}} \tilde{f}_{r}\cdot ...\cdot T^{\floor{a_k(r+h)}}f_k  }_{L^2(\m)}^{2^t}\leq  \\
\sup_{\norm{f_2}_{\infty},...,\norm{f_k}_{\infty}\leq 1}   \  \underset{1\leq v\leq \floor{u(r)}-1}{\E}\ \sup_{|c_{h,r,v}|\leq 1}  \bignorm{  
  \underset{h\equiv v (mod \ \floor{u(r)})}{\E}\ c_{h,r,v}\  T^{\floor{a_1(r+h)}} \tilde{f}_{r}\cdot 
  ...\cdot T^{\floor{a_k(r+h)}}f_k                                       }_{L^2(\m)}^{2^t},
\end{multline}where the above bound follows by applying the H\"{o}lder and triangle inequalities.
We will bound the innermost average in the norm by a quantity that does not depend on $v$. 

Fix a $v$ as above. For every $h\equiv v  (mod\ \floor{u(r)})$, we can write each of the polynomials $p_{g,r}(h)$ in the previous step as a new polynomial $\tilde{p}_{r,v,g}(w)$ in the new variable $w$. We are only interested in the leading coefficients of the new polynomials. Using \eqref{expansion}, we see that it is equal to \begin{equation}\label{c_g}
    c_{g}(r)=  \frac{g^{(k_g)}(r)}{k_g!}\cdot \floor{u(r)}^{k_g}=  \frac{g^{(k_g)}(r)}{k_g!}\cdot \bigfloor{ \Big|   \frac{k_{{\tilde{g}}}!}{{\tilde{g}}^{(k_{{\tilde{g}}})}(r)} \Big|^{\frac{1}{k_{{\tilde{g}}}}}}^{{k_g}}.
\end{equation}

Now assume that $g\in \F$. The function $c_g(r)$ is not a function in the Hardy field $\mathcal{H}$,
 but we will prove that it is a good sequence (see Definition \ref{good sequence}). Therefore, we seek to approximate it by a function in $\mathcal{H}$. To achieve this, we can define the function $d_g(t)\in \mathcal {H}$ by removing the floor function: \begin{equation}\label{d_g}
    d_g(t) =  \frac{g^{(k_g)}(t)}{k_g!}\cdot  \Big|   \frac{k_{{\tilde{g}}}!}{{\tilde{g}}^{(k_{{\tilde{g}}})}(t)} \Big|^{\frac{k_g}{k_{{\tilde{g}}}}}.
\end{equation}It is obvious that $c_g(r)/d_g(r)\to 1$. However, we have something stronger:

\begin{claim}
    For all $g\in \F$, we have \begin{equation*}
        |c_g(r) -d_g(r)|=o_r(1).
    \end{equation*}
\end{claim}
\begin{proof}
 We will use the inequality \begin{equation*}
      |a^{c}-b^{c}|\leq c|a-b||a|^{c-1},          
  \end{equation*}which holds when $|b|\leq |a|$ and $c\in \N$. An application of this inequality reduces the problem to showing that \begin{equation}\label{zxcvbnm}
 |{\tilde{g}}^{(k_{{\tilde{g}}})}(t)|^{-\frac{1}{k_{{\tilde{g}}}}}\prec |g^{(k_g)}(t)|^{-\frac{1}{k_g-1}}.
  \end{equation}
  Since $L(t)\in S({\tilde{g}},k_{{\tilde{g}}})$, it is sufficient to show that \begin{equation*}
      L(t)\prec |g^{(k_g)}(t)|^{-\frac{1}{k_g-1}}
  \end{equation*}and now using the fact that $L(t)\in S(g,{k_g})$, our conclusion follows if we prove that \begin{equation*}
     |g^{(k_g+1)}(t)|^{-\frac{1}{k_g+1}} \prec |g^{(k_g)}(t)|^{-\frac{1}{k_g-1}}.
  \end{equation*}
  Substituting $g^{(k_g+1)}(t)\sim g^{(k_g)}(t)/t$ in the above equation (we use Proposition \ref{prop:basic} and the fact that the numbers $k_g$ are assumed to be large enough), this reduces to \begin{align}
      g^{(k_g)}(t)&\prec t^{\frac{1-k_g}{2}}.
  \end{align} However, recall that we have chosen a non-integer
  $q$, such that $g(t)\ll t^q$ for all $g\in \F$ and we have also chosen $k_g \geq 10q-1$. Applying Proposition \ref{prop:basic}, we have $g^{(k_g)}(t)\prec t^{q-k_g}$ and now the claim easily follows. 
\end{proof}

\begin{claim}
   a) We have that the function $d_g(t)$ in \eqref{d_g} is a sub-linear function that either satisfies  $ t^{\e}\prec d_g(t)$ for some $\e>0$ or converges to a non-zero constant\footnote{ Thus, the leading coefficients $c_g(r)$ in \eqref{c_g} are good sequences.}.\\
   b) We have the growth relation $d_g(t)\prec ({\tilde{g}}^{(k_{{\tilde{g}}})}(t))^{-\frac{1}{k_{{\tilde{g}}}}} $ and, thus, $d_g$ has sub-linear growth.
\end{claim}
\begin{proof}
  Property ($\mathcal{Q}$) implies that $d_g(t)$ converges to a non-zero constant, or dominates a fractional power $t^{\delta}$. For the second part, we observe that a simple computation shows that this is equivalent to \eqref{zxcvbnm}, which has already been established.
 \end{proof}

\begin{claim}
 If $g,h$ are distinct functions in the set $\{g_1,...,g_m\}$ such that $d_g(t) \sim d_h(t)$, then $k_g\neq k_h$.
\end{claim}

\begin{proof}  
   Assume that we have both $k_g=k_h$ and $d_g\sim d_h$. This implies that \begin{equation*}
      g^{(k_g)}(t)\sim h^{(k_h)}(t)
  \end{equation*}and L'Hospital' rule implies that $g\sim h$. Since $g,h$ have distinct growth rates, this last relation cannot hold and we arrive at a contradiction.
\end{proof}

We have seen that the functions $g_1,...,g_m$ admit a polynomial expansion and, after the change of variables above, their leading coefficients become sub-linear good sequences. Now, we look how the leading coefficients of the polynomials $q_1,...,q_k$ in \eqref{expansionform} transform after the above change of variables. Note that $q_i(r+h)$ is also a polynomial $q_{i,r}(h)$ in the variable $h$. Writing again \begin{equation*}
    h=w\floor{u(r)}+v
\end{equation*}as above, we see that $q_i(r+h)=q_{i,r,v}(w)$ where $q_{i,r,v}$ is a real polynomial.
It is straightforward to check that the leading coefficients of the $q_{i,r,v}$ have the form $c\floor{u(r)}^{\theta}$, where $c\in \R^{*}$ and $\theta\in \N^{+}$. These are good sequences, since they are asymptotically equal to \begin{equation*}
    c \Big|   \frac{k_{{\tilde{g}}}!}{{\tilde{g}}^{(k_{{\tilde{g}}})}(r)} \Big|^{\frac{\theta}{k_{{\tilde{g}}}}},
\end{equation*}which is a function in $\mathcal{H}$ (and its limit is obviously non-zero).

Now, we recall \eqref{expansionform}. When restricted to the interval $[r,r+L(r)]$, every one of our original functions $a_i$, where $1\leq i\leq k$ can be written as a sum of polynomials, whose leading coefficients are good sequences, plus an $o_r(1)$ term. We can eliminate the error terms $o_r(1)$ by using the argument in Lemma \ref{errors} once again. 
In particular, any one of these good sequences (denote $a_r$) satisfies one of the following:\\
a) there exists a sub-linear function $\phi\in \mathcal{H}$, such that $a_r=\phi(r)+o_r(1)\prec u(r)$ and $\phi(t)\gg t^{\delta}$ for some $\delta>0$, \\
b) they have the form  $c\floor{u(r)}^{\theta}$, where $c\in\R$ and $\theta$ is a positive integer or\\
c) they converge to a non-zero real number.

 We denote the polynomial corresponding to $a_i$ as $P_{i,r,v}$ and we observe that its degree is independent of $r$. In view of Property \eqref{Papeiro}, we deduce that the leading coefficient of $P_{i,r,v}-P_{j,r,v}$ is either the leading coefficient of the polynomial $q_{i,r,v}(t)-q_{j,r,v}(t)$ (which in this case must be a non-constant polynomial), or it is equal to the leading coefficient of \begin{equation}\label{RJ}
    R_{ij.r.v}(w)= \sum_{n=1}^{m} \big(c_{i,n}-c_{j,n}\big)\tilde{p}_{r,g_j,v}(w)
 \end{equation}or it is a combination of these two coefficients. In the first case, it has the form b) above and is a good sequence. In the second case, it is a linear combination of sequences of the form $a)$ or $c)$. That is, there are functions $g_{i_1},...,g_{i_{\lambda}}$, where $i_1,...,i_{\lambda}\in\{1,2,...,m\}$ such that the leading coefficients of the polynomials $\tilde{p}_{r,g_{i_j},v}$ are all sequences of the form $a)$ or $c)$ and the leading coefficient of the polynomial $R_{ij,r,v}$ in \eqref{RJ} is equal to the leading coefficient of \begin{equation}\label{arx}
     \sum_{\alpha=1}^{\lambda}(c_{i,i_{\alpha}}-c_{j,i_{\alpha}}     )\tilde{p}_{r,g_{i_{\alpha}},v    }.
 \end{equation} We will use Claim 5: if any two of the polynomials $\tilde{p}_{r,g_{i_{\alpha}},v    }$ have the same degree, then their leading coefficients are sequences with distinct growth rates. Therefore, the leading coefficient of  $R_{ij,r,v}$ is a linear combination of good sequences with pairwise distinct growth rates and it is straightforward to see that it is itself a good sequence. Finally, we observe that the final case cannot happen (namely, a combination of these two coefficients). That is because the degree of the polynomial $q_{i,r,v}(t)-q_{j,r,v}(t)$, which is equal to the degree of $q_i-q_j$, is very small compared to the degree of the polynomial in \eqref{arx}, because we chose the degrees $k_g$ of the polynomials in the Taylor expansions to be very large compared to the degrees of the polynomials $q_1,...,q_k$.

 Our original problem reduces to the following (recall \eqref{compl}): for every measure-preserving system $(X,\m,T)$ and function $f_1\in L^{\infty}(\m)$ with $f_1 \perp Z_{\tilde{s}}(X)$ for some $\tilde{s}\in \N$, there exists a positive integer $t=t(a_1,...,a_k)$ such that:  
  \begin{multline}\label{semifinalreduction}
      \lim\limits_{R\to+\infty} \ \sup_{ ||h_1||_{\infty}\leq 1,...,||h_{\ell}||_{\infty}\leq 1}\
      \underset{1\leq r\leq R}{\E}\
      \  \underset{0\leq v\leq \floor{u(r)}-1}{\E}\\
      \sup_{||f_2||_{\infty}\leq 1,...,||f_k||_{\infty} \leq 1}\ \sup_{|c_{w,r,v}|\leq 1}\ \bignorm{ \ 
  \underset{0\leq w\leq D_r}{\E} \ c_{w,r,v}\  T^{\floor{ P_{1,r,v}(w) }} \tilde{f}_{r}\cdot ...\cdot T^{\floor{P_{k,r,v}(w) } }f_k                                       }_{L^2(\m)}^{2^t}=0,
  \end{multline}where \begin{equation}\label{f1}
   \tilde{f}_r= f_1\cdot  T^{\floor{\theta_1(r)}}h_{1}\cdot...
    \cdot T^{\floor{\theta_{\ell}(r)}}h_{\ell}
    \end{equation}for functions $\theta_1,...,\theta_{\ell}\in\mathcal{H} $ that satisfy \begin{align*}
       & \log t \prec \theta_i(t)\prec t^{\delta}\\
         \log t \prec & \ \theta_i(t)-\theta_j(t)\prec t^{\delta} \ \text{ for } \ i\neq j
    \end{align*}for all $\delta>0$.
    
    Observe that \begin{multline*}
  \underset{0\leq v\leq \floor{u(r)}-1}{\E}\  \sup_{||f_2||_{\infty}\leq 1,...,||f_k||_{\infty} \leq 1}\ \sup_{|c_{w,r,v}|\leq 1}\ \bignorm{ 
  \underset{0\leq w\leq D_r}{\E} \ c_{w,r,v}\  T^{\floor{ P_{1,r,v}(w) }} \tilde{f}_{r}\cdot ...\cdot T^{\floor{P_{k,r,v}(w) } }f_k                                       }_{L^2(\m)}^{2^t}\leq \\
  \max_{0\leq v\leq \floor{u(r)}-1} \sup_{||f_2||_{\infty}\leq 1,...,||f_k||_{\infty} \leq 1}\ \sup_{|c_{w,r,v}|\leq 1}\ \bignorm{ 
  \underset{0\leq w\leq D_r}{\E} \ c_{w,r,v}\  T^{\floor{ P_{1,r,v}(w) }} \tilde{f}_{r}\cdot ...\cdot T^{\floor{P_{k,r,v}(w) } }f_k                                       }_{L^2(\m)}^{2^t}.
    \end{multline*}For each $r\in \N$, let $v_r$ be the value of $v$ for which the above max is attained. Then, the polynomial family \begin{equation*}
      \mathcal{P}_r  =\{P_{1,r,v_r},...,P_{k,r,v_r}\}
    \end{equation*}is a nice polynomial family. Indeed, the degrees of its elements are fixed integers and the leading coefficients of the polynomials and of their differences are good sequences irrespective of the value of $v_r$, as we discussed previously. Therefore, under the above assumptions, we reduce our problem to
        \begin{multline}\label{finalreduction}
      \lim\limits_{R\to+\infty}\ \sup_{ ||h_1||_{\infty}\leq 1,...,||h_{\ell}||_{\infty}\leq 1}\
      \underset{1\leq r\leq R}{\E}\\ 
      \sup_{||f_2||_{\infty},...,||f_k||_{\infty} \leq 1}\ \sup_{|c_{w,r}|\leq 1}\  \bignorm{  
  \underset{0\leq w\leq D_r}{\E} \ c_{w,r}\  T^{\floor{ P_{1,r,v_r}(w) }} \tilde{f}_{r}\cdot ...\cdot T^{\floor{P_{k,r,v_r}(w) } }f_k                                       }_{L^2(\m)}^{2^t}=0.
    \end{multline}
  We also choose functions $h_{1,R},...,h_{\ell,R}\in L^{\infty}(\m)$ so that the corresponding average is $1/R$ close to the supremum of the $h_1,...,h_{\ell}$. Namely, we want to prove \eqref{finalreduction} where $f_r$ is now the function \begin{equation*}
      f_1\cdot  T^{\floor{\theta_1(r)}}h_{1,R}\cdot...
    \cdot T^{\floor{\theta_{\ell}(r)}}h_{\ell,R}.
  \end{equation*}
  
  \subsection{Step 6: Applying the polynomial bounds}

Now, we apply Proposition \ref{PET} for the inner average in the above relation. We have established that its hypotheses are satisfied. The degree and the type of the polynomial family all depend on the initial functions $a_1,...,a_k$. Therefore, all asymptotic bounds are assumed to depend only on $a_1,...,a_k$ and we omit the indices.

Let us denote the leading vector of the family $\mathcal{P}_r$ by $(u_{1,r},...,u_{k,r})$ and recall again here that each $u_{i,r}$ satisfies one of the following:\\
a) there exists a sub-linear function $\phi_i(r)\prec u(r)$ that dominates some fractional power, such that $u_{i,r}=\phi_i(r)+o_r(1)$,\\
b) they have the form  $c\floor{u(r)}^{\theta}$, where $c\in\R$ and $\theta$ is a positive integer or \\
c) they converge to a non-zero real number.

Fix a positive integer $M$. There exist integers $s,t$, a finite set $Y$ of integers and polynomials $p_{\ue,i}$ (all depending only on the original functions $a_1,...,a_k$), where $\ue\in [[s]]$ and $1\leq i\leq k$ such that \begin{multline}\label{afterpet}
 \sup_{||f_2||_{\infty}\leq 1,...,||f_k||_{\infty} \leq 1}\ \sup_{|c_{w,r}|\leq 1}\    \bignorm{  
  \underset{0\leq w\leq D_r}{\E} \ c_{w,r}\  T^{\floor{ P_{1,r,v_r}(w) }} \tilde{f}_{r}\cdot ...\cdot T^{\floor{P_{k,r,v_r}(w) } }f_k                                       }_{L^2(\m)}^{2^t}\ll\\ 
  \frac{1}{M}+
  \sum_{{\bf h}\in     Y^{[[s]]}}^{} \underset{{\bf m}\in [-M,M]^t}{\E} \Big| \int \prod_{\underline{\e}\in [[s]]}^{}  T^{\floor{A_{\ue,r}(\bm)}+h_{\underline{\e}}}(\mathcal{C}^{|\underline{\e}|}\tilde{f}_{r})  \ d\m  \Big|+
    o_r(1),
  \end{multline}
where \begin{equation}\label{A_e}
    A_{\ue,r}(\bm)=\sum_{1\leq j\leq k} \ p_{\underline{\e},j}({\bf m})u_{j,r}.
\end{equation}The polynomials $A_{\ue}$ are essentially distinct for any value of the $u_{j,r}$ and satisfy \begin{equation*}
    A_{\ue,r}(\bm)+A_{\ue^c,r}(\bm)=A_{\underline{1},r}(\bm).
\end{equation*}In addition, for an $\ue\in [[s]]$, we have that the non-zero polynomials among the $p_{\ue,j}$ are linearly independent.

Applying the bounds of \eqref{afterpet} to \eqref{finalreduction}, we deduce that our original average is bounded by the quantity  \begin{multline}\label{1940}
    \frac{1}{M}+ \sum_{{\bf h}\in     Y^{[[s]]}}^{} \underset{{\bf m}\in [-M,M]^t}{\E}\   \underset{1\leq r\leq R}{\E} \Big| \int \prod_{\underline{\e}\in [[s]]}^{}  T^{\floor{A_{\ue,r}(\bm)}+h_{\underline{\e}}}(\mathcal{C}^{|\underline{\e}|}\tilde{f}_{r})  \ d\m  \Big|+
    o_R(1)= \\
     \frac{1}{M}+ \sum_{{\bf h}\in     Y^{[[s]]}}^{} \underset{{\bf m}\in [-M,M]^t}{\E}\   \underset{1\leq r\leq R}{\E}
     \Big| \int \prod_{\underline{\e}\in [[s]]}^{}\prod_{0\leq i\leq \ell} T^{\floor{A_{\ue,r}(\bm)}+ \floor{\theta_{i}(r)}  +h_{\underline{\e}}}(\mathcal{C}^{|\underline{\e}|}h_{i,R})  \ d\m  \Big|+
    o_R(1),
\end{multline}where we set $\theta_0(r)\equiv 0$ and $h_{0,R}\equiv f_1$ for convenience in notation. We may assume without loss of generality that $0\equiv \theta_0(r)\ll \theta_1(r)\ll...\ll\theta_{\ell}(r)$. Then, we compose with $T^{-\floor{\theta_{\ell}(r)}}$ inside the above integral and combine the integer parts to obtain that the aforementioned integral is equal to \begin{equation*}
    \int \prod_{\underline{\e}\in [[s]]}^{}\prod_{0\leq i\leq \ell} T^{\floor{A_{\ue,r}(\bm)}+ \floor{\theta_{i}(r)-\theta_{\ell}(r)}+h_{i,r}  +h_{\underline{\e}}}(\mathcal{C}^{|\underline{\e}|}h_{i,R})  \ d\m  ,
\end{equation*}where $h_{i,r}\in \{0,\pm 1\}$. Putting this in \eqref{1940}, we want to bound 
\begin{multline*}
     \frac{1}{M}+ \sum_{{\bf h}\in     Y^{[[s]]}}^{} \underset{{\bf m}\in [-M,M]^t}{\E}\   \underset{1\leq r\leq R}{\E}
     \Big| \int \prod_{\underline{\e}\in [[s]]}^{}\prod_{0\leq i\leq \ell} T^{\floor{A_{\ue,r}(\bm)}+ \floor{\theta_{i}(r)-\theta_{\ell}(r)} +h_{i,r} +h_{\underline{\e}}}(\mathcal{C}^{|\underline{\e}|}h_{i,R})  \ d\m  \Big|+
    o_R(1).
\end{multline*}Using the argument present in Lemma \ref{errors}, we deduce that the last quantity is smaller than a constant multiple of \begin{multline*}\frac{1}{M}+ \sum_{{\bf h}\in     Y^{[[s]]}}^{} \underset{{\bf m}\in [-M,M]^t}{\E}\  \ \sup_{\norm{h_1}_{\infty},...,\norm{h_{\ell}}_{\infty} \leq 1   } \\\underset{1\leq r\leq R}{\E}
   \Big| \int \prod_{\underline{\e}\in [[s]]}^{}\prod_{0\leq i\leq \ell} T^{\floor{A_{\ue,r}(\bm)} + \floor{\theta_{i}(r)-\theta_{\ell}(r)} +h_{\underline{\e}}}(\mathcal{C}^{|\underline{\e}|}h_{i})  \ d\m  \Big|+
    o_R(1).
\end{multline*}We choose again sequences of functions in place of the $h_1,...,h_{\ell}$, so that the corresponding quantity is $1/R$ close to the supremum and we denote them again      $h_{1,R},...,h_{\ell,R}$ for convenience. Note that this final quantity is essentially has the same form as the one in \eqref{1940}, but the function $\theta_0$ corresponding to $f_1$ now has maximal growth rate among the $\theta_i$. Therefore, our original problem reduces to finding a bound for \begin{multline}\label{1973}
    \frac{1}{M}+ \sum_{{\bf h}\in     Y^{[[s]]}}^{} \underset{{\bf m}\in [-M,M]^t}{\E}\   \underset{1\leq r\leq R}{\E}
     \Big| \int \prod_{\underline{\e}\in [[s]]}^{}\prod_{0\leq i\leq \ell} T^{\floor{A_{\ue,r}(\bm)}+ \floor{\theta_{i}(r)}  +h_{\underline{\e}}}(\mathcal{C}^{|\underline{\e}|}h_{i,R})  \ d\m  \Big|+
    o_R(1)
\end{multline}under the assumption that $\theta_0(t)\gg \theta_{i}(t)\succ\log t$ for every $1\leq i\leq l-1$, $\theta_{\ell}\equiv 0$ and $\theta_i(t)-\theta_j(t)\succ \log t$ for all $ i\neq j$.

We write \begin{equation*}
    B_{\bm,\bh}(r):= 
     \Big| \int \prod_{\underline{\e}\in [[s]]}^{}\prod_{0\leq i\leq \ell} T^{\floor{A_{\ue,r}(\bm)}+ \floor{\theta_{i}(r)}  +h_{\underline{\e}}}(\mathcal{C}^{|\underline{\e}|}h_{i,R})  \ d\m  \Big|.
\end{equation*}Taking the limit as $R\to+\infty$, our goal is to show that the quantity \begin{equation*}
    \frac{1}{M}+  \sum_{{\bf h}\in     Y^{[[s]]}}^{} \underset{{\bf m}\in [-M,M]^t}{\E}  \big( \limsup\limits_{R\to+\infty}  \underset{1\leq r\leq R}{\E}  B_{\bm,\bh}(r)  \big).
\end{equation*}goes to 0, as $M$ goes to infinity.

\subsection{Step 7: Another change of variables trick}

Before we proceed with the final details of the proof, we will make a final trick to reduce our problem to a statement, where the results of Section \ref{sublinearsection} can be applied. We will use a lemma very similar to \cite[Lemma 5.1]{FraHardy1}, which can also be proven similarly by a standard partial summation argument. 

\begin{lemma}\label{floor}
Let $(V_R(n))_{n,R\in\N}$ be a 1-bounded, two-parameter sequence of vectors in a normed space and let $a\in\mathcal{H}$ satisfy the growth condition $t^{\delta}\prec a(t)\prec t$. Then, we have \begin{equation*}
   \limsup\limits_{R\to +\infty} \bignorm{\underset{1\leq n\leq R}{\E} V_R(\floor{a(n)}{})  }\ll_a  \limsup\limits_{R\to +\infty} \bignorm{ \underset{1\leq n\leq R}{\E} V_R(n) }.
\end{equation*}
\end{lemma}

Our main objective is the following: since the sequences $u_{j,r}$ of the leading vector can have the form $c\floor{u(r)}^{k}$, which are tough to handle, we want to use the above lemma to replace these terms with the terms $cr^k$, which are just polynomials. In order to facilitate this, we need to write the entire integral $B_{\bm,\bh}(r)$ as a function of $\floor{u(r)}$. Note that $u(r)$ satisfies the growth condition in the statement of Lemma \ref{floor}. 
We consider three cases:\\
i) If the sequence $u_{j,r}$ has the form $c\floor{u(r)}^q$, for $c\in\R$ and $q\in \N^{*}$, then it is already written as a function of $\floor{u(r)}$.\\
ii) If the sequence $u_{j,r}$ converges to a non-zero real number $a_j$, then, we have $u_{j,r}-a_j=o_r(1)$ and the constant function $a_j$ is already written as a function of $\floor{u(r)}$.\\
iii) Finally, assume the sequence $u_{j,r}$ satisfies the remaining possible condition, namely that there exists a function $\phi_j\in\mathcal{H}$ satisfying the growth condition \begin{equation*}
    t^{\delta}\prec \phi_j(t)\prec u(t)
\end{equation*}for some $\delta>0$ and such that $$u_{j,r}=\phi_j(r)+o_r(1).$$ Let us assume that $\phi_j(t)$ is eventually positive (in the other case, we work with the number $-u_{j,r}$).
We write $\phi_j(t)=\Phi_j(u(t))$, where $\Phi_j= \phi_j\circ u^{-1}$, which is well defined and thus a function in $\mathcal{H}$ \footnote{Note that $u(t)$ is a positive function by its definition and therefore, goes to $+\infty$. Consequently, $u^{-1}$ also goes to $+\infty$.  }. We also have that $\Phi_j(t)\prec t $ (this follows easily from the fact that $\phi_j(t)\prec u(t)$) and we can easily see that $\Phi_j(t)$ also dominates some fractional power. In addition, we have \begin{equation*}
    |\Phi_j(u(t))-\Phi_j(\floor{u(t)})|\leq \sup_{x\in\R, \floor{u(t)}\leq x\leq u(t)}|\Phi'_j(t)|=o_t(1),
\end{equation*}since $\Phi'_j(t)\ll \Phi_j(t)/t\prec 1$.

  In all three cases above, we have the following: there exists a function $w_j\in\mathcal{H}$, such that \begin{equation}\label{wdefinition}
      |u_{j,r}-w_j(\floor{u(r)})|=o_r(1)
  \end{equation}and the function $w_j$ is either a monomial, or a constant function or a sub-linear (but not a sub-fractional) function. We write \begin{equation}\label{A'}
      \tilde{A}_{\ue,r }(\bm)=\sum_{1\leq j\leq k} \ p_{\underline{\e},j}({\bf m})w_{j}(\floor{u(r)})
  \end{equation}and observe that $|A_{\ue,r}(\bm)-\tilde{A}_{\ue,r}(\bm)|=o_r(1)$, for any fixed value of $\bm$.  Therefore, for $r$ large enough, we have \begin{equation}\label{asdf1}
      \floor{A_{\ue,r}(\bm)}=\floor{\tilde{A}_{\ue,r}(\bm)}+h'_{r,\ue,\bm},
  \end{equation}where $h'_{r,\ue,\bm}\in\{0,\pm 1\}$.
  
  We do the same for the function $\theta_i$. Indeed, we can use the same arguments as above to deduce that $|\theta_i(t)-\psi_i(\floor{u(t)})|=o_t(1)$, where $\psi_i(t)\in\mathcal{H}$ is the function $\theta_i\circ u^{-1}$ In addition, since $u$ dominates some fractional power, we have that $u^{-1}$ has polynomial growth and, therefore, we easily get $t^{\e}\succ \psi_i(t)\succ \log t$ for all $\e>0$, that is $\psi_i$ is a (sub-fractional) function. Finally, for $r$ large enough, we can write \begin{equation}\label{asdf2}
      \floor{\theta_i(r)}=\floor{\psi_i(u(r))}+h''_{i,r},
  \end{equation}where $h''_{i,r}\in \{0,\pm 1\}$.

     In view of the above, we have \begin{multline*}
     \ \ \  \sum_{{\bf h}\in     Y^{[[s]]}}^{} \underset{{\bf m}\in [-M,M]^t}{\E}  \  \underset{1\leq r\leq R}{\E}\  B_{\bm,\bh}(r)  = \\
      \sum_{{\bf h}\in     Y^{[[s]]}}^{} \underset{{\bf m}\in [-M,M]^t}{\E}  \  \underset{1\leq r\leq R}{\E}  \Big| \int \prod_{\underline{\e}\in [[s]]}^{}\prod_{0\leq i\leq \ell} T^{\floor{\tilde{A}_{\ue,r}(\bm)}+h'_{r,\ue,\bm}+ \floor{\psi_{i}(\floor{u(r)})}+h'' _{i,r} +h_{\underline{\e}}}(\mathcal{C}^{|\underline{\e}|}h_{i,R})  \ d\m  \Big| +o_R(1)  \leq \\
       \ \ \ \  \sum_{{\bf h}\in     Y^{[[s]]}}^{} \underset{{\bf m}\in [-M,M]^t}{\E}  \ \Big( \underset{1\leq r\leq R}{\E} \Big| \int \prod_{\underline{\e}\in [[s]]}^{}\prod_{0\leq i\leq \ell} T^{\floor{\tilde{A}_{\ue,r}(\bm)}+h'_{r,\ue,\bm}+ \floor{\psi_{i}(\floor{u(r)})}+h'' _{i,r} +h_{\underline{\e}}}(\mathcal{C}^{|\underline{\e}|}h_{i,R})  \ d\m  \Big|^2     \big)^{1/2}\\
       +o_R(1), 
  \end{multline*}where we applied the Cauchy-Schwarz inequality (the $o_R(1)$ term on the second line exists to account for small values of $r$ for which \eqref{asdf1},\eqref{asdf2} may not hold with error terms in the set $\{0,\pm 1\}$). Thus, we want to bound \begin{multline}\label{the average}
      \frac{1}{M}+ \sum_{{\bf h}\in     Y^{[[s]]}}^{} \underset{{\bf m}\in [-M,M]^t}{\E}  \\ \Big( \underset{1\leq r\leq R}{\E} \Big| \int \prod_{\underline{\e}\in [[s]]}^{}\prod_{0\leq i\leq \ell} T^{\floor{\tilde{A}_{\ue,r}(\bm)}+h'_{r,\ue,\bm}+ \floor{\psi_{i}(\floor{u(r)})}+h'' _{i,r} +h_{\underline{\e}}}(\mathcal{C}^{|\underline{\e}|}h_{i,R})  \ d\m  \Big|^2     \big)^{1/2}
      +o_R(1),
  \end{multline}where $h_{0,R}=f_1$. 
  \begin{claim}
      Proposition \ref{factors} holds in the case when all the functions $w_j$ (defined in \eqref{wdefinition}) are constant and $\ell=0$.
  \end{claim}
  
   \begin{proof}[Proof of the claim]
    This means that the polynomials $\tilde{A}_{\ue,r}(\bm)$ are actually independent of $r$ and we write them as $\tilde{A}_{\ue}(\bm)$. In addition, there are no functions $\psi_i$ in the iterates of the above quantity. Finally, the error terms $h''_{i,r}$ do not exist in this case. Our problem reduces to finding a bound for \begin{equation}\label{bnm}
      \frac{1}{M}+ \sum_{{\bf h}\in     Y^{[[s]]}}^{} \underset{{\bf m}\in [-M,M]^t}{\E}  \big( \underset{1\leq r\leq R}{\E}\   \Big| \int \prod_{\underline{\e}\in [[s]]}^{} T^{\floor{\tilde{A}_{\ue}(\bm)}+h'_{r,\ue,\bm} +h_{\underline{\e}}}(\mathcal{C}^{|\underline{\e}|}f_1)  \ d\m  \Big|^2     \big)^{1/2}\\
      +o_R(1),
  \end{equation}where $h'_{r,\ue,\bm}\in\{0,\pm 1\}$. Note that  \begin{multline*}
      \underset{1\leq r\leq R}{\E}  \Big| \int \prod_{\underline{\e}\in [[s]]}^{} T^{\floor{\tilde{A}_{\ue}(\bm)}+h'_{r,\ue,\bm} +h_{\underline{\e}}}(\mathcal{C}^{|\underline{\e}|}f_1)  \ d\m  \Big|^2\leq \\
      \sum_{h'_{\underline{\e}}\in \{0,\pm 1\},\e\in [[s]] } \Big| \int \prod_{\underline{\e}\in [[s]]}^{}T^{\floor{\tilde{A}_{\ue}(\bm)}+h'_{\ue} +h_{\underline{\e}}}(\mathcal{C}^{|\underline{\e}|}f_1)  \ d\m  \Big|^2,
  \end{multline*}which implies that the quantity in \eqref{bnm} is smaller than $O(1)$ times \begin{equation*}
       \frac{1}{M}+ \sum_{{\bf h}\in     \tilde{Y}^{[[s]]}}^{} \underset{{\bf m}\in [-M,M]^t}{\E}    \Big| \int \prod_{\underline{\e}\in [[s]]}^{} T^{\floor{\tilde{A}_{\ue}(\bm)} +h_{\underline{\e}}}(\mathcal{C}^{|\underline{\e}|}f_1)  \ d\m  \Big|\\
      +o_R(1)
  \end{equation*}for some new, larger finite set $\tilde{Y}$. The statement follows if we prove that \begin{equation*}
      \lim\limits_{M\to+\infty}\underset{{\bf m}\in [-M,M]^t}{\E}   \Big| \int \prod_{\underline{\e}\in [[s]]}^{} T^{\floor{\tilde{A}_{\ue}(\bm)} +h_{\underline{\e}}}(\mathcal{C}^{|\underline{\e}|}f_1)  \ d\m  \Big|   =0
  \end{equation*}for any $h_{\ue}\in \Z$. Note that the polynomials $\tilde{A}_{\ue}(\bm)$ are essentially distinct due to the statement of Proposition \ref{PET}. Squaring and applying the Cauchy-Schwarz inequality, we want to prove that \begin{equation*}
       \lim\limits_{M\to+\infty}\underset{{\bf m}\in [-M,M]^t}{\E}   \Big| \int \prod_{\underline{\e}\in [[s]]}^{} T^{\floor{\tilde{A}_{\ue}(\bm)} +h_{\underline{\e}}}(\mathcal{C}^{|\underline{\e}|}f_1)  \ d\m  \Big|^2   =0,
  \end{equation*}which can be rewritten as \begin{equation*}
      \lim\limits_{M\to+\infty}\underset{{\bf m}\in [-M,M]^t}{\E}    \int \prod_{\underline{\e}\in [[s]]}^{} S^{\floor{\tilde{A}_{\ue}(\bm)} +h_{\underline{\e}}}(\mathcal{C}^{|\underline{\e}|}F_1)  \ d(\m\times\m)  =0,
  \end{equation*}where $S=T\times T$ and $F_1=\overline{f_1}\otimes f_1$. This is an average where the iterates are real polynomials and using \cite[Lemma 4.3]{Frajointhardy}, we can prove that this last relation holds, provided that $\nnorm{ S^{h_{\underline{1} }}F_1 }_{\tilde{s},T\times T}=0$, for some positive integer $\tilde{s}$ that depends only on the polynomials $A_{\ue}$ (which depend on the original Hardy field functions $a_1,...,a_k$). However, since $\nnorm{F_1}_{\tilde{s},T\times T}\leq \nnorm{f_1}_{\tilde{s}+1,T}^2$, we get that the statement holds if the function $f_1$ satisfies $\nnorm{f_1}_{\tilde{s}+1,T}=0$. This completes the proof of our claim.
  \end{proof}

  From now on, we assume that either at least one of the functions $w_j$ is non-constant, or that $\ell\geq 1$ and we want to bound the quantity in \eqref{the average}.
   Writing $H_{i,R}=\overline{h_{i,R}}\otimes h_{i,R}$ and $S=T\times T$, we observe that \begin{align*}
    &\underset{1\leq r\leq R}{\E}  \Big| \int \prod_{\underline{\e}\in [[s]]}^{}\prod_{0\leq i\leq \ell} T^{\floor{\tilde{A}_{\ue,r}(\bm)}+h'_{r,\ue,\bm}+ \floor{\psi_{i}(\floor{u(r)})}+h'' _{i,r} +h_{\underline{\e}}}(\mathcal{C}^{|\underline{\e}|}h_{i,R})  \ d\m  \Big|^2 = \\
    &\underset{1\leq r\leq R}{\E}  \int \prod_{\underline{\e}\in [[s]]}^{}\prod_{0\leq i\leq \ell} S^{\floor{\tilde{A}_{\ue,r}(\bm)}+h'_{r,\ue,\bm}+ \floor{\psi_{i}(\floor{u(r)})}+h'' _{i,r} +h_{\underline{\e}}}(\mathcal{C}^{|\underline{\e}|}H_{i,R})  \ d(\m\times \m)\leq \\
    \bignorm{    &\underset{1\leq r\leq R}{\E}  \prod_{\underline{\e}\in [[s]]}^{}\prod_{0\leq i\leq \ell} S^{\floor{\tilde{A}_{\ue,r}(\bm)}+h'_{r,\ue,\bm}+ \floor{\psi_{i}(\floor{u(r)})}+h'' _{i,r} +h_{\underline{\e}}}(\mathcal{C}^{|\underline{\e}|}H_{i,R})  }_{L^2(\m\times \m)}
     \end{align*}due to the Cauchy-Schwarz inequality.
     Invoking\footnote{Note that all the error terms depending on $r$ in the iterates take values on finite sets.} Lemma \ref{errors}, we have \begin{multline*}
          \bignorm{    \underset{1\leq r\leq R}{\E}  \prod_{\underline{\e}\in [[s]]}^{}\prod_{0\leq i\leq \ell} S^{\floor{\tilde{A}_{\ue,r}(\bm)}+h'_{r,\ue,\bm}+ \floor{\psi_{i}(\floor{u(r)})}+h'' _{i,r} +h_{\underline{\e}}}(\mathcal{C}^{|\underline{\e}|}H_{i,R})  }_{L^2(\m\times \m)}\ll_{s,\ell} \\
           \sup_{|c_{r,\bm,\bh}|\leq 1}\sup_{{\norm{H_{i}}_{\infty}\leq 1} } \bignorm {  \underset{1\leq r\leq R}{\E}  c_{r,\bm,\bh} \prod_{\underline{\e}\in [[s]]}^{}\prod_{0\leq i\leq \ell} S^{\floor{\tilde{A}_{\ue,r}(\bm)}+\floor{\psi_{i}(\floor{u(r)})}+h_{\e}}(\mathcal{C}^{|\underline{\e}|}H_{i})  }_{L^2(\m\times \m)},
     \end{multline*}where $ H_{0}=\overline{f_1}\otimes f_1$ and $\bh=(h_{\ue},\ue\in[[s]])$. Note that since both $s,\ell$ depend on the original Hardy field functions $a_1,...,a_k$, the implicit constant in the last bound depends only on $a_1,...,a_k$ (which we omit from the subscripts).

     Putting everything together, we get that \begin{multline*}
        \frac{1}{M}+ \sum_{{\bf h}\in     Y^{[[s]]}}^{} \underset{{\bf m}\in [-M,M]^t}{\E}  \  \underset{1\leq r\leq R}{\E}\  B_{\bm,\bh}(r) +o_R(1) \\ \ll \frac{1}{M}+
          \sum_{{\bf h}\in     Y^{[[s]]}}^{} \underset{{\bf m}\in [-M,M]^t}{\E}  \sup_{|c_{r,\bm,\bh}|\leq 1}\ \sup_{\underset{1\leq i\leq \ell}{\norm{H_{i}}_{\infty}\leq 1}}\\ \bignorm {  \underset{1\leq r\leq R}{\E}  c_{r,\bm,\bh} \prod_{\underline{\e}\in [[s]]}^{}\prod_{0\leq i\leq \ell} S^{\floor{\tilde{A}_{\ue,r}(\bm)}+\floor{\psi_{i}(\floor{u(r)})}+h_{\e}}(\mathcal{C}^{|\underline{\e}|}H_{i})  }_{L^2(\m\times \m)}^{1/2}+o_R(1).
          \end{multline*}
          Now, we choose functions $H_{1,R},...,H_{\ell,R}$ so that the above average (over $R$) is $1/R$ close to the supremum. Then, we take the limit as $R\to+\infty$ and apply Lemma \ref{floor}
 to deduce that the limsup of this last quantity is bounded by $O_{u}(1)$ times (which is, of course, $O_{a_1,...,a_k}(1)$)\begin{multline*}
              \sum_{{\bf h}\in     Y^{[[s]]}}^{}\underset{{\bf m}\in [-M,M]^t}{\E} \limsup\limits_{R\to+\infty}
            \sup_{|c_{r,\bm,\bh}|\leq 1}\\ \bignorm {  \underset{1\leq r\leq R}{\E}  c_{r,\bm,\bh} \prod_{\underline{\e}\in [[s]]}^{}\prod_{0\leq i\leq \ell} S^{\floor{\widehat{A}_{\ue,r}(\bm)}+\floor{\psi_{i}(r)}+h_{\ue}}(\mathcal{C}^{|\underline{\e}|}H_{i,R})  }_{L^2(\m\times \m)}^{1/2},
 \end{multline*}where  we define (recall \eqref{A'}) \begin{equation*}
     \widehat{A}_{\ue,r}(\bm):= \sum_{1\leq j\leq k} \ p_{\underline{\e},j}({\bf m})w_{j}(r).
 \end{equation*}and $H_{0,R}=\overline{f_1}\otimes f_1$.
  Finally, we can combine the integer parts in the iterates of the above quantity
 (using again Lemma \ref{errors} to remove the error terms). In conclusion, our original average is bounded by $O(1)$ times\begin{multline}\label{Step7average}
     \frac{1}{M} +  \underset{{\bf m}\in [-M,M]^t}{\E} \Big( \limsup\limits_{R\to+\infty}\sup_{|c_{r,\bm}|\leq 1}
            \bignorm {  \underset{1\leq r\leq R}{\E} c_{r,\bm} \prod_{\underline{\e}\in [[s]]}^{}\prod_{0\leq i\leq \ell} S^{\floor{\widehat{A}_{\ue,r}(\bm)+\psi_{i}(r)}}(\mathcal{C}^{|\underline{\e}|}H_{i,R})  }_{L^2(\m\times \m)}^{1/2}\Big)\leq\\
           \ \ \ \frac{1}{M}+\Big(\  \underset{{\bf m}\in [-M,M]^t}{\E} \limsup\limits_{R\to+\infty}\ \sup_{|c_{r,\bm}| \leq 1}\ 
          \bignorm {  \underset{1\leq r\leq R}{\E} c_{r,\bm} \prod_{\underline{\e}\in [[s]]}^{}\prod_{0\leq i\leq \ell} S^{\floor{\widehat{A}_{\ue,r}(\bm)+\psi_{i}(r)} }(\mathcal{C}^{|\underline{\e}|}H_{i,R})  }_{L^2(\m\times \m)}   \Big)^{1/2}
 \end{multline}by the Cauchy-Schwarz inequality. Note that all implied asymptotic constants above did not depend on either $M$ or $R$.

  \subsection{Finishing the proof}
  
We describe the final step here. Our main observation is that $\widehat{A}_{\ue,r}(\bm)+\psi_i(r)$, when viewed as a function of $r$, is a sum of sub-linear functions that dominate the function $\log r$ and monomials (possibly of degree 0). Our goal is to use the bounds in Proposition \ref{sublinearseminorm} to deduce our result. However, it is not immediately obvious that in our case a linear combination of functions of the above form dominates the logarithmic function $\log r$ (the statement in general is false and a counterexample is given by the pair $(\log^2 t+\log t,\log^2 t)$). We shall establish that this is true for all $\bm\in \Z^t$ outside a negligible set. We recall here that for every large enough $r$ (large enough for $w_j(r)$ to be non-zero), the $\widehat{A}_{\ue,r}(\bm)$ are pairwise essentially distinct polynomials in the variable $\bm$ and in addition satisfy \begin{equation*}
      \widehat{A}_{\ue,r}(\bm)+\widehat{A}_{\ue^c,r}(\bm)=\widehat{A}_{\underline{1},r}(\bm).
  \end{equation*}

  We will use the following lemma:
  \begin{lemma}\label{rootdensity}
  Let $p\in \R^t({\bf x})$ be a non-zero real polynomial of degree $d$. Then, the set of integer solutions of the equation \begin{equation*}
     p({\bf m})=0
 \end{equation*}in $[-M,M]^t$ has $O_{d}(  M^{t-1})$ elements.
 \end{lemma}
  
  \begin{proof}
  For $t=1$ it is obvious, since the polynomial has at most $d$ roots. Assume we have proven the result for $t-1$. We can write $p(\bm)$ in the form \begin{equation*}
      p(m_1,...,m_t)=a_{d'}(m_1,...,m_{t-1}) m_t^{d'}+\cdots+ a_1(m_1,...,m_{t-1} )m_t+a_0(m_1,....,m_{t-1}           )   
  \end{equation*}for some $d'\leq d$. At least one of the polynomials $a_i(m_1,...,m_{t-1})$ with $1\leq i\leq d'$ is not identically zero and thus has at most $O_{d,t}(M^{t-2})$ zeroes in $[-M,M]^{t-1}$. If $(x_1,...,x_{t-1})$ is not one of these zeroes, then $p(x_1,...,x_{t-1}, m_t )$ is non-trivial as a polynomial in the variable $m_t$. Therefore, it is satisfied by no more than $d$ values of $m_t$. Summing over all tuples $(m_1,...,m_{t-1})\in [-M,M]^{t-1}$, we get the result.
 \end{proof}

  \begin{corollary}\label{linearcombhardy}
  Let $a_1\ll...\ll a_k$ be functions in $\mathcal{H}$ and let $p_1(\bm),...,p_k(\bm)\in  \R^t({\bf x})$ be non-zero linearly independent polynomials. Then, for all $\bm \in \Z^t$ outside a set of density 0, we have that \begin{equation}\label{add}
      p_1(\bm)a_1+\cdots+p_k(\bm)a_k\sim a_k.
  \end{equation}
  \end{corollary}
  
  \begin{proof}
  Let $a_{k_0},...,a_k$ be the functions among the $a_i$ that have the same growth rate as $a_k$. Then, for $k_0\leq j\leq k$, we can write $a_j(t)=c_ja_k(t)+b_j(t)$, where $c_j\in \R^{*}$ and $b_j(t)\prec a_k(t)$. Then, the function in \eqref{add} has the same growth rate as the function \begin{equation*}
      \big(c_{k_0}p_{k_0}(\bm)+\cdots+c_k p_k(\bm)\big)a_k(t)
  \end{equation*}unless of course $c_{k_0}p_{k_0}(\bm)+\cdots +c_k p_k(\bm)=0$. However, the linear independence hypothesis implies that this polynomial is non-zero, and thus the set of of $\bm\in \Z^t$ for which this last relation holds has density 0 in $\Z^t$ by Lemma \ref{rootdensity}. The conclusion follows.
  \end{proof}
  
  We use this corollary to prove the following:
  
  \begin{claim}
      For all $\bm\in \Z^{t}$ outside a set $\Lambda$ of density 0, we have that the functions (in the variable $r$) \begin{equation*}
          \widehat{A}_{\ue,r}(\bm)+\psi_i(r)=\sum_{1\leq j\leq k} \ p_{\underline{\e},j}({\bf m})w_{j}(r)+\psi_i(r)
      \end{equation*}are a sum of a sub-linear function and a real polynomial. In addition, we have that they either dominate the function $\log r$, or they are a constant function.
  \end{claim}
  
  \begin{proof}[Proof of the claim]

   We use Corollary \ref{linearcombhardy} to find a set $\Lambda\subset \Z^t$ of density zero, so that for $\bm\notin \Lambda$, we have that for any $\ue\in [[s]] $ and any subcollection $J$ of the indices $j\in \{1,2,...,k\}$, we have that \begin{equation*}
       \sum_{j\in J} p_{\ue,j}(\bm) w_j(r)\sim w_{\max(J)}(r),
   \end{equation*}where $w_{\max(J)}$ denotes a function in the collection $\{w_j,j\in J\}$ that has maximal growth rate. We show that this set $\Lambda$ is sufficient for the statement of the claim to hold.

   We split the $w_j$ into two sets: the set $S_1$ consists of those functions that are monomials, while $S_2$ contains the rest (namely the sub-linear functions). Reordering, if necessary, we may assume that $S_1=\{w_1,...,w_{k_0}\}$ while $S_2=\{w_{k_0+1},...,w_{k}\}$. We write \begin{equation}\label{splitting}
       \widehat{A}_{\ue,r}(\bm)=\sum_{j=1}^{k_0}p_{\ue,j}(\bm)w_j(r)+\sum_{j=k_0+1}^{k}p_{\ue,j}(\bm)w_j(r).
   \end{equation}For a fixed $\bm\notin \Lambda$, the first summand is a polynomial in the variable $r$ (possibly constant), while the second is a sub-linear function of $r$. Since the sub-linear functions $w_j$ with  $ k_0+1\leq j\leq k $ dominate some fractional power, we deduce that $\widehat{A}_{\ue,r}(m)$ is either a constant function \footnote{ This is the case when $p_{\ue,j}(\bm)\equiv 0$ for $j\geq    k_0+1$ and the monomials $w_j$ are constant polynomials in the variable $r$.}, or the sum of a polynomial and a sub-linear function that dominates some fractional power, since \begin{equation*}
       \sum_{j=k_0+1}^{k}p_{\ue,j}(\bm)w_j(r)\sim w_{\max\{k_0+1,...,k\}}(r)
   \end{equation*}where $w_{\max(S)}$ for $S\subseteq\{1,...,k\}$ is defined above and this is a sub-linear (but not sub-fractional) function.
   
   In addition, if $\psi_i(t)\neq\psi_{\ell}(t)$ (recall that $\psi_{\ell}(t)\equiv 0$), we can use the same argument to show that \begin{equation*}
       \widehat{A}_{\ue,r}(\bm)+\psi_i(r)
   \end{equation*} is a sum of a sub-linear function that dominates $\log r$ and a polynomial (we use the fact that $\psi_i$ and $w_j$ (for any $j$) have distinct growth rates, since the $\psi_i$ is a sub-fractional function.)
  \end{proof}
  
  Let $\Lambda\subset \Z^{t}$ be the zero density set given by the above claim. 
  Now, we isolate the iterate $S^{\floor{\widehat{A}_{\underline{1},r}(\bm)+\psi_0(t) }}(\mathcal{C}^{|\underline{1}|}H_{0})$ in \eqref{Step7average} and we also assume that $\bm \notin \Lambda$. The above proof implies that the Hardy field function involved in this iterate is a sum of a sub-linear function (that dominates the logarithm) and a polynomial. In order to apply the results of Section \ref{sublinearsection}, we have to show that the differences of this function with the rest of the functions in the iterates satisfies the same condition. That is, for every $(\ue,i)\neq (\underline{1},0)$, we have to show that the function \begin{equation*}
     \big( \widehat{A}_{\underline{1},r}(\bm)+\psi_0(r) \big)- \big( A_{\ue,r}(\bm)+\psi_i(r)          \big)
  \end{equation*}is a sub-linear function plus a polynomial, or is bounded. Rewrite the above as \begin{equation*}
      \widehat{A}_{\ue^c,r}(\bm)+(\psi_0(r)-\psi_i(r)).
  \end{equation*} 
  
  If $i\neq 0$, then we use the fact that $\psi_0-\psi_i\succ \log t$ and the argument of the previous proof to establish that \begin{equation*}
       \big( \widehat{A}_{\underline{1},r}(\bm)+\psi_0(r) \big)- \big( \widehat{A}_{\ue,r}(\bm)+\psi_i(r)          \big)\succ \log r
  \end{equation*}for all $\bm$ outside a zero density set (which we attach to the set $\Lambda$) and that this function is the sum of a polynomial and a sub-linear function.
  
  If $i=0$, then the above difference is equal to $\widehat{A}_{\ue^c,r}(\bm)$ which is either the sum of a polynomial and a sub-linear function (that dominates $\log r$), or a constant function of $r$. We use this characterization to split $[[s]]$ into two subsets: $A_2$ contains those $\ue\in[[s]]$, for which $\widehat{A}_{\ue^c,r}(\bm)$ satisfies the first condition, while the set $A_1$ contains the rest. Note that if $\ue\in A_1$, then the difference \begin{equation*}
       \big( \widehat{A}_{\underline{1},r}(\bm)+\psi_0(r) \big)- \big( \widehat{A}_{\ue,r}(\bm)+\psi_0(r)          \big)
  \end{equation*}is a (non-constant) polynomial in the variable $\bm$ and we denote it by $c_{\ue}(\bm)$. Thus, we can write \begin{equation*}
      \big( \widehat{A}_{\ue,r}(\bm)+\psi_0(r)          \big)= \big( \widehat{A}_{\underline{1},r}(\bm)+\psi_0(r) \big) -c_{\ue}(\bm).
  \end{equation*}Note that the polynomials $c_{\ue}(\bm)$ are essentially distinct, since the $\widehat{A}_{\ue,r}$ are essentially distinct.

  In view of the above, we rewrite the quantity in \eqref{Step7average} as \begin{multline}\label{kos}
      \frac{1}{M}+\Big(\  \underset{{\bf m}\in [-M,M]^t}{\E} \limsup\limits_{R\to+\infty}\ \sup_{|c_{r,\bm}| \leq 1}\ 
          \bignorm {  \underset{1\leq r\leq R}{\E} c_{r,\bm} \prod_{\underline{\e}\in A_1}S^{\floor{\widehat{A}_{\underline{1},r}(\bm)+\psi_0(r)  -c_{\ue}(\bm)}          }(\mathcal{C}^{|\ue|} H_0)\\
          \prod_{\underline{\e}\in A_2} S^{\floor{\widehat{A}_{{\ue},r}(\bm)+\psi_0(r)}}(\mathcal{C}^{|\ue|} H_0)\prod_{1\leq i\leq \ell} \prod_{\ue\in[[s]]}
         S^{\floor{\widehat{A}_{\ue,r}(\bm)+\psi_{i}(r)} }(\mathcal{C}^{|\underline{\e}|}H_{i,R})  }_{L^2(\m\times \m)}   \Big)^{1/2}.
  \end{multline}
  Note that \begin{equation*}
      \floor{\widehat{A}_{\underline{1},r}(\bm)+\psi_0(r)  -c_{\ue}(\bm)}=\floor{\widehat{A}_{\underline{1},r}(\bm)+\psi_0(r)}+\floor{-c_{\ue}(\bm)}+h_{\ue,r,\bm},
  \end{equation*}
  where $h_{\ue,r,\bm}\in \{0,\pm 1\}$. Thus, we rewrite \eqref{kos} as \begin{multline}
        \frac{1}{M}+\Big(\  \underset{{\bf m}\in [-M,M]^t}{\E} \limsup\limits_{R\to+\infty}\ \sup_{|c_{r,\bm}| \leq 1}\ 
          \bignorm {  \underset{1\leq r\leq R}{\E} c_{r,\bm}\ S^{\floor{\widehat{A}_{\underline{1},r}(\bm)+\psi_0(r)}         }
          (\prod_{\underline{\e}\in A_1} \mathcal{C}^{|\ue|} S^{\floor{-c_{\ue}(\bm)} +h_{\ue,r,\bm}}H_0)\\
          \prod_{\underline{\e}\in A_2} S^{\floor{\widehat{A}_{\underline{\e},r}(\bm)+\psi_0(r)}}(\mathcal{C}^{|\ue|} H_0)\prod_{1\leq i\leq \ell} \prod_{\ue\in[[s]]}
         S^{\floor{\widehat{A}_{\ue,r}(\bm)+\psi_{i}(r)} }(\mathcal{C}^{|\underline{\e}|}H_{i,R})  }_{L^2(\m\times \m)}   \Big)^{1/2}.
  \end{multline}Since $h_{\ue,r,\bm}$ take values in $\{0,\pm 1\}$, we can use the argument in Lemma \ref{errors} to deduce that \begin{multline*}
       \bignorm {  \underset{1\leq r\leq R}{\E} c_{r,\bm}\ S^{\floor{\widehat{A}_{\underline{1},r}(\bm)+\psi_0(r)}         }
          (\prod_{\underline{\e}\in A_1} \mathcal{C}^{|\ue|} S^{\floor{-c_{\ue}(\bm)} +h_{\ue,r,\bm}}H_0)\\
          \prod_{\underline{\e}\in A_2} S^{\floor{\widehat{A}_{\underline{\e},r}(\bm)+\psi_0(r)}}(\mathcal{C}^{|\ue|} H_0)\prod_{1\leq i\leq \ell} \prod_{\ue\in[[s]]}
         S^{\floor{\widehat{A}_{\ue,r}(\bm)+\psi_{i}(r)} }(\mathcal{C}^{|\underline{\e}|}H_{i,R})  }_{L^2(\m\times \m)} \leq \\
         \sum_{\underset{\ue\in A_1}{h_{\ue,\bm}\in\{0,\pm 1\}}}\ \sup_{|c'_{r,\bm}|\leq 1}\ \bignorm {  \underset{1\leq r\leq R}{\E} c'_{r,\bm}\ S^{\floor{\widehat{A}_{\underline{1},r}(\bm)+\psi_0(r)}         }
          (\prod_{\underline{\e}\in A_1} \mathcal{C}^{|\ue|} S^{\floor{-c_{\ue}(\bm)} +h_{\ue,\bm}}H_0)\\
          \prod_{\underline{\e}\in A_2} S^{\floor{\widehat{A}_{\underline{\e},r}(\bm)+\psi_0(r)}}(\mathcal{C}^{|\ue|} H_0)\prod_{1\leq i\leq \ell} \prod_{\ue\in[[s]]}
         S^{\floor{\widehat{A}_{\ue,r}(\bm)+\psi_{i}(r)} }(\mathcal{C}^{|\underline{\e}|}H_{i,R})  }_{L^2(\m\times \m)}.
  \end{multline*} Thus, our problem reduces to showing that \begin{multline}\label{final}
      \frac{1}{M}+\Big(\  \underset{{\bf m}\in [-M,M]^t}{\E} \limsup\limits_{R\to+\infty}\ \sup_{|c_{r,\bm}| \leq 1}\ 
          \bignorm {  \underset{1\leq r\leq R}{\E} c_{r,\bm}\ S^{\floor{\widehat{A}_{\underline{1},r}(\bm)+\psi_0(r)}         }
          (\prod_{\underline{\e}\in A_1} \mathcal{C}^{|\ue|} S^{\floor{-c_{\ue}(\bm)} +h_{\ue,\bm}}H_0)\\
          \prod_{\underline{\e}\in A_2} S^{\floor{\widehat{A}_{\underline{\e},r}(\bm)+\psi_0(r)}}(\mathcal{C}^{|\ue|} H_0)\prod_{1\leq i\leq \ell} \prod_{\ue\in[[s]]}
         S^{\floor{\widehat{A}_{\ue,r}(\bm)+\psi_{i}(r)} }(\mathcal{C}^{|\underline{\e}|}H_{i,R})  }_{L^2(\m\times \m)}   \Big)^{1/2}
  \end{multline}goes to $0$ as $M\to+\infty$ (that is, our error terms in the iterates do not depend on $r$ now).
  
  In order to be able to apply Proposition \ref{sublinearseminorm}, we need to check that the degree, type and size (as defined in the beginning of Section \ref{sublinearsection}) of the given collection of functions in the iterates is constant, as $\bm$ ranges over $\Z^t$ (so that we can use bounds that are uniform in the variable $\bm$). Recall \eqref{splitting}: the "polynomial component" of $\widehat{A}_{\ue,r}(\bm) +\psi_i(r)$ is \begin{equation*}
      \sum_{j=1}^{k_0}p_{\ue,j}(\bm)w_j(r),
  \end{equation*}where the functions $w_j(r)$ are polynomials. The conclusion follows easily: indeed, for any two real polynomials $p_1(\bm)$ and $p_2(\bm)$ we must have that they are either equal for all $\bm$, or the set of integer solutions of $p_{1}(\bm)=p_{2}(\bm)$ has density zero. Comparing coefficients, it is straightforward to see that outside a set $\Lambda'$ of density zero, the degree, type and size of the collection of functions in the iterates in \eqref{final} is independent of $\bm$ for any $\bm \notin \Lambda$ (and they all depend only on the initial Hardy field functions $a_1,...,a_k$). In addition, the elements of the leading vector of this collection are polynomials in $\bm$ (we are not concerned with their actual form here). Therefore, we write the leading vector as $(u_1(\bm),...,u_{s_0}(\bm))$, where $s_0\leq s$ is the size of the given collection of functions, which does not depend on $\bm$ outside our "negligible" set. Furthermore, for $\bm$ outside a set of density zero (which we attach to the set $\Lambda'$), we have that all the numbers $u_1(\bm),...,u_{s_0}(\bm)$ are non-zero, and thus we can now apply Proposition \ref{sublinearseminorm} for all $\bm$ outside a negligible subset of $\Z^t$.

  Write ${\bf h}_{\bm}:=(h_{\ue,\bm},\ue\in A_1)$ and $$F_{\bm,{\bf h}_{\bm}}:=\prod_{\underline{\e}\in A_1} \mathcal{C}^{|\ue|} S^{\floor{-c_{\ue}(\bm)} +h_{\ue,\bm}}H_0.$$
  Now, for any $\bm\notin \Lambda\cup \Lambda'$ we apply Proposition \ref{sublinearseminorm} (note we can have at most $2^s(\ell+1)$ different Hardy field functions in the iterates) to deduce that there exist positive integers $t',s'$, a finite set $\widetilde{Y}$ and polynomials $p'_{\ue,j}$, where $\ue\in [[s']]$ and $1\leq j\leq s_0$ (which depend only on the original functions $a_1,...,a_k$), such that \begin{multline*}
      \limsup\limits_{R\to+\infty}\ \sup_{|c_{r,\bm}| \leq 1}\ 
          \bignorm {  \underset{1\leq r\leq R}{\E} c_{r,\bm} \ S^{\floor{\widehat{A}_{\underline{1},r}(\bm)+\psi_0(r)}         }
          (\prod_{\underline{\e}\in A_1} \mathcal{C}^{|\ue|} S^{\floor{-c_{\ue}(\bm)} +h_{\ue,\bm}}H_0)\\
          \prod_{\underline{\e}\in A_2} S^{\floor{\widehat{A}_{\ue,r}(\bm)+\psi_0(r)}}(\mathcal{C}^{|\ue|} H_0)\prod_{1\leq i\leq \ell} \prod_{\ue\in[[s]]}
         S^{\floor{\widehat{A}_{\ue,r}(\bm)+\psi_{i}(r)} }(\mathcal{C}^{|\underline{\e}|}H_{i,R})  }_{L^2(\m\times \m)}^{2^{t'}}   \ll_{a_1,...,a_k} \\
       \frac{1}{M}+  \sum_{{\bf h}\in \widetilde{Y}^{[[s']]}}^{}\underset{\bm'\in [-M,M]^{t'}}{\E} \nnorm{\prod_{\ue'\in [[s']]} S^{\floor{A_{\ue'}(\bm',\bm)} +h_{\ue'}}F_{\bm,{\bf h}_{\bm}}}_{2^{s+1}(\ell+1),S}.
  \end{multline*}Here, we have defined \begin{equation*}
      A_{\ue'}(\bm',\bm)=\sum_{j=1}^{s_0}p'_{\ue,j}(\bm')u_j(\bm).
  \end{equation*}
  Therefore, since the set $\Lambda\cup \Lambda'$ has density zero, we use the H\"{o}lder inequality to get that the quantity in \eqref{final} is $\ll_{a_1,...,a_k}$ 
  \begin{equation*}
   \underset{{\bf m}\in [-M,M]^t}{\E}\big(\sum_{{\bf h}\in \tilde{Y}^{[[s']]}}^{}\underset{\bm'\in [-M,M]^{t'}}{\E} \nnorm{\prod_{\ue'\in [[s']]} S^{\floor{A_{\ue'}(\bm',\bm)} +h_{\ue'}}F_{\bm,{\bf h}_{\bm}}}_{2^{s+1}(\ell+1),S}^{1/2^{t'}} \big)^{1/2}   +o_M(1).
  \end{equation*}

  Now, we take the limit as $M\to+\infty$ and use the power mean inequality to bound the $\limsup$ of the above quantity by $O_{a_1,...,a_k}(1)$ times a power of \begin{equation*}
    \limsup\limits_{M\to+\infty}  \sum_{{\bf h}\in \tilde{Y}^{[[s']]}}^{}\underset{\bm'\in [-M,M]^{t'}}{\E}\ \underset{\bm\in [-M,M]^{t}}{\E} \nnorm{\prod_{\ue'\in [[s']]} S^{\floor{A_{\ue'}(\bm',\bm)} +h_{\ue'}}F_{\bm,{\bf h}_{\bm}}}_{2^{s+1}(\ell+1),S}.
  \end{equation*}Our result will follow if we show that for any integers $h_{\ue'}$ we have \begin{equation*}
      \limsup\limits_{M\to+\infty}  \underset{\bm'\in [-M,M]^{t'}}{\E}\ \underset{\bm\in [-M,M]^{t}}{\E} \nnorm{\prod_{\ue'\in [[s']]} S^{\floor{A_{\ue'}(\bm',\bm)} +h_{\ue'}}F_{\bm,{\bf h}_{\bm}}}_{2^{s+1}(\ell+1),S}=0.
  \end{equation*}We substitute $F_{\bm,{\bf h}_{\bm}}$ to rewrite this limit as \begin{multline}\label{ult}
        \limsup\limits_{M\to+\infty}  \underset{\bm'\in [-M,M]^{t'}}{\E}\ \underset{\bm\in [-M,M]^{t}}{\E} \nnorm{\prod_{\ue'\in [[s']]} S^{\floor{A_{\ue'}(\bm',\bm)} +h_{\ue'}}     \big( \prod_{\underline{\e}\in A_1} \mathcal{C}^{|\ue|} T^{\floor{-c_{\ue}(\bm)} +h_{\ue,\bm}}H_0\big)}_{2^{s+1}(\ell+1),S}=\\ 
          \limsup\limits_{M\to+\infty}  \underset{\bm'\in [-M,M]^{t'}}{\E}\ \underset{\bm\in [-M,M]^{t}}{\E} \nnorm{\prod_{\ue'\in [[s']]}\prod_{\underline{\e}\in A_1} S^{\floor{A_{\ue'}(\bm',\bm)} +h_{\ue'}+\floor{-c_{\ue}(\bm)} +h_{\ue,\bm}} \big( \mathcal{C}^{|\ue|} H_0\big)}_{2^{s+1}(\ell+1),S}.
  \end{multline}
    For a fixed $\bm$ outside all the negligible sets defined above, the polynomials $A_{\ue'}(\bm',\bm)$ are pairwise essentially distinct, as polynomials in $\bm'$. Therefore, they are also essentially distinct as polynomials in $(\bm',\bm)$. In addition, we have also established that the polynomials $c_{\ue}(\bm)$ are non-constant and essentially distinct. Therefore, it is easy to check that the polynomials
  $A_{\ue'}(\bm',\bm)-c_{\ue}(\bm)$ are pairwise essentially distinct.
  
  We combine the integer parts in the iterates in \eqref{ult} (correcting with some error terms with values in $\{0,\pm 1\}$). Expanding the seminorm in \eqref{ult}, we arrive at an iterated limit of polynomial averages. We also use Lemma \ref{errors} to remove the error terms in the iterates. Using\footnote{This lemma was proven for a specific F{\o}lner sequence (namely $[N]^k$), but the same argument extends to the general case. See also \cite{Leibmanseveral} for a more detailed proof in the case of integer polynomials.} \cite[Lemma 4.3]{Frajointhardy}, we deduce that the limit in \eqref{ult} is zero under the assumption that $\nnorm{H_0}_{q,T\times T}=0$ for some positive integer $q$. Since \begin{equation*}
      \nnorm{H_0}_{q,T\times T}=\nnorm{\overline{f_1}\otimes f_1}_{q,T\times T}\leq \nnorm{f_1}_{q+1, T}^2,
  \end{equation*}we deduce that the desired limit is zero if we assume that $\nnorm{f_1}_{q+1,T}=0$. The result follows.

  \begin{appendix}
  \section{Some properties of Hardy sequences}\label{Hardy}

\subsection{Growth rates of Hardy functions}
We assume that we are working with a Hardy field $\mathcal{H}$ that satisfies the properties mentioned in Section \ref{background}. Such a field contains the Hardy field $\mathcal{LE}$ of logarithmico-exponential functions and, for any two functions $f,g$ that belong to $\mathcal{H}$, we have that the limit \begin{equation*}
    \lim\limits_{t\to\infty} \frac{f(t)}{g(t)}
\end{equation*}exists. We also have the assumptions of closure under composition and compositional inversion that we made in Section \ref{background}. We will use these properties freely.

\begin{proposition}\label{prop:basic}
Let $f\in\mathcal{H}$ have polynomial growth. Then, for any natural number $k$, we have \begin{equation*}
    f^{(k)}(t) \ll \frac{f(t)}{t^k}.
\end{equation*}
In addition, if $ t^{\delta}\prec f(t)$ or $f(t)\prec t^{-\delta}$ for some $\delta>0$, we have \begin{equation*}
    f'(t) \sim \frac{f(t)}{t}.
\end{equation*}
\end{proposition}

\begin{proof}
  We will show that the limit \begin{equation*}
      \lim\limits_{t\to\infty} \frac{tf'(t)}{f(t)}
  \end{equation*}is finite. Using L'Hospital's rule, the above limit is equal to the limit \begin{equation}\label{log}
     \lim\limits_{t\to\infty} \frac{\log |f(t)|}{\log t}.
 \end{equation}Since $f$ has polynomial growth, the above limit is bounded. In particular, this implies that \begin{equation*}
     f'(t)\ll \frac{f(t)}{t}.
 \end{equation*} The first part now follows by repeated application of this relation.
 
 For the second part, we can easily see that the given condition implies that the limit in \eqref{log} is positive in the first case and negative in the second case. Therefore, the limit is non-zero and the claim follows.
 
 \end{proof}

The above proposition implies that, for any $f\in\mathcal{H}$ of polynomial growth, all derivatives of sufficiently large order of $f$ will converge monotonically to $0$. In addition, we get that for every $k$ sufficiently large, we must have \begin{equation*}
    f^{(k+1)}(t)\sim \frac{f^{(k)}(t)}{t}.
\end{equation*}Indeed, assume that $f(t)\prec t^s$, for some non-integer $s$. Then, we must have $f^{(k)}(t)\prec t^{s-k}$. Thus, if $k$ is large enough, then $f^{(k)}\prec t^{-\delta}$ for some $\delta>0$, which yields our claim.

\begin{proposition}\label{growth}
Let $f\in\mathcal{H}$ be strongly non-polynomial with $f(t)\succ \log t$. Then, for $k$ sufficiently large, we have \footnote{All the functions defined here belong to $\mathcal{H}$ due to the assumptions we have made on our Hardy field, namely, that it is closed under composition of certain functions.} \begin{equation*}
     1\prec |f^{(k)}(t)|^{-1/k}\prec  |f^{(k+1)}(t)|^{-1/(k+1)}\prec t.
\end{equation*}
\end{proposition}
\begin{remark*}The above proposition can be proven under the slightly more general condition that $|f(t)-p(t)|\succ \log t$ for all real polynomials $p$ (cf. \cite[Lemma 3.5]{Fraeq}), but we will not need this for the proofs of our main results. We give the proof here for completeness.
\end{remark*}

\begin{proof}
  The function $f$ has non-vanishing derivatives of all orders, since it is not a polynomial. Let $d$ be an integer, such that $t^{d}\prec f(t)\prec t^{d+1}$.
 Then, Proposition \ref{prop:basic} implies that $|f^{d+1}(t)|\to 0$. Therefore, for any $k\geq d+1$, we have $ f^{(k)}(t)\prec 1$. This, of course, gives the leftmost part of the required inequality. In particular, $(d+1)$ is minimal among the integers $k$, for which $f^{(k)}(t)$ converges to 0.

  To prove the rightmost inequality of the proposition, it is sufficient to prove that \begin{equation*}
      f^{(d+1)}(t)\succ t^{-d-1}.
  \end{equation*}For $k\geq d+1$, the result then follows by successive applications of L' Hospital's rule. In the case $d=0$, the above relation follows easily from L'Hospital's rule. Therefore, we may assume that $d\geq 1$. Now, since $f$ is strongly non-polynomial, we have that the function $f^{(d)}(t)$ goes to infinity. We will show that \begin{equation}\label{agrowth}
      a'(t)\gg \frac{a(t)}{t\log^2 t}
  \end{equation}where $a$ is any one of the functions $f,f',...,f^{(d)}$ (cf. \cite[Lemma 2.1]{Fraeq}). The result then follows by noting that \begin{equation*}
      f^{(d+1)(t)}\gg \frac{f(t)}{t^{d+1}(\log t)^{2d+2}}\gg \frac{1}{t(\log t)^{2d+2}}\succ \frac{1}{t^{d+1}}.
  \end{equation*}Equation \eqref{agrowth} follows by showing that the limit \begin{equation*}
      \lim\limits_{t\to \infty} \frac{a'(t)t(\log t)^2}{a(t)}
  \end{equation*}is infinite. If that is not the case, then we must have \begin{equation*}
      (\log|a(t)|)'\ll \frac{1}{t(\log t)^2}.
  \end{equation*}Integrating, we get \begin{equation*}
      \log |a(t)|\ll \frac{1}{\log t}+c
  \end{equation*}for some real number $c\in \R$. Thus, the function $\log|a(t)|$ is bounded. However, note that for any choice of the function $a$, we have $|a(t)|\to +\infty$, since the original function $f$ dominates the function $t^d$. This gives a contradiction.

  It remains to establish the middle part, namely that if $k\geq d+1$, then \begin{equation*}
      |f^{(k+1)}(t)|^{k}\prec  |f^{(k)}(t)|^{k+1}.
  \end{equation*}However, we have  $ |f^{(k+1)}(t)|^{k}\ll |f^{(k)}(t)|^k/t^k$ by Proposition \ref{prop:basic} and we easily get the conclusion by combining this relation with the relation $t^{-k}\prec f^{(k)}(t)$ that we established in the previous step.
\end{proof}

We give here a description of the polynomial approximations that we use in our arguments. Consider a strongly non-polynomial function $f\in\mathcal{H}$ that satisfies $f(t)\succ \log t$. Then, if $k$ is large enough, we can find a function $L(t)\in \mathcal{H}$ such that \begin{equation}\label{L(t)}
    |f^{(k)}(t)|^{-1/k}\prec L(t)\prec  |f^{(k+1)}(t)|^{-1/(k+1)}.
\end{equation}
Such a function always exists (one can take the geometric mean of the functions $|f^{(k)}(t)|^{-1/k}$ and $|f^{(k+1)}(t)|^{-1/(k+1)}$). We study the function $f$ in small intervals of the form $[N,N+L(N)]$. Observe that if $0\leq h\leq L(N)$, then we have \begin{equation*}
    f(N+h)=f(N)+\cdots+\frac{h^kf^{(k)}(N)}{k!} +\frac{h^{k+1}f^{(k+1)}(\xi_{h,N})}{(k+1)!}
\end{equation*}for some $\xi_{h,N} \in [N,N+h]$. We know that $|f^{(k+1)}(t)|\to 0$ monotonically for $t$ large enough. Then, we observe that (for $N$ sufficiently large) \begin{equation*}
    \Big|\frac{h^{k+1}f^{(k+1)}(\xi_{h,N})}{(k+1)!}\Big|\leq\Big| \frac{L(N)^{k+1}f^{(k+1)}(N)}{(k+1)!}\Big|\prec 1,
\end{equation*}because $L(t)\Big|f^{(k+1)}(t)\Big|^{\frac{1}{k+1}}\prec 1$, by our initial choice of $L(t)$. Using an entirely similar argument, we can prove that \begin{equation*}
    \Big|\frac{L(N)^{k}a^{(k)}(N+L(N))}{k!}\Big|\to +\infty.
\end{equation*}Indeed, since $L$ is a sublinear function, we can easily check that the functions $a^{(k)}(t+L(t))$ and $a^{(k)}(t)$ have the same growth rate and thus we only need to prove that \begin{equation}\label{epr}
   \Big|\frac{L(N)^{k}a^{(k)}(N)}{k!}\Big|\to +\infty
\end{equation}and this follows similarly as above.

In conclusion, functions that satisfy \eqref{L(t)} have the following characteristic property: the sequence $f(n)$, when restricted to intervals of the form $[N,N+L(N)]$, is asymptotically equal to a polynomial sequence (that depends on $N$) of degree exactly $k$. This motivates us to study the properties of functions $L(t)$ that satisfy \eqref{L(t)}.

\subsection{The sub-classes $S(f,k)$}\label{aproximationlemmas}

In the proofs of the main theorems, we need to do the above approximation for several Hardy field functions in tandem. In order to achieve this, we will use the results of this subsection.

 Let $f\in\mathcal{H}$ be a strongly non-polynomial Hardy function such that $f(t)\gg t^{\delta}$, for some $\delta>0$. For example, we exclude functions that grow like $(\log t)^c$, where $c>1$. 
 For such a function $f$ and $k\in \N$ sufficiently large (it is only required that $f^{(k)}(t)\to 0$), we define the subclass $S(f,k)$ of $\mathcal{H}$ as
\begin{equation*}
    S(f,k)=\{g\in  \mathcal{H}\ \colon    |f^{(k)}(t)|^{-\frac{1}{k}}\preceq g(t)\prec |f^{(k+1)}(t)|^{-\frac{1}{k+1}} \},
\end{equation*}
where the notation $g(t)\preceq f(t)$ means that the limit $\lim\limits_{t\to\infty} |f(t)/g(t)|$ is non-zero. Note that every $g\in S(f,k)$ is a sub-linear function, that is $g(t)\prec t$. Some very basic properties of the classes $S(f,k)$ are established in the following lemma.

\begin{lemma}\label{basic}Let $f\in\mathcal{H}$ be a strongly non-polynomial function with $f(t)\gg t^{\delta}$, for some $\delta>0$.\\
i) The class $S(f,k)$  is non-empty, for $k$ sufficiently large. \\
ii) For any $0< c< 1$ sufficiently close to 1, there exists $k_0\in \N$, such that the function $t\to t^c$ of $\mathcal{H}$ belongs to $S(f,k_0)$.  \\
iii) The class  $S(f,k)$ does not contain all functions of the form $t\to t^c$, for $c$ sufficiently close to 1.
\end{lemma}

\begin{proof}
i) This follows immediately from Proposition \ref{growth}. We can actually show something stronger, namely, that if $f(t)\gg t^{\delta}$ for some $0<\delta<1$, then \begin{equation}\label{20000} 
    \frac{|f^{(k+1)}(t)|^{-\frac{1}{k+1}}}{|f^{(k)}(t)|^{-\frac{1}{k}}}\gg t^{\frac{\delta}{k(k+1)}},
\end{equation}which means that the functions at the "endpoints" of $S(f,k)$ differ by a fractional power.
This last inequality follows by combining the relations \begin{equation*}
    f^{(k)}(t)\gg tf^{(k+1)}(t)\ \ \ \  \text{  and  }\ \ \ \  f^{(k)}(t)\gg t^{\delta -k}.
\end{equation*}
\\
ii) It is sufficient to show that for large $k\in \N$, we have $t^c\ll |f^{(k)}(t)|^{-\frac{1}{k}} $. Fix a non-integer $q$, such that $f(t)\ll t^q$. Then, for any $k\in\N$, we have $f^{(k)}(t)\ll t^{q-k}$. It suffices to show that for large enough $k$ we have \begin{equation*}
    t^{q-k}\ll t^{-ck}\implies t^{k-q}\gg t^{ck}.
\end{equation*}This is obvious, since $c<1$.\\
iii) Similar to ii).
\end{proof}\normalfont 
In essence, the claim implies that the classes $S(f,k)$ form a "partition" of the subclass \begin{equation*}
  A= \{  g(t)\gg t^c\colon \  \exists \  \delta>0,\  \text{ with }  g(t)\ll t^{1-\delta }  \}
\end{equation*} for some $c>0$. That means that any  sub-linear function that grows approximately as a (sufficiently large) fractional power must be contained in the union of the $S(f,k)$. This union however does not contain functions that are "logarithmically close" to linear functions, such as $t(\log t)^{-1}$. Although inaccurate, it is instructive to imagine the classes $S(f,k)$ as (disjoint) intervals on the real line. For example, if $S(f,k)=\{g(t)\colon \sqrt{t}\preceq g(t)\prec t^{2/3}\}$, then we can think that $S(f,k)$ is represented by the interval $[\frac{1}{2},\frac{2}{3})$.

The following proposition relates the behavior of the subclasses $S(f,k)$ and $S(g,\ell)$ for different functions $f,g\in\mathcal{H}$.

\begin{proposition}\label{two classes}
For any two functions $f,g\in\mathcal{H}$ as in Lemma \ref{basic} that also satisfy $g(t)\ll f(t)$, we have the following:

i) The relation $S(f,k)=S(g,k)$ holds for some $k\in \N$ if and only if $f\sim g$.

ii) If $S(f,k)\cap S(g,\ell)\neq \emptyset$, then $ k\geq \ell$. In addition, if the function $(f^{(k)}(t))^{-\frac{1}{k}}$ is contained in $S(g,\ell)$ and $f\not\sim g$, then $k\geq \ell +1$.

iii) There exist infinitely many pairs of integers $(k,\ell)$, such that $S(f,k)\cap S(g,\ell)\neq \emptyset$.

\end{proposition} 

\begin{proof}

i) It is a straightforward application of L' Hospital's rule.\\
ii) Since the given intersection is non-empty, we must necessarily have $|g^{(\ell)}(t)|^{-\frac{1}{\ell}}\preceq |f^{(k+1)}(t)|^{-\frac{1}{k+1}}$. Suppose that $k<\ell$, so that we have the inequalities $|g^{(\ell)}(t)|^{-\frac{1}{\ell}}\preceq |f^{(k+1)}(t)|^{-\frac{1}{k+1}}\preceq |f^{(l)}(t)|^{-\frac{1}{\ell}}$, which implies that $f^{(l)}(t)\preceq g^{(l)}(t)$. Because we also have $g(t)\ll f(t)$, we can easily deduce that $f(t)\sim g(t)$ using the fact that both of these functions are strongly non-polynomial. Thus, the intersection $S(f,k)\cap S(g,\ell)$ is non-empty if and only if $k=\ell$, which is a contradiction.

For the proof of the second part, we use immediately the fact that $k\geq \ell$, which follows by the first part. Suppose that $k=l$ and we shall arrive at a contradiction. If $(f^{(k)}(t))^{-\frac{1}{k}}\in S(g,k)$, then, we must have $(f^{(k)}(t))^{-\frac{1}{k}}\succeq (g^{(k)}(t))^{-\frac{1}{k}}$, which implies that $g^{(k)}(t)\succeq f^{(k)}(t)$. This contradicts the assumption that $f(t)\succ g(t)$ (apply L' Hospital's rule $k$ times).

iii) For any $c$ close to 1, we can find $k$, such that the function $t^c$ belongs to $S(f,k)$ (this follows from the second statement of Lemma \ref{basic}) and similarly for the Hardy function $g$. Then, the intersection $S(f,k)\cap S(g,\ell)$ is non-empty. Taking $c\to 1^{-}$ and using the third statement of Lemma \ref{basic}, we can find infinitely many such pairs.

\end{proof}
\begin{remark*}

It is straightforward to generalize the third statement of the above proposition to the case of $k$ distinct functions $f_1,...,f_k$ in $\mathcal{H}$. We will use this observation in our arguments to find a function $L$ in the intersection of these classes. Note that our previous discussion implies that for such a function $L$, all the involved functions $f_1,...,f_k$ will have a polynomial expansion on intervals of the form $[N,N+L(N)]$ and this will play a crucial role in our approximations.
\end{remark*}

\subsection{The subclasses $S_{sml}(f,k)$}
We can similarly define analogs of the classes $S(f,k)$ for functions with small growth rate, that is sub-fractional functions. Let $f\in\mathcal{H}$ be a sub-fractional function such that $\log t\prec f(t)$.
If $k\geq 1$, we can define the class \begin{equation*}
    S_{sml}(f,k)=\{g\in \mathcal{H}\colon   \  |f^{(k)}(t)|^{-\frac{1}{k}}\preceq g(t)\prec |f^{(k+1)}(t)|^{-\frac{1}{k+1}} \}.
\end{equation*}
The properties of Proposition \ref{two classes} proven for the classes $S(f,k)$ are carried verbatim to this new setting. The major difference is that now every function $g\in S_{sml}(f,k)$ dominates all functions of the form $t^{1-\delta}$ for $\delta>0$ (an example is the function $t/\log t$). In particular, $S_{sml}(f,k)$ has trivial intersection with the classes $S(h,\ell)$ defined above for any integers $k,\ell$ and appropriate functions $f,h$.

As an example, let us consider a fractional power $t^{\delta}$ with $0<\delta<1$ and two functions $f,g\in\mathcal{H}$ such that $f(t)\gg t^{\e}$ for some $\e>0$, while $\log t\prec g(t)$ and $g$ is sub-fractional. A typical case is the pair $(t^{3/2},\log^2 t)$. We know that if $\delta$ is close enough to 1, then the function $t^{\delta}$ will belong to $S(f,k)$ for some $k\in\N$. Using approximations similar to the ones in the previous subsection, we can see that the sequence $f(n)$ becomes a polynomial sequence of degree $k$ on intervals of the form $[N,N+N^{\delta}]$. On the other hand, the sequence $g(n)$, restricted to the same interval, is $o_N(1)$ close to the value $g(N)$, which means that it is "essentially" constant on this interval. This difference in behavior leads to some added complexity in our proofs, since some of our functions may be approximated by polynomials, while other functions become constant.

On the other hand, a function $f\in\mathcal{H}$ with $f(t)\ll \log t$, when restricted to intervals of the form $[N,N+L(N)]$, is $o_N(1)$-close to the value $f(N)$ for any sub-linear function $L(t)$. Functions of this form always collapse to a constant when restricted to intervals of the above form.

  \end{appendix}

\end{document}